\documentclass[a4paper]{amsart}

\usepackage{amscd,amsthm,amsfonts,amssymb,amsmath}
\usepackage[colorlinks=true,citecolor=blue]{hyperref}
\usepackage{fullpage}
\usepackage{tikz}
\usepackage[all]{xy}
\usepackage{mathdots}
\usepackage{multirow}

\newcommand{\lddot}{\mathop{..}}

\let\Im\undefined

\DeclareMathOperator{\Im}{Im}

\DeclareMathOperator{\rank}{rank}

\newcommand{\glob}{{\mathrm{glob}}}

\newcommand{\sstr}{{\text{\rm-str}}}

\newcommand{\nat}{{\mathrm{nat}}}

\renewcommand{\pmod}[1]{\,\left(\mathrm{mod}\,#1\right)}

    \newcommand{\BE}{{\mathbb {E}}} \newcommand{\BF}{{\mathbb {F}}}

     \newcommand{\BP}{{\mathbb {P}}}
    \newcommand{\BQ}{{\mathbb {Q}}} \newcommand{\BR}{{\mathbb {R}}}

     \newcommand{\BZ}{{\mathbb {Z}}}

    \newcommand{\CE}{{\mathcal {E}}}

     \newcommand{\RT}{{\mathrm {T}}}

     \newcommand{\fX}{{\mathfrak{X}}}

    \newcommand{\Aut}{{\mathrm{Aut}}}
     
    \newcommand{\Alt}{{\mathrm{Alt}}}
    
    \DeclareMathOperator{\Avg}{Avg}

    \DeclareMathOperator{\corank}{corank}

    \DeclareMathOperator{\diag}{diag}

    \newcommand{\Gal}{{\mathrm{Gal}}} \newcommand{\GL}{{\mathrm{GL}}}
    
    \DeclareMathOperator{\Hom}{Hom}

    \newcommand{\loc}{{\mathrm{loc}}}\newcommand{\Mat}{{\mathrm{Mat}}}

    \newcommand{\new}{{\mathrm{new}}}

    \DeclareMathOperator{\ord}{ord}
    \DeclareMathOperator{\Prob}{Prob}

    \newcommand{\RRM}{{\mathrm{RM}}}
    \newcommand{\Sel}{{\mathrm{Sel}}}

    \newcommand{\Sp}{{\mathrm{Sp}}}

    \newcommand{\tr}{{\mathrm{tr}}}

    \newcommand{\unr}{{\mathrm{unr}}}

    \theoremstyle{plain}
    
    \newtheorem{thm}{Theorem}
    \newtheorem{cor}[thm]{Corollary}
    \newtheorem{lem}[thm]{Lemma}
    \newtheorem{prop}[thm]{Proposition}
    
    \newtheorem{conj}[thm]{Conjecture}
    \theoremstyle{definition}
    \newtheorem{defn}[thm]{Definition}

    \theoremstyle{remark}
    \newtheorem{remark}[thm]{Remark}
    \newtheorem{example}[thm]{Example}
    \numberwithin{thm}{section}

    \numberwithin{equation}{section}
    \newcommand{\lrd}[2]{\left(\frac{#1}{#2}\right)}
     \newcommand{\lrdd}[2]{\left[\frac{#1}{#2}\right]}

\begin{document}

\title{On the Distribution of $2$-Selmer ranks\\
of Quadratic Twists of Elliptic Curves over $\BQ$}

\author{Jinzhao Pan}
\address{School of Mathematical Sciences,
Tongji University, Shanghai, China}
\email{jzpan@tongji.edu.cn}

\author{Ye Tian}
\address{Morningside Center of Mathematics,
Chinese Academy of Sciences, Beijing, China}
\email{ytian@math.ac.cn}

\date{}

\begin{abstract} We characterize the distribution
of $2$-Selmer ranks
of quadratic twists of elliptic curves over $\BQ$
with full rational $2$-torsion.
We propose a new type of random alternating matrix model
$M_{*,\mathbf t}^\Alt(\BF_2)$
over $\BF_2$
with $0$, $1$ or $2$ ``holes'',
with associated Markov chains, described by parameter
$\mathbf t=(t_1,\cdots,t_s)\in\BZ^s$ where $s$ is the number of ``holes''.
We proved that
for each equivalence classes
of quadratic twists of elliptic curves:
\begin{itemize}\item[(1)] The distribution
of $2$-Selmer ranks agrees with the distribution of coranks of matrices in
$M_{*,\mathbf t}^\Alt(\BF_2)$;
\item[(2)]
The moments of $2$-Selmer groups agree with that of
$M_{*,\mathbf t}^\Alt(\BF_2)$, in particular, the average order of essential
$2$-Selmer groups is $3+\sum_i2^{t_i}$.
\end{itemize}
Our work extends the works of Heath-Brown, Swinnerton-Dyer,  Kane, and Klagsbrun-Mazur-Rubin
where the matrix only has $0$ ``holes'',
the matrix model is the usual random alternating matrix model,
and the average order of essential $2$-Selmer groups is $3$.
A new phenomenon is that different equivalence classes in the same quadratic twist family
could have different parameters,
hence have different distribution of $2$-Selmer ranks.
The irreducible property of the Markov chain associated to
$M_{*,\mathbf t}^\Alt(\BF_2)$ gives the positive density results on
the distribution of $2$-Selmer ranks.
\end{abstract}
\maketitle
\tableofcontents

\section{Introduction}

Let $\CE$ be a quadratic twist family of elliptic curves over $\BQ$.
There are two fundamental conjectures on it:
\begin{itemize}
\item
Goldfeld (\cite{Gol79}, Conjecture B)
conjectured that the density of $E\in\CE$ such that
$\ord_{s=1}L(E,s)=r$ is $1/2$ if $r=0, 1$.
\item
Kolyvagin (\cite{Kol90}, Conjecture F)
has a heuristic that for any prime $p$ and $r=0,1$,
there is an elliptic curve $E$ in $\CE$ such that
$\ord_{s=1}L(E,s)=r$ and
the analytic Sha of $E$ is a $p$-adic unit.
\end{itemize}

To study the $L$-values and Mordell-Weil groups for elliptic curves,
one may start with their Selmer groups.
The first significant result on the distribution of Selmer groups
in quadratic twist families of elliptic curves is by
Swinnerton-Dyer \cite{SD08} and Kane \cite{Kane13}.
They studied the distribution of $2$-Selmer groups for
families $\CE$
with full rational $2$-torsion points and satisfying certain additional conditions.
For simplicity let's call them of type (A):
\begin{itemize}
\item[(A)]
families $\CE$
with full rational $2$-torsion points
such that $E\in \CE$ does not have a rational cyclic $4$-isogeny, for example, the congruent number elliptic curves $ny^2=x^3-x$.
\end{itemize}
Their works are inspired by the pioneering work of
Heath-Brown \cite{HB93}, \cite{HB94}
on the distribution of $2$-Selmer groups
for the family $ny^2=x^3-x$.

Meanwhile, Poonen and Rains \cite{PR12},
Bhargava, Kane, Lenstra, Poonen and Rains \cite{BKLPR}
proposed that
the distribution of $2$-Selmer groups for
all elliptic curves over $\BQ$ ordered by height
should follow the model of random alternating matrices over $\BF_2$.
More precisely, let $M_{2k+r}^\Alt(\BF_2)$ be the set of alternating matrices over $\BF_2$
of size $(2k+r)\times(2k+r)$, endowed with counting measure,
then the model defines the probability
$$
P_r^\Alt(d):=
\lim_{k\to\infty}\BP\big(\corank(B)=d\mid B\in M_{2k+r}^\Alt(\BF_2)\big)
$$
for $d\geq 0$ and $r\in\BZ$.
The work of Heath-Brown, Swinnerton-Dyer and Kane
indicates that the distribution of $2$-Selmer groups
in type (A) quadratic twist families coincides with the general $2$-Selmer model.

A crucial feature of the matrix model $M_{2k+r}^\Alt(\BF_2)$
is that there is a Markov chain attached to it,
which was a key ingredient in \cite{SD08}.
More precisely,  let $X_k$ be the random variable $\corank(B)$ as
$B\in M_{2k+r}^\Alt(\BF_2)$,
then the sequence $(X_k)_{k\geq 0}$ form a Markov chain
with transition probability
$$
\BP\big(X_{k+1}=m_\new\mid X_k=m\big)=\begin{cases}
2^{-1-2m},&\text{if }m_\new=m+2, \\
2^{-m}(3-5\cdot 2^{-1-m}),&\text{if }m_\new=m, \\
(1-2^{-m})(1-2^{1-m}),&\text{if }m_\new=m-2, \\
0,&\text{otherwise}.
\end{cases}
$$
Swinnerton-Dyer  \cite{SD08} analyzed the behavior of $2$-Selmer groups of $E^{(n)}$
for type (A) curves
when the number $k$ of prime factors of $n$ increases.
It is described by an ``almost Markov'' process whose transition probability
differs from the above one by a controllable error term as $k\to\infty$.
This ultimately deduces the distribution of $2$-Selmer for type (A)
families.

After the work on type (A) families,
Klagsbrun, Mazur and Rubin \cite{KMR13}, \cite{KMR14} found that
the same matrix model and Markov chain also applies to quadratic twists of elliptic curves
$E$ over $\BQ$ such that $\Gal(\BQ(E[2])/\BQ)\cong S_3$.
To summarize, we have the following result due to Heath-Brown, Swinnerton-Dyer, Kane,
and Klagsbrun-Mazur-Rubin.
\begin{thm}
[Heath-Brown, Swinnerton-Dyer, Kane,
and Klagsbrun-Mazur-Rubin]
\label{main type A}
Let $\CE$ be a quadratic twist family, which of type (A),
or the curves $E$ in $\CE$ satisfies $\Gal(\BQ(E[2])/\BQ)\cong S_3$.
Then for any $d\geq 0$,
$$
\Prob\big(
\dim_{\BF_2}S(E)=d\mid E\in\CE
\big)=\frac{1}{2}\sum_{r=0,1}P_r^\Alt(d),
$$
where $S(E):=\Sel_2(E/\BQ)/E(\BQ)[2]$ is the essential $2$-Selmer group of $E$.
In particular, the average order of $S(E)$
as $E$ runs over $\CE$ is $3$.
\end{thm}

\subsection{Matrix model, Markov chain, and distribution of $2$-Selmer rank}

For the distribution of $2$-Selmer rank for families $\CE$
with full rational $2$-torsion points
which are not of type (A), no one has even made a conjecture
on the complete description of them.
It turns out that the question is quite subtle:
the distribution has nicer behavior when restricted to equivalence classes,
and can vary drastically among different equivalence classes
even in the same quadratic twist family.
The main goal of this paper is
to study the distribution of $2$-Selmer groups for them (see Theorems \ref{main theorem}
and \ref{main average}).

First we define $\Sigma$-equivalence classes.
Let
$$
\Sigma_0=\Sigma_0(\CE):=\{\text{primes dividing conductors of all curves in }\CE\} \cup \{2, \infty\}.
$$
Note that there exists an elliptic curve $E\in\CE$ with conductor supported in $\Sigma_0(\CE)$.
Let $\Sigma$ be any finite set of places of $\BQ$ containing $\Sigma_0$.
Two elliptic curves in $\CE$ are called $\Sigma$-equivalent
if they are isomorphic over $\BQ_v$ for all $v\in \Sigma$.
The root number $\epsilon(E)\in \{\pm 1\}$ depends only on the
$\Sigma$-equivalence class $\fX$ of $E$, denoted by $\epsilon(\fX)$.
Let $r(\fX):=(1-\epsilon(\fX))/2\in \{0, 1\}$.

Next we define the relevant model of random alternating matrices,
extending the general $2$-Selmer model $M_{2k+r}^\Alt(\BF_2)$.
A matrix of size $(2k+r)\times(2k+r)$ can be considered as
``a block matrix with $2\times 2$ big blocks of size roughly $k\times k$''.
The general model $M_{2k+r}^\Alt(\BF_2)$, which we will call type (A) model,
can be considered as ``all $2\times 2$ big blocks are freely chosen, only subject to
alternating condition''.
We propose type (B) and type (C) model which will add condition that one and two
diagonal blocks are zero (``holes''), respectively.
More precisely, we define
\begin{align*}
\text{type (A) model: }&
M_{2k+r,\varnothing}^\Alt(\BF_2)
:=M_{2k+r}^\Alt(\BF_2), \\
\text{type (B) model: }&
M_{2k+r,(t_1)}^\Alt(\BF_2)
:=\left\{B\in M_{2k+r}^\Alt(\BF_2)
\text{ of form }\left(\begin{smallmatrix}
0 & B_{12} \\ B_{21} & B_{22}
\end{smallmatrix}\right)\right\}, \\
\text{type (C) model: }&
M_{2k+r,(t_1,t_2)}^\Alt(\BF_2)
:=\left\{B\in M_{2k+r}^\Alt(\BF_2)
\text{ of form }\left(\begin{smallmatrix}
0 & B_{12} & B_{13} \\ B_{21} & 0 & B_{23} \\ B_{31} & B_{32} & B_{33}
\end{smallmatrix}\right)\right\},
\end{align*}
with $0$-block of size $k+\frac{t_1+r}2$ for type (B);
$0$-blocks of sizes $k+\frac{t_1+r}2$ and $k+\frac{t_2+r}2$ for type (C).
We say $s=0$, $s=1$, $s=2$ for model of type (A), (B), (C), respectively,
and call $r$, $s$, and $\mathbf t:=(t_1,\cdots,t_s)$ the parameters of the model.
Clearly the parameters are subject to constraints
$t_i\equiv r\pmod 2$, moreover $t_1+t_2\leq 0$ for type (C).

Similar to \cite{SD08},
the Markov chain associated to our model $M_{2k+r,\mathbf t}^\Alt(\BF_2)$
plays a crucial role.
The new feature for types (B) and (C) is that, if let
$X_k$ be the random variable $\corank(B)$ as
$B\in M_{2k+r,\mathbf t}^\Alt(\BF_2)$,
then the sequence $(X_k)_{k\geq 0}$ does not form a Markov chain anymore.
Instead, we need to consider the refinement of $X_k$:
\begin{itemize}
\item
for type (B), there is $B_1':=\left(\begin{smallmatrix}
0 \\ B_{21}
\end{smallmatrix}\right)$ associated to $B=\left(\begin{smallmatrix}
0 & B_{12} \\ B_{21} & B_{22}
\end{smallmatrix}\right)\in M_{2k+r,(t_1)}^\Alt(\BF_2)$;
\item
for type (C), there are $B_1':=\left(\begin{smallmatrix}
0 \\ B_{21} \\ B_{31}
\end{smallmatrix}\right)$ and $B_2':=\left(\begin{smallmatrix}
B_{12} \\ 0 \\ B_{32}
\end{smallmatrix}\right)$ associated to $B=\left(\begin{smallmatrix}
0 & B_{12} & B_{13} \\ B_{21} & 0 & B_{23} \\ B_{31} & B_{32} & B_{33}
\end{smallmatrix}\right)\in M_{2k+r,(t_1,t_2)}^\Alt(\BF_2)$.
\end{itemize}
Let $Y_k$ be the random variable $(X_k,\corank(B_1'),\cdots,\corank(B_s'))$
which is a refinement of $X_k$,
then $(Y_k)_{k\geq 0}$ forms a Markov chain
(see Appendix \ref{s:model}),
moreover, such a Markov chain is irreducible
(see Appendix \ref{s:basic-property-markov}),
which means that any two points in state space are connected by a chain
of positive transition probabilities.

Once we have the model $M_{2k+r,\mathbf t}^\Alt(\BF_2)$ we can define the probability
$$
P_{r,\mathbf t}^\Alt(d):=\lim_{k\to\infty}\BP\big(
\corank(B)=d\mid B\in M_{2k+r,\mathbf t}^\Alt(\BF_2)\big).
$$

Our main result is as follows.

\begin{thm}
\label{main theorem}
Let $\CE$ be a quadratic twist family of elliptic curves with full rational $2$-torsion, and
$\fX\subset \CE$ a $\Sigma$-equivalence class.
Let $r=r(\fX)$.
Then there exists $s\in\{0,1,2\}$ depending only on $\CE$,
a parameter $\mathbf t=\mathbf t_\fX=(t_1,\cdots,t_s)$ depending only on $\fX$,
such that for any $d\geq 0$, we have
$$
\Prob\big(
\dim_{\BF_2}S(E)=d\mid E\in\fX\big)=P_{r,\mathbf t}^\Alt(d).
$$
More precisely,
there exist absolute constants $C_1>0$ and $c>0$, such that for any sufficiently large $N$
(depending on $\fX$),
\begin{equation}
\label{e:main-theorem}
\sum_{d\geq 0}\left|\BP\big(
\dim_{\BF_2}S(E)=d\mid E\in\fX,N(E)<N
\big)-P_{r,\mathbf t}^\Alt(d)\right|<C_1\cdot\exp(-c(\log\log\log N)^{1/2}),
\end{equation}
where $N(E)$ is the conductor of $E$.

Moreover, the parameter $\mathbf t=\mathbf t_{\fX}$
is invariant under the refinement of equivalence classes:
if $\Sigma\subset \Sigma'$ and $\fX'\subset \fX$ is a $\Sigma'$-equivalence class,
then the class $\fX'$ has the same parameter $\mathbf t$ as that of $\fX$.
\end{thm}

We will call $\CE$ is of type (A), (B), (C)
if the above $s$ is $0,1,2$, respectively.
This type (A) coincides with the type that Swinnerton-Dyer and Kane studied
in Theorem \ref{main type A}.
There are also characterizations of type (B) and (C):
\begin{itemize}
\item[(B)]
$E\in \CE$  has a rational cyclic $4$-isogeny,
and  $E[4]\not\subset E(\BQ(\sqrt{-1}))$ for all $E\in \CE$, for example,
twists $ny^2=x(x-1)(x+3)$ of $X_0(24)$.
\item[(C)]
$E\in \CE$  has a rational cyclic $4$-isogeny,
and $E[4]\subset E(\BQ(\sqrt{-1}))$ for some $E\in \CE$, for example,
twists $ny^2=x(x-9)(x-25)$ of $X_0(15)$.
\end{itemize}

\begin{remark}
Types (B) and (C) are not isogeny invariant, for example, the curve
$X_0(15):y^2=x(x-9)(x-25)$ which is of type (C) has a $2$-isogenous curve
$y^2=x(x-1)(x+15)$ which is of type (B).
\end{remark}

\subsection{Irreducible property and $2$-Selmer rank}

The Markov chain associated to the matrix model $M_{2k+r,\mathbf t}^\Alt(\BF_2)$
satisfies the irreducible property, that is,
any two points in state space are connected by a chain
of positive transition probabilities.
This implies that for $d\geq 0$,
$P_{r,\mathbf t}^\Alt(d)>0$ if and only if $d\equiv r\pmod 2$
and $d\geq\max_i(t_i)$.
In this case, Theorem \ref{main theorem} tells us that the
$E\in\fX$ such that $\dim_{\BF_2}S(E)=d$ is of positive density.
Conversely,
if $P_{r,\mathbf t}^\Alt(d)=0$,
then Proposition \ref{selmer lower bound} implies that
there are only finitely many $E\in\fX$
(in fact, no $E\in\fX$ with bad prime outside $\Sigma_0$) such that
$\dim_{\BF_2}S(E)=d$.
As an application, we have the following result.

\begin{prop}\label{main irred}
Let $\CE$ be a family with full rational $2$-torsion, and
$\fX\subset \CE$ a $\Sigma$-equivalence class.
Let $E_0\in\fX$ be a curve which has at least one bad prime outside $\Sigma_0$.
Let $d$ be any integer with $d\equiv \dim_{\BF_2}S(E_0)\pmod 2$ and $d\geq \dim_{\BF_2}S(E_0)$.
Then
$$
\Prob\big(\dim_{\BF_2}S(E)=d\mid E\in\fX\big)>0.
$$
\end{prop}

\begin{proof}
We only need to prove for families of type (B) and (C).
Let $\mathbf t=(t_i)$ be the parameter of $\fX$.
Then by Proposition \ref{selmer lower bound},
$\dim_{\BF_2}S(E_0)\geq\max_i(t_i)$.
Hence $\Prob\big(\dim_{\BF_2}S(E)=d\mid E\in\fX\big)=P_{r,\mathbf t}^\Alt(d)>0$
by $d\geq \dim_{\BF_2}S(E_0)\geq\max_i(t_i)$.
\end{proof}

The same question can be asked for the whole quadratic twist family $\CE$.
Inspired by the $p=2$ case of a conjecture by Kolyvagin
(\cite{Kol90}, Conjecture F),
we have the following conjecture.

\begin{conj}\label{conj1}
Let $\CE$ be a family of type (B) or (C).
Then for $r=0$ and $1$,
there exists a $\Sigma$-equivalence class $\fX$
with $r(\fX)=r$ and
whose parameter $\mathbf t_\fX=(t_i)$ satisfies $\max_i(t_i)\leq r$.

As a consequence, let $\CE$ be a quadratic twist family of elliptic curves over $\BQ$ with full rational $2$-torsion points. Then for any integer $d\geq 0$,
$$\Prob\left (\dim_{\BF_2}S(E)=d\ \big|\ E\in \CE\right)>0.$$
In particular, there is a positive density of $E\in\CE$
whose $2^\infty$-part of Shafarevich-Tate group is trivial.
\end{conj}

For families of type (A) this is known by Heath-Brown, Swinnerton-Dyer and Kane
(Theorem \ref{main type A}).

\subsection{Moments of $2$-Selmer groups}

We are also interested in the average order
and the ``essential'' average order of the essential
$2$-Selmer group $S(E)$ for $E\in\fX$.
Here we define the average order of $S(E)$ to be
$$
\Avg\big(\#S(E)\mid E\in\fX\big)
:=\lim_{N\to\infty}\BE\big(\#S(E)\mid E\in\fX,N(E)<N\big),
$$
and the ``essential'' average order of $S(E)$ to be
$$
\Avg_0\big(\#S(E)\mid E\in\fX\big)
:=\lim_{N\to\infty}\BE\big(\#S(E)\mid E\in\fX,N(E)<N,\#\Sigma_1(E)\sim\log\log N\big),
$$
where $\Sigma_1(E)$ is the set of bad primes of $E$ outside $\Sigma$,
and for positive real numbers $r$ and $N$, the notation $r\sim\log\log N$
indicates that $|r-\log\log N|\leq(\log\log N)^{2/3}$.

By Theorem \ref{main theorem},
$\Prob\big(\#S(E)=2^d\mid E\in\fX\big)=P_{r,\mathbf t}^\Alt(d)$,
where $\mathbf t$ is the parameter of $\fX$.
Hence it is reasonable to expect that the above average order
and ``essential'' average order are also equal to
the average order given by $P_{r,\mathbf t}^\Alt(d)$,
that is,
$\sum_{d=0}^\infty P_{r,\mathbf t}^\Alt(d)2^d$.

\begin{thm}\label{main average}
Let $\CE$ be a family with full rational $2$-torsion, and
$\fX\subset \CE$ a $\Sigma$-equivalence class with parameter $\mathbf t$.
Let $\xi\geq 1$ be an integer.
Then the $\xi$-th moment of essential $2$-Selmer groups for $\fX$ is
equal to the $\xi$-th moment of the corresponding matrix model:
\begin{equation}
\label{e:average-moment}
\Avg\big(\#S(E)^\xi\ |\ E\in \fX\big)= \Avg_0 \big(\#S(E)^\xi\ |\ E\in \fX\big)=\sum_{d=0}^\infty P_{r, \mathbf{t}}^{\Alt}(d) 2^{\xi d}.
\end{equation}
In particular, for average order
\begin{equation}
\label{e:average-order}
\Avg\big(\#S(E)\mid E\in\fX\big)
=\Avg_0\big(\#S(E)\mid E\in\fX\big)
=\sum_{d=0}^\infty P_{r,\mathbf t}^\Alt(d)2^d
=3+\sum_i2^{t_i}.
\end{equation}
Moreover, the $\xi$-th moment has error term
that there exists $C>0$ depending only on $\fX$ and $\xi$,
such that for any sufficiently large $N$ (depending on $\fX$)
\begin{equation}
\label{e:average-order-1}
\left|\BE\big(\#S(E)^\xi\mid E\in\fX,N(E)<N\big)-\sum_{d=0}^\infty P_{r, \mathbf{t}}^{\Alt}(d) 2^{\xi d}\right|
<C\cdot\frac{(\log\log N)^{2^{2\xi+1}}}{(\log N)^{1/2^{2\xi}}}.
\end{equation}
The ``essential'' $\xi$-th moment has similar error term
\begin{equation}
\label{e:essential-average-order-1}
\left|\BE\big(\#S(E)^\xi\mid E\in\fX,N(E)<N,\#\Sigma_1(E)\sim\log\log N\big)-\sum_{d=0}^\infty P_{r, \mathbf{t}}^{\Alt}(d) 2^{\xi d}\right|
<C\cdot\frac{\log\log\log N}{(\log\log N)^{1/3}}.
\end{equation}
\end{thm}

In particular, for any equivalence class $\fX$ of type (A), the average order of  $S(E)$  is always $3$, while for any equivalence class $\fX$ of type (B) and (C), the average order of $S(E)$ is strictly greater than $3$.
\begin{remark} The case of non-type (A) curves was also considered by the paper of Yu in 2005
\cite{Yu05},
whose method follows Heath-Brown's,
claimed that as long as
$\CE$ is a family with full rational $2$-torsion points,
there are equivalence classes $\fX\subset\CE$
such that the average order of $S(E)$ in $\fX$ is $3$
(\cite{Yu05}, Theorem 2), but which contradicts with our result.
\end{remark}

\subsection{Relation with other arithmetic problems}

\subsubsection{$\theta$-congruent number problem}

The so called $\theta$-congruent number families are
one-to-one correspondence to
all families of elliptic curves with full rational $2$-torsion points, and the classification of types can be easily formulated.
Therefore Theorems \ref{main theorem} and \ref{main average} apply
to $\theta$-congruent number problem:
for each $\theta$ one can compute the parameter $\mathbf t$
for each equivalence classes of $\theta$-congruent number families,
and according to $\mathbf t$, deduces results on the distribution of $2$-Selmer groups
(in particular, positive density of non-$\theta$-congruent numbers) for them.

For a fixed angle $\theta \in (0, \pi)$, a positive integer $n$ is called a $\theta$-congruent number \cite{Kan00}
if $n\sin \theta$ is the area of a triangle with rational side lengths and with an angle $\theta$ (thus $\cos \theta$ must be rational).
To determine whether $n$ is $\theta$-congruent or not, it is essentially to determine if the elliptic curve (called a $\theta$-congruent number curve)
$$ny^2=(x-\cos \theta)(x^2-1)$$
has positive Mordell-Weil rank.  The $\theta$-congruent number curve family is
of type
\begin{itemize}
\item[(A)]
if $\sin (\theta/2)$, $\cos (\theta/2)$ are non-rational;
\item[(B)]
if $\sin (\theta/2)$, $\cos (\theta/2)$ have one rational and one non-rational;
\item[(C)]
if $\sin (\theta/2)$, $\cos (\theta/2)$ are all rational.
\end{itemize}
The $\theta=\frac{\pi}{2}$ case is the classical congruent number problem
(\cite{HB93}, \cite{HB94}, \cite{T14}, \cite{TYZ}, \cite{Smith} and so on),
corresponding to the family $ny^2=x^3-x$, which is of type (A).

\subsubsection{$X_0(24)$ and tiling number problem}

A positive integer $n$ is called a tiling number,
if an equilateral triangle can be dissected
into $nk^2$ congruent triangles for some integer $k\geq 1$.
Let $E=X_0(24):y^2=x(x-1)(x+3)$.
Then a square-free $n\geq 5$ is a tiling number if any only if one of
the two elliptic curves $E^{(\pm n)}$ defined over $\BQ$ has positive Mordell-Weil rank
\cite{Lac}.
These are families of $\theta$-congruent number curves
of $\theta=\frac{\pi}{2}$ and $\theta=\frac{2\pi}{3}$,
which are of type (B).
In \cite{FLPT} we studied the behaviors of $2$-Selmer groups of them being minimal,
as a consequence, we proved that non-tiling numbers have positive density.
We noticed the differences of them compared to congruent number problem,
which was the beginning of the work in this paper.

\begin{remark}
It's clear that one can put two curves $E^{(\pm n)}$ together, and the followings are equivalent:
\begin{itemize}
\item[(i)]
one of
the two elliptic curves $E^{(\pm n)}$ defined over $\BQ$ has positive Mordell-Weil rank;
\item[(ii)]
the quadratic twist $A^{(n)}$
of the abelian surface $A:=E\times E^{(-1)}$ over $\BQ$ has positive Mordell-Weil rank;
\item[(iii)]
the elliptic curve $E$ over the biquadratic field $\BQ(\sqrt{-1},\sqrt{n})$
has positive Mordell-Weil rank.
\end{itemize}
\end{remark}

\subsubsection{$X_0(15)$ and modularity over imaginary quadratic fields}

Recent work of Caraiani and Newton \cite{CN23} relates
the quadratic twists of $E=X_0(15):y^2=x(x-9)(x-25)$
with modularity over imaginary quadratic fields.
The main theorem in \cite{CN23} states that,
if $n\geq 1$ is such that $E^{(-n)}$
over $\BQ$ is of Mordell-Weil rank $0$,
then all elliptic curves over the imaginary quadratic field
$\BQ(\sqrt{-n})$ is modular.
This family is of type (C), and from the computation of the parameter $\mathbf t_\fX$,
Theorem \ref{main theorem} implies that the set of  such $n$ is of positive density.

\subsection{Idea of proof and organization of this paper}

\subsubsection{Moments}

Theorem \ref{main average} is not a direct consequence of
Theorem \ref{main theorem}, as the error term in
Theorem \ref{main theorem} prevents the terms in
Theorem \ref{main average} from converging.
In Theorem \ref{main average},
the estimation of average order extends Heath-Brown's argument
in \cite{HB93}, \cite{HB94},
which was for the family $ny^2=x^3-x$;
the estimation of the ``essential'' average order extends Kane's argument
in \cite{Kane13}, which was for all type (A) families.
It turns out that in both of these arguments, the algebra and the analysis can be done separately. Moreover, these methods also work for an arbitrary high-rank R\'edei matrix $B$
(Theorem \ref{p:natural-average-order}).
There are following quantities associated to $B$:
\begin{itemize}
\item
a quantity \eqref{e:high-rank-moment} purely in terms of
linear algebra data of $B$,
\item
the ``$\omega$-average order''
which can be think of the limit as $k\to\infty$ of the average order when restricted to
integers with exactly $k$ prime factors,
\item
natural average order
and ``essential'' average order
similar to that in Theorem \ref{main average}.
\end{itemize}
The algebra part
(Proposition \ref{p:high-rank-moment-limit}) asserts that,
the linear algebra quantity \eqref{e:high-rank-moment}
is equal to ``$\omega$-average order''.
The analysis part in \S\ref{s:natural average} and \S\ref{s:essential average},
which is based Heath-Brown's and Kane's work,
asserts that the average order,
as well as the ``essential'' average order,
are also equal to the linear algebra quantity \eqref{e:high-rank-moment},
so they are equal to ``$\omega$-average order''.

Here we emphasize the importance of the existence of the matrix $B$,
which will guide us to the main term of moments computation
when employing the ideas of \cite{HB93}, \cite{HB94} and \cite{Kane13}.
In \cite{Yu05}, the main term and error term are calculated incorrectly due to lack of matrix guidance.

In the case of the $2$-Selmer groups for elliptic curves,
to compare the
``$\omega$-average order''
with the theoretical average order $\sum_{d=0}^\infty P_{r,\mathbf t}^\Alt(d)2^d$
predicted by the matrix model,
we compare the almost Markov chain
for $2$-Selmer groups and the Markov chain for
the matrix model $M_{2k+t,\mathbf t}^\Alt(\BF_2)$.
The estimation of differences of these two transitional probabilities
allows one to deduce that
these two average orders are equal,
similarly for moments
(Theorem \ref{redei matrix main thm}
and \eqref{e:redei limit moment equal}).

We remark that in \cite{HB93}, \cite{HB94} and \cite{Kane13}
which deals with type (A) case,
the ``$\omega$-average order'' and the average order predicted by type (A) matrix model
$M_{2k+t}^\Alt(\BF_2)$, as well as moments, are computed directly, and compared to be equal,
without using Markov chain.
On the other hand, in our arguments, these two are not computed directly
(i.e.~without closed form expression),
but they are compared via Markov chain argument.

The reason to study the ``essential'' average order is that,
among positive square-free integers $<N$,
the ones with number of prime factors $r\sim\log\log N$
are of density $1-O(\exp(-c(\log\log N)^{1/3}))=1-o(1)$
as $N\to\infty$,
hence it is expected that the property of these ``good'' integers
reflects the property of all positive square-free integers.

\subsubsection{From moments to distribution, or from Markov chain to distribution}

Conversely, one may ask if Theorem \ref{main theorem}
(without error term estimation) is a consequence of Theorem \ref{main average}.
This is indeed the case for type (A) families.
For type (A),
in Heath-Brown's \cite{HB94} and Kane's \cite{Kane13} work,
the estimation of moments
give rise to the distribution of rank of $2$-Selmer groups
without error term estimation (Theorem \ref{main type A}).
We expect that
these methods also work for type (B).
However, it's unlikely for type (C) that
they still work without significant modification,
as the moments in Theorem \ref{main average} grows too fast for type (C).

As we will see in Appendix \ref{s:model},
for $\mathbf t=\varnothing=(-\infty,-\infty)$ for type (A),
or $\mathbf t=(t_1)=(t_1,-\infty)$ for type (B),
or $\mathbf t=(t_1,t_2)$ for type (C),
the generating function of $P_{r,\mathbf t}^\Alt(d)$ is
$$
\sum_{d=0}^\infty P_{r,\mathbf t}^\Alt(d)X^d
=F_{r,\mathbf t}^\Alt(X)
:=F_{\mathbf t}^\Alt(X)+(-1)^rF_{\mathbf t}^\Alt(-X)\in\BR[[X]],
$$
where the formal power series
$$
F_{\mathbf t}^\Alt(X):=\frac{\prod_{j=0}^\infty(1+2^{-j}X)}
{\prod_{j=0}^\infty(1+2^{-j})}
\sum_{i_1,i_2\geq 0}2^{i_1t_1+i_2t_2}
\cdot\frac{2^{2i_1i_2}\prod_{j=1}^{i_1+i_2}(2^{1-j}X-1)}
{\prod_{j=1}^{i_1}(2^{2j}-1)\prod_{j=1}^{i_2}(2^{2j}-1)}\in\BR[[X]].
$$
Here if $t_1=-\infty$ or $t_2=-\infty$, then the corresponding $i_1$
or $i_2$ in the sum only takes $0$.
The generating function converges absolutely on the whole complex plane.
Hence it gives formula for the $\xi$-th moment $\sum_{d=0}^\infty P_{r,\mathbf t}^\Alt(d)2^{\xi d}$ for all $\xi\in\BZ_{\geq 0}$ by taking $X=2^\xi$;
in this case, the $i_1$ and $i_2$ only takes
$i_1,i_2\geq 0$, $i_1+i_2\leq\xi$.
This also gives the upper bound and lower bound of moments:
they are $2^{\xi^2/2+O(\xi)}$, $2^{2\xi^2/3+O(\xi)}$
and $2^{\xi^2+O(\xi)}$ for types (A), (B) and (C), respectively.

For type (C) the moments grow too fast (that is, $2^{\xi^2+O(\xi)}$),
the moment-to-distribution methods of Heath-Brown and Kane
do not apply.
We use the following approach to prove the distribution result Theorem \ref{main theorem}
instead, for all types (A), (B) and (C), without using moments.
\begin{itemize}
\item
The $2$-Selmer group $\Sel_2(E^{(n)}/\BQ)$
depends only on the $\Sigma$-equivalence classes and Legendre symbols of prime factors of $n$.
This is well-known in \cite{HB94} and \cite{SD08}.
More importantly, the $2$-Selmer group is in fact the kernel of some matrix over $\BF_2$
whose coefficients are Legendre symbols of prime factors of $n$
arranged in a particularly uniform way (we will call them R\'edei matrices).
This is suggested in the proofs in \cite{HB94} and \cite{SD08}.
This is discussed in \S\ref{s:2-Selmer}.
\item
Thanks to the existence of such matrix whose kernel is $2$-Selmer,
we can analyze the behavior of $\Sel_2(E^{(n)}/\BQ)$ as the number $k$ of prime factors
of $n$ increases.
Swinnerton-Dyer's work \cite{SD08} also did this (without explicit matrix) to some extent.
Fully writing down the matrix allows us to discover the ``holes''
in the matrix, similarly to the matrix model $M_{2k+r,\mathbf t}^\Alt(\BF_2)$.
For type (A), the corank of such matrix already form an ``almost Markov chain'',
as in \cite{SD08}.
For types (B) and (C) they are not: since the existence of ``holes'',
to determine the transition probability one must know the corank of submatrices
involving ``holes''. This is also reflected in the Markov chain corresponding
to the relevant matrix model $M_{2k+r,\mathbf t}^\Alt(\BF_2)$.
By the same refinement method in the matrix model, we consider the joint distribution
of corank of the matrix representing $2$-Selmer, as well as coranks of its submatrices,
then they form an ``almost Markov chain'' whose limit is the Markov chain
corresponding to the relevant matrix model.
The analysis of this process is discussed in \S\ref{s:markov},
and the Markov chain
corresponding to the relevant matrix model is discussed in Appendix \S\ref{s:model}.
\item
Now we know that, if assuming the equidistribution of Legendre symbols,
then the limit as $k\to\infty$ of the distribution of $\Sel_2(E^{(n)}/\BQ)$ with the number $k$ of prime factors of $n$ fixed, coincides with the prediction of
the relevant matrix model $M_{2k+r,\mathbf t}^\Alt(\BF_2)$.
To compare this with the natural density in Theorem \ref{main theorem},
the equidistribution of Legendre symbols need to be established.
This is mainly done in Smith's earlier work \cite{Smith2}.
In fact, one may deduce that, if for each $k$,
a permutation invariant subset of all possible Legendre symbols of $k$ primes is given,
such that as $k\to\infty$ the limit density of these subsets exists,
then it is equal to the natural density of the subset of positive square-free integers
associated to these subsets of all possible Legendre symbols,
and with error term estimations.
This is discussed in \S\ref{s:natural-density}.
\end{itemize}
Finally, one may find details of proof of main results in \S\ref{s:proof}.

\subsection{Joint distribution with $\phi$-Selmer groups}

We also study the joint distribution
of $2$-Selmer group with $\phi$-Selmer groups in quadratic twist families.

Let $\CE$ be a family with full rational $2$-torsion points.
Let $\phi:\CE\to\CE_\phi$ be a compatible family of rational $2$-isogenies.
There are three different choices of $\phi$.
For each $\phi:E\to E_\phi$ in $\CE\to\CE_\phi$, there is
$\phi$-Selmer group $\Sel_\phi(E/\BQ)$
which injects into $2$-Selmer group:
$\Sel_\phi(E/\BQ)\hookrightarrow\Sel_2(E/\BQ)$.
We define essential $2$-Selmer group $S(E)$
and essential $\phi$-Selmer group $S_\phi(E)\hookrightarrow S(E)$ as
\begin{equation}
\label{e:essential-2-Selmer}
S(E):=\Sel_2(E/\BQ)/E[2],\qquad
S_\phi(E):=\Sel_\phi(E/\BQ)/E[\phi].
\end{equation}
On the other hand, the $\phi$
induces the map
$\pi:E[2]\stackrel\phi\twoheadrightarrow\phi(E[2])\cong\BF_2$,
and it is easy to see that the $\phi$-Selmer group $\Sel_\phi(E/\BQ)$
is equal to the following $\pi$-strict $2$-Selmer group:
$$
\Sel_{2,\pi\sstr}(E/\BQ):=
\ker(\Sel_2(E/\BQ)\hookrightarrow H^1(\BQ,E[2])\xrightarrow\pi H^1(\BQ,\BF_2))
\subset\Sel_2(E/\BQ).
$$
In \S\ref{s:2-Selmer} we will see the key point that,
the $2$-Selmer is kernel of some matrix $B$,
while the $\pi$-strict $2$-Selmer is the kernel of certain submatrix of $B$.

We define relevant model of random rectangular matrices,
describing the distribution of $\phi$-Selmer groups.
The model is similar to the random matrix models for class groups in \cite{Ger84}.
For $t\in\BZ\cup\{-\infty\}$ and $d\geq 0$,
define
$$
P_t^\Mat(d):=\lim_{k\to\infty}\BP\big(\corank(B)=d\mid
B\in M_{(k-t)\times k}(\BF_2)\big)
\text{ for }t\in\BZ,
\quad\text{and}\quad
P_{-\infty}^\Mat(d):=\begin{cases}
1,&\text{if }d=0, \\
0,&\text{if }d\geq 1.
\end{cases}
$$

\begin{thm}
\label{main strict}
Let $\CE$ be a family with full rational $2$-torsion, and
$\fX\subset \CE$ a $\Sigma$-equivalence class.
Then for each of the three $\phi$ as above,
there exists a $t_\phi=t_{\phi,\fX}\in\BZ\cup\{-\infty\}$
depending only on $\phi$ and $\fX$,
and absolute constants $C_1>0$ and $c>0$,
such that for any sufficiently large $N$
(depending on $\fX$),
$$
\sum_{d\geq 0}\left|\BP\big(
\dim_{\BF_2}S_\phi(E)=d\mid E\in\fX,N(E)<N
\big)-P_{t_\phi}^\Mat(d)\right|<C_1\cdot\exp(-c(\log\log\log N)^{1/2}).
$$
The number of $\phi$ such that $t_\phi\neq-\infty$ is equal to the $s$
in Theorem \ref{main theorem},
and the $t_\phi$'s not equal to $-\infty$
is equal to $\mathbf t=(t_1,\cdots,t_s)$
in Theorem \ref{main theorem}.
\end{thm}

Now we recall the matrix model $M_{2k+r,\mathbf t}^\Alt(\BF_2)$.
Recall that for type (B), there is a submatrix $B_1'$ associated to $B$,
and for type (C), there are submatrices $B_1',B_2'$ associated to $B$.
Recall that $Y_k=(X_k,\corank(B_1'),\cdots,\corank(B_s'))$
is a refinement of $X_k=\corank(B)$.
The $Y_k$ give rise to a joint probability:
$$
P_{r,\mathbf t}^\Alt(d,d_1',\cdots,d_s'):=\lim_{k\to\infty}\BP\left(
\begin{array}{l}
\corank(B)=d\text{ and} \\
\corank(B_i')=d_i'\text{ for }1\leq i\leq s
\end{array}
\middle|~B\in M_{2k+r,\mathbf t}^\Alt(\BF_2)\right).
$$
It's clear that for $i=1,\cdots,s$,
$$
\lim_{k\to\infty}\BP\big(
\corank(B_i')=d_i'\mid B\in M_{2k+r,\mathbf t}^\Alt(\BF_2)\big)
=P_{t_i}^\Mat(d_i'),
$$
hence $M_{2k+r,\mathbf t}^\Alt(\BF_2)$ also extends the model $M_{(k-t_i)\times k}(\BF_2)$.

\begin{thm}
\label{main theorem refined}
Let $\CE$ be a family with full rational $2$-torsion, and
$\fX\subset \CE$ a $\Sigma$-equivalence class.
Let $r=r(\fX)$ and $\mathbf t=\mathbf t_\fX=(t_1,\cdots,t_s)$ be the parameter,
depending only on $\fX$.
Then there exist absolute constants $C_1>0$ and $c>0$, such that
the joint distribution of $2$-Selmer groups with $\phi$-Selmer groups
satisfies that for any sufficiently large $N$
(depending on $\fX$),
\begin{multline}
\label{e:main-theorem-refined}
\sum_{d,d_1',\cdots,d_s'\geq 0}\left|\BP\left(
\begin{array}{l}
\dim_{\BF_2}S(E)=d\text{ and} \\
\dim_{\BF_2}S_{\phi_i}(E)=d_i'\text{ for }1\leq i\leq s
\end{array}
\middle|~E\in\fX,N(E)<N
\right)-P_{r,\mathbf t}^\Alt(d,d_1',\cdots,d_s')\right| \\
<C_1\cdot\exp(-c(\log\log\log N)^{1/2}).
\end{multline}
\end{thm}

Theorem \ref{main average} also holds
for essential $\phi$-Selmer groups $S_\phi(E)$:
in \eqref{e:average-moment} it is $\sum_{d=0}^\infty P_{t_\phi}^\Mat(d)2^{\xi d}$,
in \eqref{e:average-order} it is $1+2^{t_\phi}$,
and the error term in \eqref{e:average-order-1}
is $C\cdot(\log\log N)^{2^{\xi+1}}/(\log N)^{1/2^{2\xi}}$.

\subsection{Related and future work}

\subsubsection{Chronology} Some results of this paper were announced in the 2022 ICM talk of the second-named author \cite{ICM} and also in \cite{BT}.

Upon completion of this paper, we became aware of the recent preprint \cite{Smith25} of Smith. The latter proves striking results on the distribution of coranks of $2^\infty$-Selmer groups in quadratic twist families. However, it does not pursue the Markov chain model for the distribution of ranks of $2$-Selmer groups. In contrast the Markov chain aspect of the distribution is the focus of this paper, which is bound to lead to a refinement of the results of \cite{Smith25}.

\subsubsection{Propsects} The results and ideas of this paper lead to Markov chain property for the distribution of $2$-Selmer groups in the remaining case that $\BQ(E[2])$ is quadratic, and will appear in a sequel. As alluded to above, they are also relevant to the Selmer counterpart of the Goldfeld conjecture, pertaining to distribution of coranks of $2^\infty$-Selmer groups.

\section{$2$-Selmer groups and R\'edei matrix}
\label{s:2-Selmer}

Let $\Sigma$ be a finite set containing $-1,2$ and some primes of $\BQ$.
For each $k\geq 0$
let $S_k(\Sigma)$ be the set of positive square-free integers coprime to $\Sigma$
with exactly $k$ prime factors,
and let $S(\Sigma):=\bigsqcup_{k\geq 0}S_k(\Sigma)$
be the set of positive square-free integers coprime to $\Sigma$.
Let $\fX$ be a $\Sigma$-equivalence class of elliptic curves over $\BQ$
with full rational $2$-torsion points.
Then any curve in $\fX$ can be written as $E^{(n)}$
for some $n\in S(\Sigma)$ and some elliptic curve $E/\BQ$
depending only on $\fX$,
with full rational $2$-torsion points
and has good reduction outside $\Sigma$.
Recall that the $2$-Selmer group $\Sel_2(E^{(n)}/\BQ)$ is
\begin{equation}
\label{e:2-Selmer}
\Sel_2(E^{(n)}/\BQ):=\ker\left(H^1(G_{S,\BQ},E^{(n)}[2])\to
\bigoplus_{p\in S}\frac{H^1(\BQ_p,E^{(n)}[2])}
{\Im(\kappa_p)}\right),
\end{equation}
here $S$ is any finite set of places of
$\BQ$ containing $2$, $\infty$ and bad places for $E^{(n)}$,
and $\kappa_p:E^{(n)}(\BQ_p)/2E^{(n)}(\BQ_p)\hookrightarrow H^1(\BQ_p,E^{(n)}[2])$
is the local Kummer map.
For $\pi:E^{(n)}[2]\cong E[2]\to\BF_2$
we define the $\pi$-strict $2$-Selmer group
$$
\Sel_{2,\pi\sstr}(E^{(n)}/\BQ):=
\ker(\Sel_2(E^{(n)}/\BQ)\xrightarrow\pi H^1(\BQ,\BF_2))
\subset\Sel_2(E^{(n)}/\BQ).
$$
If $\phi:E\to E_\phi$ is a rational $2$-isogeny,
which also gives $\phi:E^{(n)}\to E_\phi^{(n)}$, then it induces the map
$\pi:E^{(n)}[2]\cong E[2]\stackrel\phi\twoheadrightarrow\phi(E[2])\cong\BF_2$,
and it's easy to see that the usual $\phi$-Selmer group
$\Sel_\phi(E^{(n)}/\BQ)$ is equal to the $\pi$-strict $2$-Selmer group
$\Sel_{2,\pi\sstr}(E^{(n)}/\BQ)$ for such specifically chosen $\pi$.
So our main results can be rephrased by using $\pi$-strict $2$-Selmer groups.
As in \eqref{e:essential-2-Selmer},
we define the essential $2$-Selmer group
$S(E^{(n)}):=\Sel_2(E^{(n)}/\BQ)/E^{(n)}[2]$
and essential $\phi$-Selmer group $S_\phi(E^{(n)})$,
similarly we can define the $\pi$-strict essential $2$-Selmer group
$S_{\pi\sstr}(E^{(n)}):=\Sel_{2,\pi\sstr}(E^{(n)}/\BQ)/\ker(\pi)\subset
S(E^{(n)})$
which is equal to the essential $\phi$-Selmer group $S_\phi(E^{(n)})$.
Later we will see that the following modified
$\pi$-strict $2$-Selmer group $\Sel_{2,\pi\sstr}'(E^{(n)}/\BQ)$
is also useful:
$$
\Sel_{2,\pi\sstr}(E^{(n)}/\BQ)
\subset\Sel_{2,\pi\sstr}'(E^{(n)}/\BQ):=
\ker\left(\Sel_2(E^{(n)}/\BQ)\xrightarrow\pi
H^1(\BQ,\BF_2)\xrightarrow{\loc_{\Sigma_1}}
\bigoplus_{p\in\Sigma_1}H^1(\BQ_p,\BF_2)\right),
$$
where $\Sigma_1$ is the set of bad places of $E^{(n)}$ outside $\Sigma$.
Namely, it only requires that the localization at $\Sigma_1$ is contained in
the kernel of $\pi$.
Similarly define $S_{\pi\sstr}'(E^{(n)})$.

We will see by $2$-descent method (Theorem \ref{p:2-descent-redei-mat},
\ref{p:2-descent-redei-mat-2}, \ref{p:2-descent-redei-mat-0} and
Corollary \ref{selmer mat unrestricted version}) that

\begin{thm}
\label{p:2-Sel-depends-only-on-lrd}
If $n=\ell_1\cdots\ell_k$ and $n'=\ell_1'\cdots\ell_k'$ are two elements
of $S_k(\Sigma)$ such that $\ell_i/\ell_i'\in(\BQ_v^\times)^2$
and $\lrd{\ell_j}{\ell_i}=\lrd{\ell_j'}{\ell_i'}$
for all $i,j$ and $v$,
then under the isomorphism $H^1(G_{S,\BQ},E^{(n)}[2])\cong
\BQ(S,2)^{\oplus 2}\xrightarrow\sim H^1(G_{S',\BQ},E^{(n')}[2])\cong
\BQ(S',2)^{\oplus 2}$, $\ell_i\mapsto\ell_i'$, $p\mapsto p$ for $p\in\Sigma$,
here $S:=\Sigma\cup\{\ell_1,\cdots,\ell_k\}$ and
$S':=\Sigma\cup\{\ell_1',\cdots,\ell_k'\}$,
we have
$\Sel_2(E^{(n)}/\BQ)=\Sel_2(E^{(n')}/\BQ)$,
$\Sel_{2,\pi\sstr}(E^{(n)}/\BQ)=\Sel_{2,\pi\sstr}(E^{(n')}/\BQ)$,
and $\Sel_{2,\pi\sstr}'(E^{(n)}/\BQ)=\Sel_{2,\pi\sstr}'(E^{(n')}/\BQ)$.
Same as essential versions.
\end{thm}

For $2$-Selmer groups this result is well-known
(see Monsky's Appendix
in \cite{HB94}, or Swinnerton-Dyer's work \cite{SD08}).
The argument is similar for $\pi$-strict $2$-Selmer groups.
In this paper we want to emphasize the crucial point
that this dependence is ``uniform'' as $k$
varies when restricted to $\Sigma$-equivalence classes
$\fX$ of $\CE$.

For each $k\geq 0$ and each element $n\in S_k(\Sigma)$, we fix an order $\ell_1,\cdots,\ell_k$
of its prime factors,
and define the following matrices and vectors
of coefficients in $\BF_2$ associated to $n$.
\begin{itemize}
\item
Define $A=A(n):=(a_{ij})_{\substack{1\leq i\leq k\\
1\leq j\leq k}}\in M_k(\BF_2)$
where $a_{ij}:=\lrdd{\ell_j}{\ell_i}$ if $i\neq j$,
and $a_{ii}:=\lrdd{n/\ell_i}{\ell_i}$.
Here $\lrdd{\ell_j}{\ell_i}:=\frac{1}{2}\left(1-\lrd{\ell_j}{\ell_i}\right)\in\BF_2$
is the additive Legendre symbol.
\item
For $d\in\BQ(\Sigma,2)$,
define $z_d=z_d(n):=(z_1^{(d)},\cdots,
z_k^{(d)})^{\mathrm T}\in\BF_2^k$,
where $z_i^{(d)}:=\lrdd{d}{\ell_i}$.
The $z_d^{\mathrm T}$ is the transpose of $z_d$.
Define $D_d:=\diag(z_d)\in M_k(\BF_2)$
and $A_d:=D_d+z_dz_d^{\mathrm T}\in M_k(\BF_2)$.
\end{itemize}
Clearly $\sum_{j=1}^ka_{ij}=0$ for all $i$,
also by quadratic reciprocity law,
$a_{ji}=a_{ij}+z_i^{(-1)}z_j^{(-1)}$ for all $i\neq j$;
equivalently, $A+A^{\mathrm T}=A_{-1}$.
This motivates us to define the probability space of Legendre symbols
\begin{equation}
\label{e:Omega-k-Sigma-*}
\Omega_k^{\Sigma,*}:=
\left\{
\big((a_{ij}),(z_p)_{p\in\Sigma}\big)\in
M_k(\BF_2)\times(\BF_2^k)^{\#\Sigma}
\ \middle|\begin{array}{l}
a_{ji}=a_{ij}+z_i^{(-1)}z_j^{(-1)}\text{ for all }i\neq j, \\
\sum_ja_{ij}=0\text{ for all }i
\end{array}
\right\},
\end{equation}
and a map $S_k(\Sigma)\to\Omega_k^{\Sigma,*}$,
$n\mapsto\omega(n)=\big(A(n),(z_p(n))_{p\in\Sigma}\big)$.
The above $A$, $D_d$ and $z_cz_d^\RT$ for $c,d\in\BQ(\Sigma,2)$,
are viewed as maps $\Omega_k^{\Sigma,*}\to M_k(\BF_2)$
for each $k$; $z_d$ are maps $\Omega_k^{\Sigma,*}\to\BF_2^k$
for each $k$.
The Theorem \ref{p:2-Sel-depends-only-on-lrd}
that $2$-Selmer groups depend only on Legendre symbols
is equivalent to say
that the $\Sel_2(E^{(n)}/\BQ)$ as well as $S_{\pi\sstr}(E^{(n)})$
depends only on $\omega=\omega(n)\in\Omega_k^{\Sigma,*}$.
By abuse of notation, we denote these $2$-Selmer groups by
$\Sel_2(E^{(\omega)}/\BQ)$, $S_{\pi\sstr}(E^{(\omega)})$, and so on.

To describe the ``uniform'', we define the concept of
a \emph{R\'edei matrix}, which
is a block matrix whose entries are linear combinations
of $z_d$, $z_d^{\mathrm T}$, $A$, $D_d$ and $z_cz_d^{\mathrm T}$,
where $c,d\in\BQ(\Sigma,2)$.
A R\'edei matrix $B$ is viewed as a ``uniform'' rule,
which maps an element $\omega$ of $\Omega_k^{\Sigma,*}$ for each $k$,
to a matrix $B(\omega)$;
it also maps an element $n$ of $S(\Sigma)$
to a matrix $B(n)$.
The set of R\'edei matrices of size $(ak+b)\times(ck+d)$
is denoted by $\RRM_{(ak+b)\times(ck+d)}$.

For a R\'edei matrix $B$ of size $k\times k$
(namely, linear combinations of $A$, $D_d$ and $z_cz_d^{\mathrm T}$),
write it uniquely as $B=B_H+B_M+B_L$
``the decomposition by ranks'',
where
\begin{itemize}
\item
$B_H$ is the ``high-rank part'' which is linear combination of $A$,
\item
$B_M$ is the ``medium-rank part'' which is linear combination of $D_d$,
\item
$B_L$ is the ``low-rank part'' which is linear combination of $z_cz_d^\RT$.
\end{itemize}
The $B$ is called
of ``high-rank'' if $B_H\neq 0$;
of ``medium-rank'' if $B_H=0$ and $B_M\neq 0$;
of ``low-rank'' if $B_H=B_M=0$.
Here we give a remark but without proof that,
these $3$ types of $k\times k$ R\'edei matrices are classified
according to their behavior of ranks
when viewed as a random variable $\Omega_k^{\Sigma,*}\to M_k(\BF_2)$:
\begin{itemize}
\item
a ``high-rank R\'edei matrix'' is that,
its rank is mostly $k+o(k)$ as $k\to\infty$;
\item
a ``medium-rank R\'edei matrix'' is that,
its rank is mostly $\frac{1}{2}k+o(k)$ as $k\to\infty$;
\item
a ``low-rank R\'edei matrix'' is that,
its rank is mostly $o(k)$ as $k\to\infty$.
\end{itemize}

The main result of this section is as follows,
which is a reinterpretation of Swinnerton-Dyer's work \cite{SD08}.

\begin{thm}
\label{p:2-descent-redei-mat}
Let $\fX$ be a $\Sigma$-equivalence class
of quadratic twist family of elliptic curves over $\BQ$
with full rational $2$-torsion points,
as in the beginning of the introduction.
Then there exists an integer $t_\fX\equiv r(\fX)\pmod 2$ depends only on $\fX$,
and a R\'edei matrix
$$
B=\begin{pmatrix}
B_{11} & B_{12} & B_{13} \\
B_{21} & B_{22} & B_{23} \\
B_{31} & B_{32} & B_{33}
\end{pmatrix}\in\RRM_{(2k+t_\fX)\times(2k+t_\fX)},
$$
depends only on $\fX$, where
\begin{itemize}
\item
$B$ is alternating, namely $B_{ji}=B_{ij}^\RT$
and the diagonal of $B_{ii}$ is $0$,
\item
$B_{11},B_{12},B_{21},B_{22}\in\RRM_{k\times k}$,
$B_{13},B_{23}\in\RRM_{k\times t_\fX}$,
and $B_{33}=0\in M_{t_\fX\times t_\fX}(\BF_2)$,
\item
the entries of $B_{13}$ and $B_{23}$ are $z_d$, $d\in\BQ(\Sigma,2)$,
and $\left(\begin{smallmatrix}
B_{13} \\ B_{23}
\end{smallmatrix}\right)\in\RRM_{2k\times t_\fX}$
is of generically full rank,
which means that, there exists $k\geq 1$ and $\omega\in\Omega_k^{\Sigma,*}$
such that
$\left(\begin{smallmatrix}
B_{13}(\omega) \\ B_{23}(\omega)
\end{smallmatrix}\right)$
is of rank $t_\fX$,
\item
$B_{12}$ and $B_{21}$ are high-rank R\'edei matrices;
the three matrices $B_{11}$, $B_{22}$ and $B_{11}+B_{12}+B_{21}+B_{22}$
are all medium-rank or low-rank R\'edei matrices,
and the number of low-rank R\'edei matrices among them
is $0$, $1$ or $2$ if $\fX$ is of type (A), (B) or (C), respectively,
\end{itemize}
such that for any $E\in\fX$,
define $B(E)$ to be $B(n)$
where $n\in S(\Sigma)$ is the product of bad primes of $E$
outside $\Sigma$,
then
$\dim_{\BF_2}\Sel_2(E/\BQ)=\corank(B(E))$,
and the three modified $\pi$-strict $2$-Selmer groups
$\Sel_{2,\pi\sstr}'(E/\BQ)$ are of dimensions
$\corank(B_j'(E))$, $j=1,2,3$, where
$B_j'\in\RRM_{(2k+t_\fX)\times(k+t_\fX)}$ are
$$
B_j':=\begin{pmatrix}
B_{1j} & B_{13} \\
B_{2j} & B_{23} \\
B_{3j} & B_{33}
\end{pmatrix}\text{ for }j=1,2,
\qquad\text{and}\qquad
B_3':=\begin{pmatrix}
B_{11}+B_{12}+B_{21}+B_{22} & B_{13}+B_{23} \\
B_{21}+B_{22} & B_{23} \\
B_{31}+B_{32} & B_{33}
\end{pmatrix},
$$
in which exactly one $k\times k$ block is high-rank R\'edei matrix
($B_{21}$, $B_{12}$ and $B_{21}+B_{22}$, respectively).
\end{thm}

The entries $B_{ij}$ of $B$ has more precise description
if we fix equations of $E$ in $\fX$, see Theorem
\ref{p:2-descent-redei-mat-2}.

Later in \S\ref{s:gram} we will see that for $E\in\fX$,
the $B(E)$ is the Gram matrix of certain alternating form $\theta$ whose kernel is isomorphic to $\Sel_2(E/\BQ)$.
Furthermore, when $E$ has a bad prime outside $\Sigma$, then the $2$-torsion points
$E[2]$ injects into $\Sel_2(E/\BQ)$, which
implies that the sum of columns of $\left(\begin{smallmatrix}
B_{1i} \\ B_{2i} \\ B_{3i}
\end{smallmatrix}\right)$, $i=1,2$,
are linear combinations of columns of $\left(\begin{smallmatrix}
B_{13} \\ B_{23} \\ B_{33}
\end{smallmatrix}\right)$.
For the modified $\pi$-strict $2$-Selmer group it's similar.
Hence we have the following corollary.

\begin{cor}
\label{selmer mat unrestricted version}
For $E\in\fX$ with a bad prime outside $\Sigma$,
the dimension of its essential $2$-Selmer group
\eqref{e:essential-2-Selmer} is
$\dim_{\BF_2}S(E)=\corank(B(E))-2=\corank(\widetilde B(E))$ with
$\widetilde B(E)=\widetilde B(n)=(\widetilde B_{ij})_{\substack{1\leq i\leq 3\\
1\leq j\leq 3}}$,
where $n\in S_{\geq 1}(\Sigma):=\bigsqcup_{k\geq 1}S_k(\Sigma)$
is the product of bad primes of $E$ outside $\Sigma$,
\begin{itemize}
\item
$\widetilde B_{11},\widetilde B_{12},\widetilde B_{21},\widetilde B_{22}$
are the corresponding $B_{ij}$ with last row and column removed,
\item
$\widetilde B_{13},\widetilde B_{23}$
are the corresponding $B_{ij}$ with last row removed,
\item
$\widetilde B_{31},\widetilde B_{32}$
are the corresponding $B_{ij}$ with last column removed,
\item
$\widetilde B_{33}=B_{33}$.
\end{itemize}
Similarly, there are $\widetilde B_j'$
such that the three modified $\pi$-strict essential $2$-Selmer group $S_{\pi\sstr}'(E)$
are of dimensions $\corank(\widetilde B_j'(E))$, $j=1,2,3$.
\end{cor}

We introduce crucial invariants for an equivalence class $\fX$ of type (B) and (C).

\begin{defn}
\label{parameter defn 1}
Let $s$ be the number of low-rank R\'edei matrices among
$B_{11}$, $B_{22}$ and $B_{11}+B_{12}+B_{21}+B_{22}$,
namely, $0$, $1$ or $2$ if $\fX$ is of type (A), (B) or (C), respectively.
If $\fX$ is of type (B), we may assume that
$B_{11}$ is low-rank R\'edei matrix;
if $\fX$ is of type (C), we may assume that
$B_{11}$ and $B_{22}$ are low-rank R\'edei matrices.
Define
$B_j'':=\left(\begin{smallmatrix}
B_{jj} & B_{j3} \\
B_{3j} & B_{33}
\end{smallmatrix}\right)$
for $1\leq j\leq s$, which is a submatrix of $B_j'$;
similarly define $\widetilde B_j''$ which is a submatrix of $\widetilde B_j'$.
Note that $\rank(B_j''(E))=\rank(\widetilde B_j''(E))$ for $1\leq j\leq s$
and $E\in\fX$ which has a bad prime outside $\Sigma$.
Define the parameter $\mathbf t=\mathbf t_\fX=(t_1,\cdots,t_s)$ of $\fX$,
which is an unordered tuple of $s$ integers, to be
$$
t_j:=t_\fX-\max_{E\in\fX}\rank(B_j''(E))
=t_\fX-\max_{E\in\fX}\rank(\widetilde B_j''(E)).
$$
\end{defn}

Proposition \ref{p:low-rank-max-rank}
gives us a way to find $E\in\fX$ such that
$\rank(B_j''(E))=\rank(\widetilde B_j''(E))$ is maximal,
hence the parameter $t_j$ is computable.
In \S\ref{s:inv of param} we will see that
if $\Sigma\subset \Sigma'$ and $\fX'\subset \fX$ is a $\Sigma'$-equivalence class,
then the class $\fX'$ has the same parameter $\mathbf t=\mathbf t_\fX$
as that of $\fX$.

There is an obvious lower bound of Selmer rank.

\begin{prop}
\label{selmer lower bound}
Let $\fX$ be a $\Sigma$-equivalence class of type (B) or (C)
with parameter $\mathbf t=(t_i)$.
Then for any $E\in\fX$ with a bad prime outside $\Sigma$,
for any $1\leq i\leq s$, the $\dim_{\BF_2}S_{\pi_i\sstr}'(E)$
and $\dim_{\BF_2}S(E)$ satisfy
$0\leq\dim_{\BF_2}S_{\pi_i\sstr}'(E)\leq \dim_{\BF_2}S(E)
\leq 2\dim_{\BF_2}S_{\pi_i\sstr}'(E)-t_i$.
In particular, we have $\dim_{\BF_2}S_{\pi_i\sstr}'(E)\geq t_i$
and $\dim_{\BF_2}S(E)\geq\max_i(t_i)$.
\end{prop}

\begin{proof}
By Corollary \ref{selmer mat unrestricted version}
we have $\dim_{\BF_2}S_{\pi_j\sstr}'(E)=\corank(\widetilde B_j'(E))$
and $\dim_{\BF_2}S(E)=\corank(\widetilde B(E))$.
By linear algebra,
$$
0\leq\corank(\widetilde B_j'(E))\leq\corank(\widetilde B(E)),
$$
and
$$
2\rank(\widetilde B_j'(E))\leq\rank(\widetilde B(E))+\rank(\widetilde B_j''(E))
\leq\rank(\widetilde B(E))+t_\fX-t_j,
$$
so
$$
\corank(\widetilde B(E))\leq 2\corank(\widetilde B_j'(E))-t_j.
$$
Hence $\corank(\widetilde B(E))\geq\corank(\widetilde B_j'(E))\geq\max\{t_j,0\}$.
\end{proof}

\begin{example}
Let $E:y^2=x(x-1)(x+3)$ of type (B).
Then for any $\fX$ contained in
$\{E^{(-n)}\mid n\equiv 1\pmod{12}\text{ positive square-free}\}$,
its parameter $\mathbf t_\fX=(2)$.
This can be seen from the matrix calculation in \cite{FLPT}.
\end{example}

\subsection{The alternating matrix associated to $2$-Selmer groups}
\label{s:gram}

\subsubsection{$2$-Selmer groups and linear algebra}

Let $E/\BQ$ be an elliptic curve with full rational $2$-torsion points,
$S$ be a finite set of places of $\BQ$ containing $2$, $\infty$ and bad places for $E$.
Besides \eqref{e:2-Selmer}, there is another point of view of
the $2$-Selmer group $\Sel_2(E/\BQ)$:
\begin{itemize}
\item
let $V=\bigoplus_{p\in S}V_p=\bigoplus_{p\in S}H^1(\BQ_p,E[2])$
be the space of local Galois cohomology,
\item
let $U:=\Im\big(\loc_S:H^1(G_{S,\BQ},E[2])\hookrightarrow V\big)\subset V$
be the image of global Galois cohomology,
\item
let $W=\bigoplus_{p\in S}W_p=\bigoplus_{p\in S}\Im(\kappa_p)\subset V$
be the image of local Kummer maps,
\end{itemize}
then $\Sel_2(E/\BQ)
\cong\Im\big(\loc_S:\Sel_2(E/\BQ)\hookrightarrow V\big)
=U\cap W\subset V$.
The space $V$ is endowed with a non-degenerating alternating pairing
$e:V\times V\to\BF_2$ which is induced by Tate local duality and the Weil pairing
$e_E:E[2]\times E[2]\to\BF_2$.
The $U$ and $W$ are two Lagrangian subspaces of the symplectic space
$(V,e)$.

Note the following linear algebra fact:
if $(V,e)$ is a symplectic space over $\BF_2$, and $W,U,K$ are three Lagrangian subspaces
satisfying $V=U\oplus K$, then
$$
\theta:W/(W\cap K)\times W/(W\cap K)\to\BF_2,
\quad
(x,y)\mapsto e(x_U,y_K)
$$
is a symmetric pairing, and the natural projection $W\twoheadrightarrow W/(W\cap K)$
gives $W\cap U\cong\ker(\theta)$;
furthermore, if there exists a ``quadratic form'' $\phi:V\to\BF_2$ associated to $e$
in the sense that $e(x,y)=\phi(x+y)+\phi(x)+\phi(y)$ for all $x,y\in V$,
and such that $\phi(U)=\phi(W)=\phi(K)=0$,
then the $\theta$ is alternating.
Based on this observation, Swinnerton-Dyer \cite{SD08}
constructs systematic $K$ when $E$ varies in a quadratic twist family,
such that $\theta$ is further alternating.
In the following we recall this construction.

\subsubsection{$2$-Selmer groups in quadratic twist family}

Let $\fX$ be a $\Sigma$-equivalence class with full rational $2$-torsion points,
and $E\in\fX$. Let $\Sigma_1=\Sigma_1(E)$ be the set of bad places of $E$ outside $\Sigma$,
and take $S=\Sigma\sqcup\Sigma_1$.
Then the image $W_p:=\Im(\kappa_p)$ of the local Kummer map has the following properties:
\begin{itemize}
\item
for $p\in\Sigma$, $W_p$ depends only on $\fX$;
\item
for $p\in\Sigma_1$, $W_p=\kappa_p(E[2])$.
\end{itemize}

The ``quadratic form'' $\phi:V\to\BF_2$ is the sum of local
``quadratic forms'' $\phi_p:V_p\to\BF_2$ for $p\in S$,
where each $\phi_p$ is associated to $e_p:V_p\times V_p\to\BF_2$ and satisfies
$\phi_p(W_p)=0$.
The definition of $\phi_p$ is invariant under quadratic twists:
suppose $\CE$ is the quadratic twists of $y^2=x(x-e_1)(x-e_2)$
and $E\in\fX$ is $y^2=x(x-e_1m)(x-e_2m)$,
then $\phi_p$ can be written down explicitly as
$$
\phi_p\begin{pmatrix}
b_1 \\ b_2
\end{pmatrix}:=[e_1e_2b_1,e_1(e_1-e_2)b_2]_p,
$$
independent of $m$.
Here $[\ ,\ ]_p$ is additive Hilbert symbol,
and we used $V_p\cong\big(\BQ_p^\times/(\BQ_p^\times)^2\big)^{\oplus 2}$
according to the basis $\{(e_1m,0),(0,0)\}$ of $E[2]$.

The subspace $K$ of $V$ is constructed as $K:=K_\Sigma\oplus K_{\Sigma_1}$,
where $K_{\Sigma_1}=\bigoplus_{p\in\Sigma_1}H_\unr^1(\BQ_p,E[2])$
is the unramified Galois cohomology.
The $K_\Sigma$ is a subspace of $V_\Sigma$, depending only on $\fX$,
satisfying the following properties:
\begin{itemize}
\item
$V_\Sigma=U_\Sigma\oplus K_\Sigma$.
\item
$K_\Sigma$ is a Lagrangian subspace of $(V_\Sigma,e_\Sigma)$.
\item
$\phi_\Sigma(K_\Sigma)=0$.
\item
$W_\Sigma=(U_\Sigma\cap W_\Sigma)\oplus(K_\Sigma\cap W_\Sigma)$.
\end{itemize}
The existence of $K_\Sigma$ is guaranteed by the following linear algebra result.

\begin{lem}[\cite{SD08}, Lemma 2]
\label{p:linalg-SD-new-2}
Let $(V,e)$ be a finite dimensional symplectic space over $\BF_2$.
Let $U$ and $W$ be two Lagrangian subspaces.
Let $\phi:V\to\BF_2$ be a map which is a ``quadratic form'' associated to $e$,
and such that $\phi(U)=\phi(W)=0$.
Then there exists a Lagrangian subspace $K$ of $V$
such that $V=U\oplus K$, $\phi(K)=0$
and $W=(U\cap W)\oplus(K\cap W)$.
\end{lem}

\subsubsection{Determining the entries of the alternating matrix}

The choice of $K$ makes $W/(W\cap K)=W_{\Sigma_1}\oplus\overline W_\Sigma$
where $\overline W_\Sigma:=W_\Sigma/(W_\Sigma\cap K_\Sigma)\cong U_\Sigma\cap W_\Sigma$
which only depends on $\fX$.
The $t_\fX$ in Theorem
\ref{p:2-descent-redei-mat}
is just the dimension of $\overline W_\Sigma$.
Fix basis $\{P_1,P_2\}$ of $E[2]$ uniformly for $E\in\fX$,
then
$$
W/(W\cap K)
=W_{\Sigma_1}\oplus\overline W_{\Sigma}
=W_{\Sigma_1}^{(1)}\oplus W_{\Sigma_1}^{(2)}\oplus\overline W_{\Sigma}
$$
where $W_{\Sigma_1}^{(i)}:=\bigoplus_{p\in\Sigma_1}\kappa_p(\langle P_i\rangle)$.
The matrix $B=(B_{ij})$ in Theorem \ref{p:2-descent-redei-mat}
is just the Gram matrix $B$ of $\theta$ under this decomposition;
note that the dimensions of $W_{\Sigma_1}^{(1)}$,
$W_{\Sigma_1}^{(2)}$ and $\overline W_{\Sigma}$ are $k$, $k$ and $t_\fX$,
respectively, where $k=\#\Sigma_1(E)$.

It's easy to write down the correspondence of elements of $\Sel_2(E/\BQ)$
and elements of $\ker(B)$ explicitly.
Suppose $x=(x^{(1)};x^{(2)})\in\Sel_2(E/\BQ)\subset H^1(G_{S,\BQ},E[2])
\cong\BQ(S,2)^{\oplus 2}$ under our choice of basis of $E[2]$.
Then the corresponding $y=\overline{\loc_S(x)}\in\ker(B)\subset W/(W\cap K)
=W_{\Sigma_1}^{(1)}\oplus W_{\Sigma_1}^{(2)}\oplus\overline W_{\Sigma}$
is $(y_{\Sigma_1}^{(1)};y_{\Sigma_1}^{(2)};y_\Sigma)$,
where $y_{\Sigma_1}^{(i)}=\loc_{\Sigma_1}(x^{(i)})
=(\ord_\ell(x^{(i)}))_{\ell\in\Sigma_1}\in W_{\Sigma_1}^{(i)}\cong\BF_2^k$
is obvious to describe, while
$y_\Sigma=\overline{\loc_\Sigma(x)}\in\overline W_\Sigma$ is
the image of $x$ under the map
$\Sel_2(E/\BQ)\xrightarrow{\loc_\Sigma}W_\Sigma\twoheadrightarrow
\overline W_\Sigma$, which is not so obvious to describe,
since it depends on the choice of $\overline W_\Sigma$.

From this it's easy to give a description of modified $\pi$-strict
$2$-Selmer groups in terms of the kernel of matrix.
For example, if $\pi:E[2]\to\BF_2$ is $P_1\mapsto 0$, $P_2\mapsto 1$,
then an element $x=(x^{(1)};x^{(2)})\in\Sel_2(E/\BQ)$ is contained in
$\Sel_{2,\pi\sstr}'(E/\BQ)$ if and only if
$y_{\Sigma_1}^{(2)}=\loc_{\Sigma_1}(x^{(2)})=0$.
Hence the corresponding $(y_{\Sigma_1}^{(1)};y_\Sigma)$
is contained in $\ker(B_1')$. For other $\pi$ it's similar.

To prove Theorem \ref{p:2-descent-redei-mat} as well as getting more
precise description of $B_{ij}$, it's better to fix equations of $E$ in $\fX$,
and we have the following result.

\begin{thm}
\label{p:2-descent-redei-mat-2}
Let $E:y^2=x(x-e_1)(x-e_2)$ be an elliptic curve, where $e_1,e_2\in\BZ$.
Let $\Sigma$ be containing $-1$ all primes dividing $2e_1e_2(e_1-e_2)$.
Let $\fX$ be a $\Sigma$-equivalence class of square-free integers,
and $q\in\BQ(\Sigma,2)$ be the $\Sigma$-part of $\fX$.
Then Theorem \ref{p:2-descent-redei-mat}
holds for $\{E^{(m)}\mid m\in\fX\}$, namely,
for any $m=qn\in\fX$, where
$n\in S(\Sigma)$ the prime-to-$\Sigma$ part of $m$,
we have
$\dim_{\BF_2}\Sel_2(E^{(m)}/\BQ)=\corank(B(n))$.
Moreover, the $B_{11},B_{21},B_{22}$ have form
\begin{align*}
B_{11}&=A_{e_1(e_1-e_2)}+B_{11,L}, \\
B_{21}&=A+D_{e_1q}+B_{21,L}, \\
B_{22}&=A_{e_1e_2}+B_{22,L},
\end{align*}
where
$B_{11,L},B_{21,L},B_{22,L}$ are low-rank R\'edei matrices,
namely, linear combinations of $z_cz_d^\RT$.
\end{thm}

In the above theorem,
the medium-rank part of
$B_{11}$, $B_{22}$ and $B_{11}+B_{12}+B_{21}+B_{22}$
are $D_{e_1(e_1-e_2)}$, $D_{e_1e_2}$
and $D_{-e_2(e_1-e_2)}$, respectively.
It turns out that the quadratic twist family
containing $E:y^2=x(x-e_1)(x-e_2)$
is of type (A), (B) or (C)
if and only if there are $0$, $1$ or $2$ perfect squares
among $e_1(e_1-e_2)$, $e_1e_2$
and $-e_2(e_1-e_2)$.
This justifies the assertion on the type in Theorem \ref{p:2-descent-redei-mat}.

In the following we prove Theorem \ref{p:2-descent-redei-mat-2},
fixing a choice $P_1=(e_1m,0)$ and $P_2=(0,0)$
of the basis of $E[2]$ where $E:y^2=x(x-e_1m)(x-e_2m)$
(which is the $E^{(m)}$ in Theorem \ref{p:2-descent-redei-mat-2},
and is the $E$ in Theorem \ref{p:2-descent-redei-mat}).
We use the following notation: if $T\subset S$,
let $H_T^1:=H^1(G_{T,\BQ},E[2])$
and let $U_T$ be the image of $\loc_T:H_T^1\hookrightarrow V_T$.

\subsubsection*{The $\theta|_{\overline W_\Sigma\times\overline W_\Sigma}$}

We have $\theta|_{\overline W_\Sigma\times\overline W_\Sigma}=0$.
In fact, if $x\in W_\Sigma\subset W$,
then there are unique $x_1\in U$, $x_2\in K$,
$x_1'\in U_\Sigma\cap W_\Sigma$ and $x_2'\in K_\Sigma\cap W_\Sigma$
such that $x=x_1+x_2=x_1'+x_2'$.
Let $x_1^\glob\in H_\Sigma^1\subset H_S^1$ be the unique element
such that $x_1'=\loc_\Sigma(x_1^\glob)$.
Then it's easy to see that $x_1=\loc_S(x_1^\glob)
=x_1'+\loc_{\Sigma_1}(x_1^\glob)$
and $x_2=x_2'+\loc_{\Sigma_1}(x_1^\glob)$,
both of them are contained in $W_\Sigma\oplus K_{\Sigma_1}$,
which is a Lagrangian subspace of $V$,
since $W_\Sigma$ and $K_{\Sigma_1}$
are Lagrangian subspaces of $V_\Sigma$ and $V_{\Sigma_1}$, respectively.
In particular,
if $x,y\in W_\Sigma$, then
$\theta(x,y)=e(x_1,y_2)=0$.

\subsubsection*{The $\theta|_{\overline W_\Sigma\times W_{\Sigma_1}}$}

For $\ell\in\Sigma_1$, the local Kummer map $\kappa_\ell$ under our chosen basis
$$
\kappa_\ell:\BF_2\cdot\{P_1,P_2\}=E[2]\hookrightarrow V_\ell\cong\big(\BQ_\ell^\times/(\BQ_\ell^\times)^2\big)^{\oplus 2}
$$
is represented by the matrix
$\left(\begin{smallmatrix}
e_{11}m & e_{12} \\ e_{21} & e_{22}m
\end{smallmatrix}\right)$
where
$\left(\begin{smallmatrix}
e_{11} & e_{12} \\ e_{21} & e_{22}
\end{smallmatrix}\right)
:=\left(\begin{smallmatrix}
e_1 & e_1e_2 \\ e_1(e_1-e_2) & -e_1
\end{smallmatrix}\right)$.

The Weil pairing $e_E:E[2]\times E[2]\to\BF_2$ induces a bilinear map
$d:H_\Sigma^1\times E[2]\to\BQ(\Sigma,2)$ independent of $m$,
such that for any $m\in\fX$, any $\ell\in\Sigma_1$,
any $x\in H_\Sigma^1$ and any $P\in E[2]$, we have
$$
e_\ell\big(\loc_\ell(x),\kappa_\ell(P)\big)=\lrdd{d(x,P)}{\ell}.
$$
In fact, let $H=(h_{ij})=\left(\begin{smallmatrix}
0 & 1 \\ 1 & 0
\end{smallmatrix}\right)$ be the matrix representing $e_E$
under the basis of $E[2]$,
if $x=(b_1,b_2)^{\mathrm T}\in\BQ(\Sigma,2)^{\oplus 2}\cong H_\Sigma^1$
and $P=P_t$, $t=1,2$,
then by the formula of the cup product $e_\ell$ and the local Kummer map
$\kappa_\ell$,
it's easy to see that
$d(x,P)=\sum_ih_{it}b_i$, here the group operation on $\BQ(\Sigma,2)$
is written additively.
Namely, the $d$ is represented by the same matrix as the Weil pairing $e_E$.

Suppose $0\neq x\in U_\Sigma\cap W_\Sigma$
and $0\neq P\in E[2]$.
Consider $\theta|_{\langle x\rangle\times W_{\Sigma_1}^{(t)}}$.
Say $\ell\in\Sigma_1$ and $y=\kappa_\ell(P)$.
We have $x=x_1'=\loc_\Sigma(x_1^\glob)$,
hence $x_1=\loc_S(x_1^\glob)$
and
$$
\theta(x,y)=e(x_1,y)=e\big(\loc_S(x_1^\glob),\kappa_\ell(P)\big)
=e_\ell\big(\loc_\ell(x_1^\glob),\kappa_\ell(P)\big)
=\lrdd{d(x_1^\glob,P)}{\ell}.
$$
That is to say,
$\theta|_{\langle x\rangle\times W_{\Sigma_1}^{(t)}}$
is represented by the R\'edei matrix
$z_d^{\mathrm T}\in\RRM_{1\times k}$
with $d=d(x_1^\glob,P)\in\BQ(\Sigma,2)$.
When $x$ runs over a basis of $U_\Sigma\cap W_\Sigma$
and $P$ runs over a basis of $E[2]$, we obtain that
$\theta|_{\overline W_\Sigma\times W_{\Sigma_1}}$
is represented by a R\'edei matrix
in $\RRM_{\dim\overline W_\Sigma\times 2k}$,
which is the submatrix $(B_{31},B_{32})$
of $B$ in Theorem \ref{p:2-descent-redei-mat}.

We claim that this R\'edei matrix
is generically full rank.
Equivalently, when $x$ runs over a basis of $U_\Sigma\cap W_\Sigma$,
the $\big(d(x_1^\glob,P_t)\big)_{t=1,2}\in\BQ(\Sigma,2)^{\oplus 2}$
are linearly independent.
By the formula of $d$ and that $H=(h_{ij})$ is invertible,
this is equivalent to that, when $x$ runs over a basis of $U_\Sigma\cap W_\Sigma$,
the $x_1^\glob\in H_\Sigma^1\cong\BQ(\Sigma,2)^{\oplus 2}$
are linearly independent.
But we know that such $x$ satisfies $x=x_1'=\loc_\Sigma(x_1^\glob)$,
so $x_1^\glob$ must be linearly independent.

\subsubsection*{The $\theta|_{W_{\Sigma_1}\times W_{\Sigma_1}}$}

For $\ell\in\Sigma_1$,
define $Z_{\ell,\tr}^1$
to be the subspace $\langle\ell\rangle^{\oplus 2}$ of $(\BQ_\ell^\times/(\BQ_\ell^\times)^2)^{\oplus 2}\cong V_\ell$.
Then clearly $V_\ell=K_\ell\oplus Z_{\ell,\tr}^1$.
Define $\kappa_{\ell,\tr}$ to be
$E[2]\xrightarrow{\kappa_\ell}V_\ell\to Z_{\ell,\tr}^1$
and define $\loc_{\ell,\tr}$ to be
$H_\ell^1\xrightarrow{\loc_\ell}V_\ell\to Z_{\ell,\tr}^1$.
Then it's easy to see that these two maps are isomorphisms.
So we obtain an isomorphism $\lambda_\ell:=\loc_{\ell,\tr}^{-1}
\circ\kappa_{\ell,\tr}:E[2]\xrightarrow\sim H_\ell^1$,
in fact,
under our chosen basis it is just represented by the identity matrix.
For $x=\kappa_\ell(P)\in W_\ell$ write
$x=x_1+x_2$
according to $V=U\oplus K$.
Let $x_1^\glob\in H_S^1$ be such that
$x_1=\loc_S(x_1^\glob)$.
Write $x_1^\glob=x_{1,\Sigma}^\glob
+x_{1,\Sigma_1}^\glob$
according to $H_S^1=H_\Sigma^1\oplus H_{\Sigma_1}^1$.
Write $x_2=x_{2,\Sigma}
+x_{2,\Sigma_1}$
according to $K=K_\Sigma\oplus K_{\Sigma_1}$.
Then it's easy to see that
$$
x_{1,\Sigma_1}^\glob=\lambda_\ell(P)
\qquad\text{and}\qquad
x_{1,\Sigma}^\glob=(\loc_\Sigma|_{H_\Sigma^1})^{-1}
(\loc_\Sigma(\lambda_\ell(P))_1'),
$$
here $\loc_\Sigma(\lambda_\ell(P))=\loc_\Sigma(\lambda_\ell(P))_1'
+\loc_\Sigma(\lambda_\ell(P))_2'$ according to $V_\Sigma=U_\Sigma\oplus K_\Sigma$.

Suppose $x=\kappa_\ell(P)$ and $y=\kappa_{\ell'}(Q)$
with $P,Q\in E[2]$ and $\ell,\ell'\in\Sigma_1$.
Then
$$
\theta(x,y)
=e(x_1,y)
=e_{\ell'}\big(\loc_{\ell'}(x_1^\glob),y\big)
=e_{\ell'}\big(\loc_{\ell'}(x_{1,\Sigma}^\glob),y\big)
+e_{\ell'}\big(\loc_{\ell'}(\lambda_\ell(P)),y\big).
$$
We have $\loc_{\ell'}(x_{1,\Sigma}^\glob)\in K_{\ell'}$,
while $\loc_{\ell'}(\lambda_\ell(P))\in K_{\ell'}$
if $\ell'\neq\ell$,
and $\loc_{\ell'}(\lambda_\ell(P))\in Z_{\ell',\tr}^1$
if $\ell'=\ell$.
Therefore, if we suppose $P=P_s$ and $Q=P_t$ for $s,t\in\{1,2\}$,
then we have
\begin{equation}
\label{e:SD-new2-mat1}
e_{\ell'}\big(\loc_{\ell'}(x_{1,\Sigma}^\glob),y\big)
=\sum_ih_{it}\lrdd{\pi_i(x_{1,\Sigma}^\glob)}{\ell'},
\end{equation}
where $\pi_i:H_\Sigma^1\to\BQ(\Sigma,2)$ is the projection to $i$-th component
according to our chosen basis; similarly,
\begin{equation}
\label{e:SD-new2-mat2}
e_{\ell'}\big(\loc_{\ell'}(\lambda_\ell(P)),y\big)
=\sum_jh_{sj}[\ell,e_{jt}]_{\ell'}
+h_{st}[\ell,m]_{\ell'}
=\begin{cases}
h_{st}\lrdd{\ell}{\ell'},&\text{if }\ell'\neq\ell, \\
\sum_jh_{sj}\lrdd{e_{jt}}{\ell}
+h_{st}\lrdd{-m/\ell}{\ell},&\text{if }\ell'=\ell,
\end{cases}
\end{equation}
recall that $e_{ij}\in\BQ(\Sigma,2)$ describes the local Kummer map $\kappa_\ell$.

We fix $s,t$, let $\ell,\ell'\in\Sigma_1$ varies,
and study the behavior of \eqref{e:SD-new2-mat1}
and \eqref{e:SD-new2-mat2}.
Clearly, the \eqref{e:SD-new2-mat2} is represented by the
R\'edei matrix $D_d+h_{st}(A^{\mathrm T}+D_{-q})\in\RRM_{k\times k}$
with $d=\sum_jh_{sj}e_{jt}\in\BQ(\Sigma,2)$,
here the group operation on $\BQ(\Sigma,2)$ is written additively.

In the following we prove that
the \eqref{e:SD-new2-mat1} is represented by a R\'edei matrix
whose entries are linear combinations of $z_cz_d^\RT$ independent of $n$,
where $c,d\in\BQ(\Sigma,2)$.
First we consider the map $\BQ(\Sigma_1,2)\to H_\Sigma^1$
given by $\ell\mapsto x_{1,\Sigma}^\glob
=(\loc_\Sigma|_{H_\Sigma^1})^{-1}
(\loc_\Sigma(\lambda_\ell(P))_1')$ for $\ell\in\Sigma_1$.
It is the composition of the map
$\BQ(\Sigma_1,2)\to V_\Sigma$,
$\ell\mapsto\loc_\Sigma(\lambda_\ell(P))
=\loc_\Sigma((1,\cdots,1,\ell,1,\cdots,1)^{\mathrm T})$
where the $\ell$ is at the $s$-th row,
and the map $V_\Sigma=U_\Sigma\oplus K_\Sigma\twoheadrightarrow U_\Sigma
\xrightarrow\sim H_\Sigma^1$,
$z\mapsto(\loc_\Sigma|_{H_\Sigma^1})^{-1}(z_1')$.
Clearly the first map is represented by a R\'edei matrix
in $\RRM_{4\#\Sigma\times k}$,
and the second map is represented by a matrix in $M_{2\#\Sigma\times 4\#\Sigma}(\BF_2)$
independent of $\Sigma_1$.
Hence the map $\BQ(\Sigma_1,2)\to H_\Sigma^1$,
$\ell\mapsto x_{1,\Sigma}^\glob$ is given by a R\'edei matrix
in $\RRM_{2\#\Sigma\times k}$, and for $i=1,2$,
the map $\BQ(\Sigma_1,2)\to\BQ(\Sigma,2)$,
$\ell\mapsto\pi_i(x_{1,\Sigma}^\glob)$ is given by a R\'edei matrix
$B_i\in\RRM_{\#\Sigma\times k}$.
From this it's easy to that the bilinear map
$\BQ(\Sigma_1,2)\times\BQ(\Sigma_1,2)\to\BF_2$,
$(\ell,\ell')\mapsto
e_{\ell'}\big(\loc_{\ell'}(x_{1,\Sigma}^\glob),y\big)$ given by
\eqref{e:SD-new2-mat1}
is represented by the R\'edei matrix
$\sum_ih_{it}B_i^{\mathrm T}Z_\Sigma^{\mathrm T}\in\RRM_{k\times k}$
where $Z_\Sigma:=(z_p)_{p\in\Sigma}\in\RRM_{k\times\#\Sigma}$.
It is clearly a linear combination of $z_cz_d^\RT$.

\subsection{The $\pi$-strict $2$-Selmer groups}

Swinnerton-Dyer's alternating matrix approach in \cite{SD08}
is not so good to handle $\pi$-strict $2$-Selmer groups,
but only modified $\pi$-strict $2$-Selmer groups.
In this section we compare them.
To do this, we use the interpretation \eqref{e:2-Selmer} directly.
The result is similar to Theorem \ref{p:2-descent-redei-mat-2},
with a R\'edei matrix which is not alternating,
but with simpler low-rank parts.

\begin{thm}
\label{p:2-descent-redei-mat-0}
Let $E:y^2=x(x-e_1)(x-e_2)$ be an elliptic curve, where $e_1,e_2\in\BZ$.
Fix a choice $P_1=(e_1,0)$ and $P_2=(0,0)$
of the basis of $E[2]$.
Let $\Sigma$ be containing $-1$ all primes dividing $2e_1e_2(e_1-e_2)$.
Let $\fX$ be a $\Sigma$-equivalence class of square-free integers,
and $q\in\BQ(\Sigma,2)$ be the $\Sigma$-part of $\fX$.
Then there exists an integer $a$ and a R\'edei matrix
$B\in\RRM_{(2k+a)\times(2k+2\cdot\#\Sigma)}$ depending only on $E$ and $\fX$,
of form
$$
B=\left(\begin{array}{c|c}
P & Q \\
\hline
R & S
\end{array}\right)
=\left(\begin{array}{cc|cc}
A+D_{e_1q} & D_{e_1e_2} & Z_\Sigma \\
D_{e_1(e_1-e_2)} & A+D_{-e_1q} & & Z_\Sigma \\
\hline
R_1 & R_2 & S_1 & S_2
\end{array}\right),
$$
where $Z_\Sigma:=(z_p)_{p\in\Sigma}\in\RRM_{k\times\#\Sigma}$,
such that
for any $m=qn\in\fX$, where
$n=\ell_1\cdots\ell_k\in S(\Sigma)$ the prime-to-$\Sigma$ part of $m$,
under the map $\Sel_2(E^{(m)}/\BQ)\to(\BF_2^k)^{\oplus 2}
\oplus(\BF_2^{\#\Sigma})^{\oplus 2}$
which maps
$x=(x^{(1)};x^{(2)})$ to $((\ord_{\ell_i}(x^{(1)}))_{i=1}^k;
(\ord_{\ell_i}(x^{(2)}))_{i=1}^k;
(\ord_p(x^{(1)}))_{p\in\Sigma};
(\ord_p(x^{(2)}))_{p\in\Sigma})$,
we have
$\Sel_2(E^{(m)}/\BQ)=\ker(B(n))$.

If $\pi:E[2]\to\BF_2$
is given by $P_1=(e_1,0)\mapsto 0$ and $P_2=(0,0)\mapsto 1$,
then there are $B_1\in\RRM_{(2k+a)\times(k+\#\Sigma)}$
and $B_1'\in\RRM_{(2k+a)\times(k+2\cdot\#\Sigma)}$
such that for any above $m=qn$,
$\Sel_{2,\pi\sstr}(E^{(m)}/\BQ)=\ker(B_1(n))$
and $\Sel_{2,\pi\sstr}'(E^{(m)}/\BQ)=\ker(B_1'(n))$,
where $B_1$ and $B_1'$ are of form
$$
B_1=\left(\begin{array}{c|c}
A+D_{e_1q} & Z_\Sigma \\
D_{e_1(e_1-e_2)} \\
\hline
R_1 & S_1
\end{array}\right),\qquad
B_1'=\left(\begin{array}{c|cc}
A+D_{e_1q} & Z_\Sigma \\
D_{e_1(e_1-e_2)} & & Z_\Sigma \\
\hline
R_1 & S_1 & S_2
\end{array}\right).
$$
For other $\pi$ the results are similar.
\end{thm}

\begin{proof}
Write $\Sigma_1=\{\ell_1,\cdots,\ell_k\}$
and $S=\Sigma\sqcup\Sigma_1$,
then $\Sel_2(E^{(m)}/\BQ)$ consists of elements
$x=(x^{(1)};x^{(2)})$ of $\BQ(S,2)^{\oplus 2}$,
such that $\loc_p(x)\in\Im(\kappa_p)$ for all $p\in S$.

When $p\in\Sigma$,
it's easy to see that the map $\loc_p:\BQ(S,2)\to\BQ_p^\times/(\BQ_p^\times)^2$,
under our choice of basis, is represented by a R\'edei matrix
whose entries are $0,1$ or $z_d^\RT$, $d\in\BQ(\Sigma,2)$.
The map $\big(\BQ_p^\times/(\BQ_p^\times)^2\big)^{\oplus 2}
=H^1(\BQ_p,E^{(m)}[2])\to H^1(\BQ_p,E^{(m)}[2])/\Im(\kappa_p)$
depends only on $E$ and $\fX$, hence is represented by a fixed matrix.
Therefore, the condition
that $\loc_p(x)\in\Im(\kappa_p)$ for all $p\in\Sigma$
can be represented by some $R,S$ in $B$.

For $\ell=\ell_i\in\Sigma_1$, recall that
$\kappa_\ell:\BF_2\cdot\{P_1,P_2\}=E[2]
\cong E^{(m)}[2]
\hookrightarrow \big(\BQ_\ell^\times/(\BQ_\ell^\times)^2\big)^{\oplus 2}$
is represented by the matrix
$\left(\begin{smallmatrix}
e_{11}m & e_{12} \\ e_{21} & e_{22}m
\end{smallmatrix}\right)$
where
$\left(\begin{smallmatrix}
e_{11} & e_{12} \\ e_{21} & e_{22}
\end{smallmatrix}\right)
:=\left(\begin{smallmatrix}
e_1 & e_1e_2 \\ e_1(e_1-e_2) & -e_1
\end{smallmatrix}\right)$.
If we let $a\in\BZ$ be any quadratic non-residue modulo $\ell$,
let $\{(\begin{smallmatrix}
a \\ 0
\end{smallmatrix}),
(\begin{smallmatrix}
0 \\ a
\end{smallmatrix}),
(\begin{smallmatrix}
\ell \\ 0
\end{smallmatrix}),
(\begin{smallmatrix}
0 \\ \ell
\end{smallmatrix})\}$
be a basis of $\big(\BQ_\ell^\times/(\BQ_\ell^\times)^2\big)^{\oplus 2}$,
then it's easy to see that
$$
\Im(\kappa_\ell)=\ker
\begin{pmatrix}
1 & 0 & \lrdd{e_{11}m/\ell}{\ell} & \lrdd{e_{12}}{\ell} \\
0 & 1 & \lrdd{e_{21}}{\ell} & \lrdd{e_{22}m/\ell}{\ell}
\end{pmatrix}.
$$
On the other hand, the map
$\loc_\ell:\BQ(S,2)\to\BQ_\ell^\times/(\BQ_\ell^\times)^2
=\BF_2\cdot\{a,\ell\}$ is represented by
$$
\left(\begin{array}{ccccccc|c}
\lrdd{\ell_1}{\ell_i} & \cdots &
\lrdd{\ell_{i-1}}{\ell_i} & 0 & \lrdd{\ell_{i+1}}{\ell_i} &
\cdots & \lrdd{\ell_k}{\ell_i} & \left(\lrdd{p}{\ell_i}\right)_{p\in\Sigma} \\
0 & \cdots & 0 & 1 & 0 & \cdots & 0 & 0
\end{array}\right).
$$
Now it's easy to see that we can put these matrices
together for all $\ell\in\Sigma_1$, and
after reordering rows, we get R\'edei matrices of the form
we want for $P,Q$ in $B$.

For the $\pi$-strict $2$-Selmer groups
for the above $\pi$,
recall that an element $x=(x^{(1)};x^{(2)})\in\Sel_2(E^{(m)}/\BQ)$
is contained in $\Sel_{2,\pi\sstr}(E^{(m)}/\BQ)$ if and only if
$x^{(2)}=0$,
and is contained in
$\Sel_{2,\pi\sstr}'(E^{(m)}/\BQ)$ if and only if
$\loc_{\Sigma_1}(x^{(2)})=0$.
Hence the form of $B_1$ and $B_1'$ is clear.
\end{proof}

\begin{remark}
\label{p:2-descent-redei-mat-general}
By the same method one can prove a similar result for a
general $2$-adic Galois representation
with full rational $2$-torsion points.
Let $W=(\BQ_2/\BZ_2)^a$ be endowed with a continuous $\BZ_2$-linear
$G_\BQ$-action, unramified outside $\Sigma$,
such that the $G_\BQ$-action on $W[2]$ is trivial.
For each $v\in\Sigma$, fix a local condition at $v$
for the family $\{W^{(m)}\mid m\in\fX\}$ (note that they are isomorphic
as $G_{\BQ_v}$-modules), take unramified local condition at $v$ for $v\notin\Sigma$,
define the $2$-Selmer group $\Sel_2(W^{(m)}/\BQ)$ for $m\in\fX$.
Then there exist $b,d\geq 0$,
a R\'edei matrix $B=(B_{ij})_{\substack{1\leq i\leq a+1\\
1\leq j\leq a+1}}\in\RRM_{(ak+b)\times(ak+d)}$, where
$B_{ij}=D_{e_{ij}}$ for $i,j\leq a$, $i\neq j$,
$B_{ii}=A+D_{e_{ii}q}$ for $i\leq a$, such that
$\dim_{\BF_2}\Sel_2(W^{(m)}/\BQ)=\corank(B(n))$ for all $m\in\fX$.
Here $e_{ij}\in\BQ(\Sigma,2)$ is such that
the $G_\BQ$ action $\rho:G_\BQ\to\Aut(W[4])=\GL_a(\BZ/4\BZ)$ on $W[4]$
is given by $\rho=1+(\chi_{e_{ij}})$
where $\chi_{e_{ij}}:G_\BQ\to 2\BZ/4\BZ\cong\{\pm 1\}$
is the quadratic character associated to $\BQ(\sqrt{e_{ij}})/\BQ$.
The above theorem can be viewed as a special case:
$a=2$, $\left(\begin{smallmatrix}
e_{11} & e_{12} \\
e_{21} & e_{22}
\end{smallmatrix}\right)=\left(\begin{smallmatrix}
e_1 & e_1e_2 \\
e_1(e_1-e_2) & -e_1
\end{smallmatrix}\right)$.
\end{remark}

Now it's clear from linear algebra that we have the following comparison
of these two types of $\pi$-strict $2$-Selmer groups.

\begin{cor}
\label{p:strict-Sel-eq}
Under the setup of Theorem \ref{p:2-descent-redei-mat-0},
suppose $\pi:E[2]\to\BF_2$
is $P_1=(e_1,0)\mapsto 0$ and $P_2=(0,0)\mapsto 1$,
then if $e_1(e_1-e_2)$ is a perfect square and
if $m=qn$ is such that $(z_p(n)\in\BF_2^k)_{p\in\Sigma}$
are linearly independent, then
$\Sel_{2,\pi\sstr}(E^{(m)}/\BQ)=\Sel_{2,\pi\sstr}'(E^{(m)}/\BQ)$.
For other $\pi$ there are similar results.
\end{cor}

By Proposition \ref{p:low-rank-max-rank}
it's expected that for most $n$, the
$(z_p(n)\in\BF_2^k)_{p\in\Sigma}$
are linearly independent,
hence if $e_1(e_1-e_2)$ is a perfect square,
then it's expected that for most $n$, these two types of $\pi$-strict
$2$-Selmer groups are equal, so Theorem \ref{main strict}
should comes from the similar result
for modified $\pi$-strict
$2$-Selmer groups, which can be studied by Theorem
\ref{p:2-descent-redei-mat} using alternating R\'edei matrices.
We can't say much for now when
$e_1(e_1-e_2)$ is not a perfect square; later in \S\ref{s:high-rank-non-square}
we will see that
in this case, these two types of $\pi$-strict
$2$-Selmer groups are expected to be trivial,
corresponding to $t_\pi=-\infty$ case in Theorem \ref{main strict}.

\section{Markov analysis}
\label{s:markov}

Recall that (Theorem \ref{p:2-descent-redei-mat},
\ref{p:2-descent-redei-mat-2}, \ref{p:2-descent-redei-mat-0} and
Corollary \ref{selmer mat unrestricted version})
for a $\Sigma$-equivalence class $\fX$ of quadratic twist family of elliptic curves
over $\BQ$ with full rational $2$-torsion points,
the dimension of $2$-Selmer group of $E\in\fX$
is equal to the corank of $B(E)=B(n)$ where $n\in S(\Sigma)$
is the product of bad primes of $E$ outside $\Sigma$;
similarly for $\pi$-strict $2$-Selmer groups
$\Sel_{2,\pi\sstr}(E/\BQ)$ and $\Sel_{2,\pi\sstr}'(E/\BQ)$.
Later (in Theorem \ref{redei matrix main thm}
and Corollary \ref{strict redei matrix main thm})
we will see that the matrix $\widetilde B$ and $\widetilde B_i'$ in
Corollary \ref{selmer mat unrestricted version}
for essential $2$-Selmer groups
$S(E)$
and modified $\pi$-strict essential $2$-Selmer groups
$S_{\pi\sstr}'(E)$
have better Markov behavior
as the number $k\geq 1$ of the prime factors of $n$ varies.
In this section we will do the Markov analysis on such matrices.

For each $k\geq 1$ and each element $\omega
=\big((a_{ij}),(z_p)_{p\in\Sigma}\big)\in\Omega_k^{\Sigma,*}$
as in \eqref{e:Omega-k-Sigma-*},
we define the following matrices and vectors of coefficients in $\BF_2$
associated to $\omega$.
\begin{itemize}
\item
Define $A'=A'(\omega):=(a_{ij})_{\substack{1\leq i\leq k-1\\
1\leq j\leq k-1}}\in M_{k-1}(\BF_2)$.
\item
For $d\in\BQ(\Sigma,2)$, define $z_d'=z_d'(\omega)
:=(z_1^{(d)},\cdots,z_{k-1}^{(d)})^\RT\in\BF_2^{k-1}$.
The $(z_d')^\RT$ is the transpose of $z_d'$.
Define $D_d':=\diag(z_d')\in M_{k-1}(\BF_2)$
and $A_d':=A'+z_d'(z_d')^\RT\in M_{k-1}(\BF_2)$.
\end{itemize}
In other words, $A'$, $A_d'$ and $z_c'(z_d')^\RT$
are the corresponding $A$, $A_d$ and $z_cz_d^\RT$ with
last row and column removed; $z_d'$ is the corresponding $z_d'$ with
last element removed.
Unlike $A$, the restriction
$\sum_{j=1}^ka_{ij}=0$ does not affect $A'$, hence the $A$, $z_d$, etc.~are
also called \emph{restricted} R\'edei matrices;
the $A'$, $z_d'$, etc.~are
called \emph{unrestricted} R\'edei matrices.
The set of unrestricted R\'edei matrices of size $(a(k-1)+b)\times(c(k-1)+d)$
is denoted by $\RRM'_{(a(k-1)+b)\times(c(k-1)+d)}$.
Similar to restricted R\'edei matrices of size $k\times k$,
the unrestricted R\'edei matrices of size $(k-1)\times(k-1)$
are also divided into $3$ types according to their behavior of ranks,
and also have decompositions by ranks.

When $E$ varies in a $\Sigma$-equivalence class $\fX$,
the product $n\in S(\Sigma)$ of bad primes of $E$ outside $\Sigma$
also varies in a $\Sigma$-equivalence class,
as in Theorem \ref{p:2-descent-redei-mat-2}
and Theorem \ref{p:2-descent-redei-mat-0}.
Here $n_1,n_2\in\BQ^\times/(\BQ^\times)^2$ are called $\Sigma$-equivalent to each
other, if $n_1/n_2\in(\BQ_p^\times)^2$ for all $p\in\Sigma$.
It's clear that to give a $\Sigma$-equivalence class of $S(\Sigma)$
is equivalent to give $\mathbf s=(s_p)_{p\in\Sigma}\in\BF_2^{\#\Sigma}$,
by
\begin{equation}
\label{e:Ss-Sigma}
S^{\mathbf s}(\Sigma):=\left\{n\in S(\Sigma)~\middle|~
\lrdd{p}{n}=s_p\text{ for all }p\in\Sigma\right\}\subset S(\Sigma).
\end{equation}
Correspondingly, define
$$
\Omega_k^{\Sigma,\mathbf s}
:=\left\{
\omega=\big((a_{ij}),(z_p)_{p\in\Sigma}\big)\in\Omega_k^{\Sigma,*}~\middle|~
\sum_{i=1}^kz_i^{(p)}=s_p\text{ for all }p\in\Sigma\right\}
\subset\Omega_k^{\Sigma,*}.
$$
The map $S_k(\Sigma)\to\Omega_k^{\Sigma,*}$, $n\mapsto\omega(n)$
restricts to $S_k^{\mathbf s}(\Sigma)\to\Omega_k^{\Sigma,\mathbf s}$.

In this section we study the distribution of $\corank(B(\omega))$
as $k\geq 1$, $\omega\in\Omega_k^{\Sigma,\mathbf s}$ varies,
for the following unrestricted R\'edei matrix $B$,
which is the abstraction of the matrix $\widetilde B$ in Corollary
\ref{selmer mat unrestricted version}.

\begin{defn}
\label{p:high-rank-alternating-defn}
A ``high-rank alternating R\'edei matrix of level $2$''
is a R\'edei matrix $B=(B_{ij})_{\substack{1\leq i\leq 3\\
1\leq j\leq 3}}\in\RRM'_{(2(k-1)+t)\times(2(k-1)+t)}$ where
\begin{itemize}
\item
$B$ is alternating, namely $B_{ji}=B_{ij}^\RT$
and the diagonal of $B_{ii}$ is $0$,
\item
$B_{11},B_{12},B_{21},B_{22}$ are of size $(k-1)\times(k-1)$,
$B_{13},B_{23}$ are of size $(k-1)\times t$ for some fixed integer $t$ independent of $k$,
and $B_{33}$ is of size $t\times t$,
\item
the entries of $B_{13}$ and $B_{23}$ are $z_d'$, $d\in\BQ(\Sigma,2)$,
and $\left(\begin{smallmatrix}
B_{13} \\ B_{23}
\end{smallmatrix}\right)$
is of generically full rank,
namely, there exists $k\geq 1$ and $\omega\in\Omega_k^{\Sigma,*}$
such that
$\left(\begin{smallmatrix}
B_{13}(\omega) \\ B_{23}(\omega)
\end{smallmatrix}\right)$
is of rank $t$,
\item
$B_{12}$ and $B_{21}$ are high-rank R\'edei matrices;
the three matrices $B_{11}$, $B_{22}$ and $B_{11}+B_{12}+B_{21}+B_{22}$
are all medium-rank or low-rank R\'edei matrices.
\end{itemize}
\end{defn}

Suppose the medium-rank part of $B_{11}$ and $B_{22}$ are $D_{d_1}'$
and $D_{d_2}'$, respectively, $d_1,d_2\in\BQ(\Sigma,2)$.
Then the medium-rank part of $B_{11}+B_{12}+B_{21}+B_{22}$
is $D_{-d_1d_2}'$.
Such $B$ is classified, roughly, according to the number $s$
of the low-rank R\'edei matrices among $B_{11}$, $B_{22}$ and $B_{11}+B_{12}+B_{21}+B_{22}$,
namely, number $s$ of $1$ among $d_1,d_2,-d_1d_2\in\BQ(\Sigma,2)$,
as the following types:
\begin{itemize}
\item[(A)]
$s=0$, $(d_1,d_2)\neq(-1,-1)$;
\item[(A')]
$s=0$, $(d_1,d_2)=(-1,-1)$;
\item[(B)]
$s=1$;
\item[(C)]
$s=2$.
\end{itemize}
Note that the matrix coming from elliptic curves over $\BQ$ in Corollary
\ref{selmer mat unrestricted version}
cannot be of type (A'), since from Theorem \ref{p:2-descent-redei-mat-2},
$d_1=e_1(e_1-e_2)$ and $d_2=e_1e_2$, so $d_1+d_2=e_1^2>0$.
In this paper we only study such $B$ of types (A), (B) or (C),
but let's give a remark on type (A')
compared to type (A): the $\left(\begin{smallmatrix}
B_{11}+B_{12}+B_{21} & B_{11}+B_{22} \\
B_{12}+B_{21}+B_{22} & B_{11}+B_{12}+B_{21}
\end{smallmatrix}\right)$ is a low-rank R\'edei matrix for type (A'),
while it's a medium-rank R\'edei matrix for type (A),
which causes extra difficulty similar to types (B) or (C)
compared to type (A).

We introduce crucial invariants for $B$ of type (B) and (C).
Similar to Definition \ref{parameter defn 1},
if $B$ is of type (B), we may assume that $B_{11}$ is low-rank R\'edei matrix
(namely $d_1=1$ and $d_2\neq 1$);
if $B$ is of type (C), we may assume that $B_{11}$ and $B_{12}$
are low-rank R\'edei matrices
(namely $d_1=d_2=1$).
Define
$$
B_j':=\begin{pmatrix}
B_{1j} & B_{13} \\
B_{2j} & B_{23} \\
B_{3j} & B_{33}
\end{pmatrix}
\qquad\text{and}\qquad
B_j'':=\begin{pmatrix}
B_{jj} & B_{j3} \\
B_{3j} & B_{33}
\end{pmatrix}
\qquad\text{for }1\leq j\leq s,
$$
corresponding to $\widetilde B_j'$ and $\widetilde B_j''$ in
Corollary \ref{selmer mat unrestricted version} and
Definition \ref{parameter defn 1}
(we may also define them for $j=3$ case, similar to that
in Theorem \ref{p:2-descent-redei-mat}),
then we define the parameter $\mathbf t=(t_1,\cdots,t_s)$ of $B$,
which is an unordered tuple of $s$ integers, to be
$$
t_j:=t-\max_{k\geq 1,\omega\in\Omega_k^{\Sigma,\mathbf s}}\rank(B_j''(\omega))
\qquad\text{for }1\leq j\leq s.
$$
The $B_j''$ is a low-rank R\'edei matrix,
whose properties are summarized in \S\ref{s:low-rank},
most importantly, the number $t_j$ is easily computable
by Proposition \ref{p:low-rank-max-rank}.

The main result in this section is the following,
which is related to Theorem \ref{main theorem}.

\begin{thm}
\label{redei matrix main thm}
The random variables
$$
Y_k:\Omega_k^{\Sigma,\mathbf s}\to I:=\BZ_{\geq 0}^{s+1},\qquad
\omega\mapsto\big(\corank(B(\omega)),\corank(B_1'(\omega)),\cdots,\corank(B_s'(\omega))\big)
$$
satisfy that for all $k\geq 1$,
\begin{equation}
\label{e:redei-matrix-transition-alt}
\sum_{\mathbf m_\new}\left|\BP(Y_{k+1}=\mathbf m_\new)
-\sum_{\mathbf m}P_{\mathbf m\to\mathbf m_\new}^\Alt\cdot
\BP(Y_k=\mathbf m)\right|
<C\cdot\alpha^k,
\end{equation}
where $0<\alpha<1$ and $C>0$ are constants independent of
$\mathbf m$, $\mathbf m_\new$ and $k$,
namely, they
form an almost Markov chain, whose limit transition matrix
is $P_{\mathbf m\to\mathbf m_\new}^\Alt
=P_{\mathbf m\to\mathbf m_\new,\mathbf t}^\Alt$
for the model $M_{2k+t,\mathbf t}^\Alt(\BF_2)$
in Appendix \ref{s:model}.
In particular, the random variables $Y_k$ and
$X_k:\Omega_k^{\Sigma,\mathbf s}\to\BZ_{\geq 0}$,
$\omega\mapsto\corank(B(\omega))$ satisfy
\begin{equation}
\label{e:redei-matrix-prob-alt}
\lim_{k\to\infty}\BP(Y_k=\mathbf m)
=P_{t,\mathbf t}^\Alt(\mathbf m)
\qquad\text{and}\qquad
\lim_{k\to\infty}\BP(X_k=m)
=P_{t,\mathbf t}^\Alt(m),
\end{equation}
and moreover, for any $\xi\in\BR$
there exists $C>0$ and $0<\alpha<1$ such that for all $k\geq 1$,
\begin{equation}
\label{e:redei-matrix-sum-prob-alt}
\sum_{m\geq 0}2^{\xi\cdot m}\cdot\left|\BP(X_k=m)
-P_{t,\mathbf t}^\Alt(m)\right|
\leq\sum_{\mathbf m=(m,m_1',\cdots,m_s')\in I}2^{\xi\cdot m}\cdot\left|\BP(Y_k=\mathbf m)
-P_{t,\mathbf t}^\Alt(\mathbf m)\right|
<C\cdot\alpha^k.
\end{equation}
\end{thm}

\begin{proof}
The \eqref{e:redei-matrix-transition-alt} comes from Theorem
\ref{redei matrix trans prob},
the fact that the range of $\mathbf m=(m,m_1',\cdots,m_s')$ is
$0\leq m_j'\leq m\leq 2m_j'-t_j$
for $1\leq j\leq s$ (same as that in Proposition \ref{selmer lower bound}),
and the fact that there are at most $3\cdot 4^s$ choices of
$\mathbf m_\new$ when $\mathbf m$ is fixed
(vice versa).
Now Theorem \ref{p:MC}(ii) implies the \eqref{e:redei-matrix-prob-alt}.
As for \eqref{e:redei-matrix-sum-prob-alt}, we only need to prove $\xi\geq 1$ case, which
comes from Theorem \ref{p:MC-2}(ii),
taking $\nu:I\to\BR_{\geq 1}$ to be
$\mathbf m=(m,m_1',\cdots,m_s')\mapsto 2^{\xi\cdot m}$,
utilizing an upper bound of $\BP(X_k=m)$ in Corollary \ref{redei matrix upper bound}.
\end{proof}

The above result for type (A) is proved in \cite{SD08}
in terms of linear algebra and vector spaces
(i.e.~without choosing a basis and written as matrix form),
also proved in \cite{Smith} in the matrix form.
The advantage of matrix form instead of pure linear algebra
becomes clearer for types (B) and (C):
in these types it's easy to spot the low-rank submatrices $B_j''$
of $B$ (which cause difficulty in the analysis)
when written as matrix form.

Summing up the relevant terms in \eqref{e:redei-matrix-transition-alt},
the same argument allows us to deduce the following result,
which is related to the $t_\pi\neq-\infty$ case of Theorem \ref{main strict}.

\begin{cor}
\label{strict redei matrix main thm}
For $1\leq j\leq s$, the random variables
$Z_{j,k}:\Omega_k^{\Sigma,\mathbf s}\to\BZ_{\geq 0}$,
$\omega\mapsto\corank(B_j'(\omega))$ satisfy the similar results:
for all $k\geq 1$,
\begin{equation}
\label{e:redei-matrix-transition-mat}
\sum_{m'_\new}\left|\BP(Z_{j,k+1}=m'_\new)
-\sum_{m'}P_{m'\to m'_\new}^\Mat\cdot\BP(Z_{j,k}=m')\right|
<C\cdot\alpha^k,
\end{equation}
where $0<\alpha<1$ and $C>0$ are constants independent of
$m'$, $m'_\new$ and $k$,
namely, they
form an almost Markov chain, whose limit transition matrix
is $P_{m'\to m'_\new}^\Mat=P_{m'\to m'_\new,t_j}^\Mat$
for the model $M_{(k-t_j)\times k}(\BF_2)$
in Appendix.
In particular, they satisfy
\begin{equation}
\label{e:redei-matrix-prob-mat}
\lim_{k\to\infty}\BP(Z_{j,k}=m')
=P_{t_j}^\Mat(m'),
\end{equation}
and moreover, for any $\xi\in\BR$
there exists $C>0$ and $0<\alpha<1$ such that for all $k\geq 1$,
\begin{equation}
\label{e:redei-matrix-sum-prob-mat}
\sum_{m'\geq 0}2^{\xi\cdot m'}\cdot\left|\BP(Z_{j,k}=m')
-P_{t_j}^\Mat(m')\right|<C\cdot\alpha^k.
\end{equation}
\end{cor}

To prove Theorem \ref{redei matrix main thm},
we use an isomorphism
$$
\Omega_k^{\Sigma,\mathbf s}\cong\BF_2^{k(k-1)/2}\times(\BF_2^{k-1})^{\#\Sigma},
\qquad
\omega\mapsto\big((a_{ij})_{1\leq i\leq j\leq k-1},(z_p')_{p\in\Sigma}\big).
$$
This induces a natural projection $\Omega_{k+1}^{\Sigma,\mathbf s}\twoheadrightarrow
\Omega_k^{\Sigma,\mathbf s}$.
Under this projection, the behavior of $A'$, $z_d'$ and $D_d'$ are better than $A$, $z_d$ and $D_d$.
Define $\Omega_\infty^{\Sigma,\mathbf s}:=\varprojlim\Omega_k^{\Sigma,\mathbf s}$
and view $\{Y_k\}$ as a sequence of random variables on $\Omega_\infty^{\Sigma,\mathbf s}$.
The computation of the transition probability
$\BP(Y_{k+1}=\mathbf m_\new\mid Y_k=\mathbf m)$ is very
similar to that in Theorem \ref{model main thm},
but more ingredients are needed:
\begin{itemize}
\item
Analysis involving low-rank R\'edei matrix.
\item
How to compute the probability (Lemma \ref{how to compute probability}),
which is also used in \cite{Smith}.
\end{itemize}

\subsection{Change of corank from $k$ to $k+1$}
\label{s:prob-Smith}

The projection $\Omega_{k+1}^{\Sigma,\mathbf s}\twoheadrightarrow
\Omega_k^{\Sigma,\mathbf s}$ has a splitting of probability spaces:
$$
\Omega_{k+1}^{\Sigma,\mathbf s}=\Omega_k^{\Sigma,\mathbf s}\times\Omega_{k,\Delta}^{\Sigma,\mathbf s}
\quad\text{where}\quad
\Omega_{k,\Delta}^{\Sigma,\mathbf s}:=
\left\{\big(u_1\in\BF_2^{k-1},a_{kk}\in\BF_2,
(z_k^{(p)}\in\BF_2)_{p\in\Sigma}\big)\right\}
=\BF_2^{k-1}\times\BF_2\times\BF_2^{\#\Sigma},
$$
here $u_1:=(a_{ik})_{1\leq i\leq k-1}\in\BF_2^{k-1}$.
Denote $u_2:=(a_{ki})_{1\leq i\leq k-1}=u_1+z_k^{(-1)}z_{-1}'\in\BF_2^{k-1}$.
If $\omega_\new$ is an element of $\Omega_{k+1}^{\Sigma,\mathbf s}$, we may write
$\omega_\new=(\omega,\omega_\Delta)$ according to the above decomposition,
and we say that ``$\omega_\new$ comes from $\omega$''.

Let $B=(B_{ij})$ be a
``high-rank alternating R\'edei matrix of level $2$''
in Definition \ref{p:high-rank-alternating-defn}.
Let $\omega$ and $\omega_\new$ be as above, with $\omega$ fixed.
By abuse of notation, write $B=B(\omega)$ and $B_\new=B(\omega_\new)$,
similarly for $B_{ij}$ and $B_{ij,\new}$.
Then
$$
B_{ii,\new}=\begin{pmatrix}
B_{ii} & v_{ii} \\
v_{ii}^\RT & 0
\end{pmatrix}~(i=1,2),
\quad
B_{12,\new}=\begin{pmatrix}
B_{12} & v_{12} \\
v_{21}^\RT & v_7
\end{pmatrix},
\quad
B_{i3,\new}=\begin{pmatrix}
B_{i3} \\
v_{3i}^\RT
\end{pmatrix}~(i=1,2),
$$
and
$$
B_\new\sim\left(\begin{array}{ccc|cc}
B_{11} & B_{12} & B_{13} & v_{11} & v_{12} \\
B_{21} & B_{22} & B_{23} & v_{21} & v_{22} \\
B_{31} & B_{32} & B_{33} & v_{31} & v_{32} \\
\hline
v_{11}^{\mathrm T} & v_{21}^{\mathrm T} & v_{31}^{\mathrm T} & 0 & v_7 \\
v_{12}^{\mathrm T} & v_{22}^{\mathrm T} & v_{32}^{\mathrm T} & v_7 & 0 \\
\end{array}\right),
$$
where $v_{11},v_{22}\in\BF_2^{k-1}$,
$v_{31},v_{32}\in\BF_2^t$
denend only on $(z_k^{(p)})_{p\in\Sigma}$;
$v_{12},v_{21}\in\BF_2^{k-1}$
depend only on $u_1$ and $(z_k^{(p)})_{p\in\Sigma}$;
and $v_7\in\BF_2$
depends only on $u_1$, $a_{kk}$ and $(z_k^{(p)})_{p\in\Sigma}$.
If the precise forms of $B_{ij}$ are given (for example in
Theorem \ref{p:2-descent-redei-mat-2}),
it's easy to write down the explicit formula of the map
$\big(u_1,a_{kk},(z_k^{(p)})_{p\in\Sigma}\big)\mapsto
(v_{11},v_{12},v_{21},v_{22},v_{31},v_{32},v_7)$.
But this is not necessary in our argument.
We only need the following.
\begin{itemize}
\item
The map
\begin{equation}
\label{e:W-omega}
\BF_2^{k-1}\times\BF_2^{\#\Sigma}\to(\BF_2^{k-1})^4\times(\BF_2^t)^2,
\qquad
\big(u_1,(z_k^{(p)})_{p\in\Sigma}\big)\mapsto
(v_{11},v_{12},v_{21},v_{22},v_{31},v_{32})
\end{equation}
is a linear map.
It maps $(u_1,0)$ to $(0,u_1,u_1,0,0,0)$.
\item
If $u_1$ and $(z_k^{(p)})_{p\in\Sigma}$ are fixed, then the map $\BF_2\to\BF_2$,
$a_{kk}\mapsto v_7$ is bijective.
\end{itemize}
This is the ``Markov chain nature'' of $B\to B_\new$
(compare with corresponding part in Appendix).

To compute the transition probability
$\BP(Y_{k+1}=\mathbf m_\new\mid Y_k=\mathbf m)$, an
important simplification comes from low-rank R\'edei matrix.
It turns out that the following event
(i.e.~subset of $\Omega_k^{\Sigma,\mathbf s}$) is important:
\begin{equation}
\label{e:lin-indep-event}
F:=\{\omega\in\Omega_k^{\Sigma,\mathbf s}\mid
(z_p'(\omega)\in\BF_2^{k-1})_{p\in\Sigma}\text{ are linearly independent}\}.
\end{equation}
For example, the event $F$ implies that
$(z_p(\omega)\in\BF_2^k)_{p\in\Sigma}$ are also linearly independent,
which we already seen as a condition of Corollary \ref{p:strict-Sel-eq}.
It's clear that
$$
\BP(F)=\prod_{i=1}^{\#\Sigma}(1-2^{i-k})
\geq 1-\sum_{i=1}^{\#\Sigma}2^{i-k}
=1-2\cdot(2^{\#\Sigma}-1)\cdot 2^{-k}=1-C\cdot 2^{-k}
$$
with constant $C=2\cdot(2^{\#\Sigma}-1)>0$
depending only on $\#\Sigma$.
That is to say, for almost all $\omega$, the
$(z_p'(\omega)\in\BF_2^{k-1})_{p\in\Sigma}$ are linearly independent.
The relation of the event $F$ and low-rank R\'edei matrices
is via Proposition \ref{p:low-rank-max-rank}.
Recall that we have defined $B_j'':=\left(\begin{smallmatrix}
B_{jj} & B_{j3} \\
B_{3j} & B_{33}
\end{smallmatrix}\right)$ for $1\leq j\leq s$,
which are of low-rank.
Let $\rho_j:=\rho(B_j'')$ be the maximal possible rank of $B_j''$.
We assumed that $(B_{31},B_{32})$ is generically full rank,
which means that its maximal possible rank is just $t$.
So by Proposition \ref{p:low-rank-max-rank}, the event $F$ implies that
$$
\rank(B_{31},B_{32})=t\qquad\text{and}\qquad
\rank(B_j'')=\rho_j\text{ for }1\leq j\leq s.
$$
Theorem \ref{redei matrix main thm}
is proved via proving the following result.

\begin{thm}
\label{redei matrix trans prob}
For any elements $\mathbf m=(m,m_1',\cdots,m_s')$
and $\mathbf m_\new=(m_\new,m_{1,\new}',\cdots,m_{s,\new}')$ of
$I:=\BZ_{\geq 0}^{s+1}$,
and any $k\geq 1$, the transition probability
$\BP(Y_{k+1}=\mathbf m_\new\mid Y_k=\mathbf m)$ is zero unless
$m_\new-m\in\{0,\pm 2\}$ and $m_{i,\new}'-m_i'\in\{-2,-1,0,1\}$ for all $i$;
for such $\mathbf m$ and $\mathbf m_\new$ we have
$$
\left|\BP(Y_{k+1}=\mathbf m_\new\mid Y_k=\mathbf m)
-P_{\mathbf m\to\mathbf m_\new}^\Alt\right|\cdot\BP(Y_k=\mathbf m)
<C\cdot\big(\alpha^{m+k}+\BP(Y_k=\mathbf m\text{ and not }F)\big),
$$
where $0<\alpha<1$ and $C>0$ are constants independent of
$\mathbf m$, $\mathbf m_\new$ and $k$.
Note that $P_{\mathbf m\to\mathbf m_\new}^\Alt=0$
if $m_{i,\new}'-m_i'\notin\{0,\pm 1\}$ for some $i$.
\end{thm}

\begin{cor}
\label{redei matrix upper bound}
For any $k\geq 1$, $m\geq 0$, we have
$\BP(X_k\geq m)<C\cdot 2^{-\alpha m^2}$,
where $C>0$ and $\alpha>0$ are constants independent of $m$ and $k$.
\end{cor}

To see why the event $F$ \eqref{e:lin-indep-event} makes simplifications,
note that a priori $\rank(B'_{i,\new})-\rank(B_i')\in\{0,1,2,3\}$,
but under the event $F$, for type (B) and (C),
we must have $(v_{11};v_{31})\in\operatorname{span}(B_1'')$,
in particular, $\rank(B'_{1,\new})-\rank(B_1')\in\{0,1,2\}$;
and for type (C), $(v_{22};v_{32})\in\operatorname{span}(B_2'')$,
in particular, $\rank(B'_{2,\new})-\rank(B_2')\in\{0,1,2\}$.
Therefore, to compute transition probabilities up to suitable error term,
we only need to compute
$\BP(Y_{k+1}\in J\times\prod_{i=1}^sJ_i'\mid Y_k=\mathbf m)$ for all $J,J_1',\cdots,J_s'$,
where $J=\BZ,\{m+2\}$ or $\{m-2\}$,
and for $i=1,\cdots,s$, $J_i'=\BZ,\{m_i'+1\}$ or $\{m_i'-1\}$;
we may also assume the event $F$ happens,
which makes the linear algebra criteria simpler.

Let's look at an example
$\BP\big(\corank(B_\new)=m+2\mid Y_k=\mathbf m\big)$,
corresponding to $J=\{m+2\}$ and $J_i'=\BZ$ for all $i$.
Let $\Omega=\{\omega\in\Omega_k^{\Sigma,\mathbf s}\mid
Y_k(\omega)=\mathbf m\}$.
For each $\omega_\new=(\omega,\omega_\Delta)$ with $\omega\in\Omega$,
let $w_i=\left(\begin{smallmatrix}
v_{1i} \\ v_{2i} \\ v_{3i}
\end{smallmatrix}\right)$ for $i=1,2$.
Then
$$
\BP\big(\corank(B_\new)=m+2\mid Y_k=\mathbf m\big)
=\frac{1}{2}\BP\big(w_1\in\operatorname{span}(B)\text{ and }
w_2\in\operatorname{span}(B)\mid Y_k=\mathbf m\big).
$$
Suppose $\omega\in\Omega$ is fixed.
Let $W=W_\omega$ consists of all possible $w=\left(\begin{smallmatrix}
w_1 \\ w_2
\end{smallmatrix}\right)$ as $\omega_\Delta$ varies,
namely, the image of the map \eqref{e:W-omega}.
Let $V=V_\omega$ be $\operatorname{span}(B)
\oplus\operatorname{span}(B)$.
Let $\widetilde V=\BF_2^{2(2(k-1)+t)}$ which contains $V_\omega$ and $W_\omega$
for all $\omega$.
Then
$$
\BP\big(w_1\in\operatorname{span}(B)\text{ and }
w_2\in\operatorname{span}(B)\mid Y_k=\mathbf m\big)
=\BP\big(w\in V_\omega\mid\omega\in\Omega,w\in W_\omega\big).
$$
In fact, the computation of other transition
probabilities also involves the computation of
the probabilities of forms $\BP\big(w\in V_\omega\mid\omega\in\Omega,w\in W_\omega\big)$.
As another example,
\begin{align*}
&\BP\big(\corank(B_\new)=m-2\mid Y_k=\mathbf m\big) \\
&=\BP\big(w_1\notin\operatorname{span}(B)\text{ and }
w_2\notin\operatorname{span}(B)\text{ and }
w_1+w_2\notin\operatorname{span}(B)\mid Y_k=\mathbf m\big) \\
&=1-\BP\big(w_1\in\operatorname{span}(B)\mid Y_k=\mathbf m\big)
-\BP\big(w_2\in\operatorname{span}(B)\mid Y_k=\mathbf m\big)
-\BP\big(w_1+w_2\in\operatorname{span}(B)\mid Y_k=\mathbf m\big) \\
&\qquad{}+2\cdot\BP\big(w_1\in\operatorname{span}(B)\text{ and }
w_2\in\operatorname{span}(B)\mid Y_k=\mathbf m\big).
\end{align*}
Error terms of the form $O(1)\cdot\BP(\text{not }F\mid
Y_k=\mathbf m)$ may be introduced if we utilize the event $F$. For example, in type (B),
under the event $F$, the $(v_{11};v_{31})\in\operatorname{span}((B_1')^\RT)$ is always true,
hence
\begin{align*}
&\BP\big(\corank(B_{1,\new}')=m_1'+1\mid Y_k=\mathbf m\big) \\
&=\frac{1}{2}\BP\big((v_{11};v_{31})\in\operatorname{span}((B_1')^\RT)\text{ and }w_1\in\operatorname{span}(B_1')\text{ and }(v_{12};v_{32})\in\operatorname{span}((B_1')^\RT)
\mid Y_k=\mathbf m\big) \\
&=\frac{1}{2}\BP\big(w_1\in\operatorname{span}(B_1')\text{ and }(v_{12};v_{32})\in\operatorname{span}((B_1')^\RT)
\mid Y_k=\mathbf m\big)+O(1)\cdot\BP(\text{not }F\mid
Y_k=\mathbf m),
\end{align*}
where $O(1)$ is bounded independent of $k$ and $\mathbf m$.

To compute $\BP\big(w\in V_\omega\mid\omega\in\Omega,w\in W_\omega\big)$,
we have the following simple linear algebra lemma.

\begin{lem}
\label{how to compute probability}
Let $\widetilde V$ be a fixed $\BF_2$-vector space of finite dimensional.
Let $(\Omega,\BP)$ be a discrete probability space,
such that each element $\omega$ of $\Omega$ produces
subspaces $V_\omega$ and $W_\omega$
of $\widetilde V$.
Then
$$
\BP\big(w\in V_\omega\mid\omega\in\Omega,w\in W_\omega\big)
=\sum_{h\in\Hom(\widetilde V,\BF_2)}
\sum_{\substack{\omega\in\Omega\\
h|_{V_\omega}=0\\
h|_{W_\omega}=0}}
\frac{\BP(\omega)}{\#(\widetilde V/V_\omega)}.
$$
\end{lem}

\begin{proof}
Denote by $P$ the left hand side.
Clearly
$$
P=\sum_{\omega\in\Omega}\BP(\omega)\cdot\BP\big(w\in V_\omega\mid w\in W_\omega\big)
=\sum_{\omega\in\Omega}\BP(\omega)\cdot\frac{\#(V_\omega\cap W_\omega)}
{\#W_\omega}.
$$
On the other hand,
for $\omega\in\Omega$,
we have
$$
\frac{\#(V_\omega\cap W_\omega)}
{\#W_\omega}
=\frac{\#V_\omega}
{\#(V_\omega+W_\omega)}
=\frac{\#(\widetilde V/(V_\omega+W_\omega))}
{\#(\widetilde V/V_\omega)}
=\frac{\#\Hom(\widetilde V/(V_\omega+W_\omega),\BF_2)}
{\#(\widetilde V/V_\omega)}.
$$
This yields the desired result.
\end{proof}

This is the key formula which we will use later.
The conditions $h|_{V_\omega}=0$
and $h|_{W_\omega}=0$ may be rewritten as $h$ lies in the left
kernel of a matrix whose columns generate $V_\omega+W_\omega$.
The key point is that,
when we take $\Omega=\{\omega\in\Omega_k^{\Sigma,\mathbf s}\mid
Y_k(\omega)=\mathbf m\}$,
we may decompose $\Hom(\widetilde V,\BF_2)$
into disjoint subsets, such that one of them produces
the ``main term'' of $P$, which has a simple form
independent of $k$, and the remaining
part of them produce the ``error terms'' of $P$, which is
$O(\alpha^k)\cdot\frac{\#\Omega_k^{\Sigma,\mathbf s}}{\#\Omega}$ as $k\to\infty$ for some $0<\alpha<1$.
To give an upper bound of ``error terms'',
we may replace $W_\omega$ by its suitable subspace,
which makes the computation simpler.

As an example, let's back to
$$
\BP\big(\corank(B_\new)=m+2\mid Y_k=\mathbf m\big)
=\frac{1}{2}\sum_{h\in\Hom(\widetilde V,\BF_2)}
\sum_{\substack{\omega\in\Omega\\
h|_{V_\omega}=0\\
h|_{W_\omega}=0}}
\frac{\BP(\omega)}{\#(\widetilde V/V_\omega)}
$$
with $\#(\widetilde V/V_\omega)=2^{2m}$.
Since elements of $\widetilde V$ is written as column vectors
$w=(v_{11};v_{21};v_{31};v_{12};v_{22};v_{32})$,
we may write $h\in\Hom(\widetilde V,\BF_2)$
as a row vector $h=(h_{11},h_{21},h_{31},h_{12},h_{22},h_{32})$,
and the $h(w)$ is just the product of a row vector with a column vector.
Recall the property of the map \eqref{e:W-omega},
which implies that
$\{(0;u_1;0;u_1;0;0)\mid u_1\in\BF_2^{k-1}\}$ is a subspace of $W_\omega$
for all $\omega$, so a necessary condition for $h|_{W_\omega}=0$
is that $h_{12}=h_{21}$.

Consider type (B) as an example,
which is $d_1=1$ and $d_2\neq 1$.
It turns out that
the $\{(h_{12},h_{22},h_{32})=0\}$ gives the main term,
in this case the $h|_{V_\omega+W_\omega}=0$ is equivalent to
$$
(h_{11},h_{31})((B_1')^\RT,B_1''')=0,
$$
where $B_1'''$ is a matrix whose columns
are $(v_{11};v_{31})$ as $(z_k^{(p)})_{p\in\Sigma}$
runs over a basis of $\BF_2^{\#\Sigma}$,
similar to the map \eqref{e:W-omega}.
Note that $(B_1'',B_1''')$ is a submatrix of
$B_1''(\widetilde\omega)$ for certain
$\widetilde\omega\in\Omega_{k+\#\Sigma}^{\Sigma,\mathbf s}$,
so under the event $F$ \eqref{e:lin-indep-event},
we have $\rank(B_1'',B_1''')=\rho_1$,
and similarly,
since $B_1''$ is a submatrix of $B_1'$,
we have $\rank((B_1')^\RT,B_1''')=\rank(B_1')$.
On the other hand, if the event $F$ does not hold, then
$\rank((B_1')^\RT,B_1''')\geq\rank(B_1')$.
Therefore the main term for type (B) is
\begin{equation}
\label{e:type-B-main-term}
\begin{aligned}
&\frac{1}{2}\sum_{\substack{h\in\Hom(\widetilde V,\BF_2)\\
(h_{12},h_{22},h_{32})=0}}
\sum_{\substack{\omega\in\Omega\\
h|_{V_\omega}=0\\
h|_{W_\omega}=0}}
\frac{\BP(\omega)}{\#(\widetilde V/V_\omega)}
=\frac{1}{2}\sum_{\omega\in\Omega}
\frac{\BP(\omega)}{\#(\widetilde V/V_\omega)}\cdot
2^{k-1+t-\rank((B_1')^\RT,B_1''')} \\
&=2^{-1-2m}\left(
\BP(F\mid Y_k=\mathbf m)\cdot 2^{m_1'}
+\BP(\text{not }F\mid Y_k=\mathbf m)\cdot O(1)\cdot 2^{m_1'}
\right) \\
&=2^{m_1'-1-2m}+O(1)\cdot\frac{\#\Omega_k^{\Sigma,\mathbf s}}{\#\Omega}
\cdot 2^{m_1'-2m}\cdot 2^{-k},
\end{aligned}
\end{equation}
where $O(1)$ is bounded independent of $k$ and $\mathbf m$.

The $\{(h_{12},h_{22},h_{32})\neq 0\}$ gives the error term; for the estimation of upper bound,
we replace $W_\omega$ by its subspace $\{(0;u_1;0;u_1;0;0)\mid u_1\in\BF_2^{k-1}\}$,
and we are going to prove that
\begin{equation}
\label{e:type-B-error-term}
\frac{1}{2}\sum_{\substack{h\in\Hom(\widetilde V,\BF_2)\\
(h_{12},h_{22},h_{32})\neq 0}}
\sum_{\substack{\omega\in\Omega\\
h|_{V_\omega}=0\\
h|_{W_\omega}=0}}
\frac{\BP(\omega)}{\#(\widetilde V/V_\omega)}
\leq
C\cdot\frac{\#\Omega_k^{\Sigma,\mathbf s}}{\#\Omega}\cdot 2^{-2m}\cdot(3/4)^k
\end{equation}
for some constant $C>0$ independent of $\mathbf m$ and $k$.

To prove \eqref{e:type-B-error-term},
recall that we must have $h_{12}=h_{21}$, and in this case $h|_{V_\omega}=0$
is equivalent to that
$$
\begin{pmatrix}
h_{11} & h_{12} & h_{31} \\
h_{12} & h_{22} & h_{32} \\
\end{pmatrix}
B=0.
$$
The cases are divided according to the dimension of
$\operatorname{span}\{h_{11},h_{12},h_{22}\}$.

For example, consider the case $\dim(\operatorname{span}\{h_{11},h_{12},h_{22}\})=1$ with
$h_{22}\neq 0$, $h_{11}=h_{12}=0$.
Then the above condition becomes
$$
\begin{pmatrix}
0 & 0 & h_{31} \\
0 & h_{22} & h_{32} \\
\end{pmatrix}
B=0,
$$
which implies that
\begin{align}
\label{e:restriction 1}
h_{22}B_{22}+h_{32}B_{32}&=0
\quad\Leftrightarrow\quad
h_{22}(D_{d_2}'+B_{22,L})=h_{32}B_{32}, \\
\label{e:restriction 2}
h_{22}B_{21}+h_{32}B_{31}&=0
\quad\Leftrightarrow\quad
h_{22}A'=h_{22}(B_{21,M}+B_{21,L})+h_{32}B_{31}, \\
\label{e:restriction 3}
h_{31}(B_{31},B_{32})&=0.
\end{align}
These three equations are three typical types of restrictions.

\begin{lem}
\label{error term ingredient 1}
Let $1\neq d\in\BQ(\Sigma,2)$, $B\in\RRM'_{(k-1)\times(k-1)}$
be a low-rank R\'edei matrix.
Let $(z_p')_{p\in\Sigma}$ varies.
Then the average order of $\ker(D_d'+B)$ is
$$
\frac{1}{\#\{(z_p')_{p\in\Sigma}\}}\sum_{(z_p')_{p\in\Sigma}}
\#\ker(D_d'+B)\leq 2^{\rho(B)}(3/2)^{k-1}.
$$
\end{lem}

\begin{proof}
We have $\#\{(z_p')_{p\in\Sigma}\}=2^{(k-1)\cdot\#\Sigma}$,
and since $d\neq 1$, for each $0\leq m\leq k-1$,
there are $\binom{k-1}{m}2^{(k-1)\cdot(\#\Sigma-1)}$ elements
in $\{(z_p')_{p\in\Sigma}\}$ such that $z_d'$ has exactly $m$
non-zero entries.
For such elements, $\rank(D_d')=m$ and
$\rank(D_d'+B)\geq\rank(D_d')-\rank(B)\geq m-\rho(B)$.
The desired inequality follows easily.
\end{proof}

\begin{lem}
\label{error term ingredient 2}
Let $m\geq 0$ be an integer, $X,Y\in M_{m\times(k-1)}(\BF_2)$ be two fixed matrices, with
$\rank(X)=r$.
Let $\omega\in\Omega_k^{\Sigma,\mathbf s}$ with $(z_p')_{p\in\Sigma}$ fixed
and $(a_{ij})_{1\leq i\leq j\leq k-1}$ varies.
Then
$$
\BP(XA'=Y)\leq 2^{-rk+r(r+1)/2}.
$$
\end{lem}

\begin{proof}
Write $X=(x_{ij})$ and $Y=(y_{ij})$.
For each $i$ let $S_i$ be the set of $j$ such that $x_{ij}\neq 0$.
By column operations on $X$ and $Y$,
we may assume that there exists $\ell_1,\cdots,\ell_r$
such that for each $i,j\in\{1,\cdots,r\}$,
$\ell_i\in S_j$ if and only if $i=j$.
Let $a_{ij}$ ($1\leq i\leq j\leq k-1$)
be fixed except for those with $i\in\{\ell_1,\cdots,\ell_r\}$
or $j\in\{\ell_1,\cdots,\ell_r\}$;
there are $rk-r(r+1)/2$ of them.
They can be solved in at most one way as follows:
\begin{itemize}
\item
$a_{\ell_ij}=y_{ij}+\sum_{k\in S_i\setminus\{\ell_1,\cdots,\ell_r\}}a_{kj}$ for
$i\in\{1,\cdots,r\}$ and
$j\in\{1,\cdots,k-1\}\setminus\{\ell_1,\cdots,\ell_r\}$,
\item
$a_{\ell_i\ell_j}=y_{i\ell_j}+\sum_{k\in S_i\setminus\{\ell_1,\cdots,\ell_r\}}a_{k\ell_j}$ for
$i,j\in\{1,\cdots,r\}$.
\end{itemize}
Now the result is clear.
\end{proof}

\begin{lem}
\label{error term ingredient 3}
Let $B\in\RRM'_{t\times u(k-1)}$ be of generically full rank,
$0\neq x\in\BF_2^t$ be a row vector.
Then there exists $1\neq d\in\BQ(\Sigma,2)$ depends only on $B$ and $x$,
such that for any $\omega=(z_p')_{p\in\Sigma}$
satisfying $x\cdot B(\omega)=0$, we have $z_d'(\omega)=0$.
\end{lem}

\begin{proof}
Write $B=((z_{d_{ij}}')^\RT)_{1\leq i\leq t,1\leq j\leq u}$
with $d_{ij}\in\BQ(\Sigma,2)$ for $1\leq i\leq t,1\leq j\leq u$.
The $B$ is of generically full rank means that
the $t$ elements $(d_{11},\cdots,d_{1u}),\cdots,(d_{t1},\cdots,d_{tu})$
in $\BQ(\Sigma,2)^{\oplus u}$ are linearly independent.
Hence the element $x\cdot(d_{ij})$ in $\BQ(\Sigma,2)^{\oplus u}$
is non-trivial, say its $j$-th component is $d\neq 1$.
Now the result is clear since the $j$-th component of
$x\cdot B(\omega)$ is $z_d'(\omega)$.
\end{proof}

To estimate the upper bound,
\begin{itemize}
\item
the first step is to estimate the number
of possible $(h,(z_p')_{p\in\Sigma})$;
\item
the second step is that for each $(h,(z_p')_{p\in\Sigma})$, estimate the number
of possible $(a_{ij})$.
\end{itemize}

For the first step,
\begin{itemize}
\item
if $h_{31}=0$, apply Lemma \ref{error term ingredient 1}
on \eqref{e:restriction 1}
(note that type (B) implies $d_2\neq 1$), obtain that
there is a constant $C>0$ such that
the number of possible $(h,(z_p')_{p\in\Sigma})$ is
$\leq C\cdot 2^{(k-1)\cdot\#\Sigma}(3/2)^k$;
\item
if $h_{31}\neq 0$, apply Lemma \ref{error term ingredient 3}
on \eqref{e:restriction 3}, obtain that
there is a constant $C>0$ such that
the number of possible $(h,(z_p')_{p\in\Sigma})$ is
$\leq C\cdot 2^k\cdot 2^{(k-1)\cdot(\#\Sigma-1)}$.
\end{itemize}
The sum of these two is
$\leq C\cdot 2^{(k-1)\cdot\#\Sigma}(3/2)^k$.

For the second step, apply Lemma \ref{error term ingredient 2}
on \eqref{e:restriction 2}, obtain that
there is a constant $C>0$ such that
the number of possible $(a_{ij})$ is
$\leq C\cdot 2^{k(k-1)/2}\cdot 2^{-k}$.
Therefore
\begin{align*}
\frac{1}{2}\sum_{\substack{h\in\Hom(\widetilde V,\BF_2)\\
h_{22}\neq 0,h_{11}=h_{12}=0}}
\sum_{\substack{\omega\in\Omega\\
h|_{V_\omega}=0\\
h|_{W_\omega}=0}}
\frac{\BP(\omega)}{\#(\widetilde V/V_\omega)}
&\leq
C\cdot\frac{2^{-2m}}{\#\Omega}
\cdot(2^{(k-1)\cdot\#\Sigma}(3/2)^k)\cdot(2^{k(k-1)/2}\cdot 2^{-k}) \\
&=
C\cdot\frac{\#\Omega_k^{\Sigma,\mathbf s}}{\#\Omega}\cdot 2^{-2m}(3/4)^k.
\end{align*}
The other cases of $h$ are similar, also
using the above three types of restrictions.
Combine them we obtain the error term estimate \eqref{e:type-B-error-term}.

Now we know that the transition probability
$\BP\big(\corank(B_\new)=m+2\mid Y_k=\mathbf m\big)$ for type (B)
is the sum of \eqref{e:type-B-main-term} and \eqref{e:type-B-error-term},
on the other hand, from Appendix we know that for type (B),
$P_{(m,m_1')\to(m+2,*)}^\Alt=2^{m_1'-1-2m}$,
hence from $m\leq 2m_1'-t_1$
it's easy to obtain the inequality in Theorem \ref{redei matrix trans prob}.
The same transition probability for types (A)
and (C) are similar.
For type (A), the $\{h=0\}$ gives the main term, which is
$2^{-1-2m}=P_{m\to m+2}^\Alt$;
for type (C), the $\{h_{12}=0\}$ gives the main term, which is
$2^{m_1'+m_2'-1-2m}+O(1)\cdot\frac{\#\Omega_k^{\Sigma,\mathbf s}}{\#\Omega}
\cdot 2^{m_1'+m_2'-2m}\cdot 2^{-k}$
(note that $P_{(m,m_1',m_2')\to(m+2,*)}^\Alt=2^{m_1'+m_2'-1-2m}$ from Appendix).
The error term is the same as type (B), namely the right hand side of
\eqref{e:type-B-error-term}.

To prove Theorem \ref{redei matrix trans prob},
there are $3^{s+1}$ transition probabilities
need to be computed (there will be less when considering symmetry
and some obvious linear algebra restrictions),
including $\BP\big(\corank(B_\new)=m+2\mid Y_k=\mathbf m\big)$ we have just
computed.
They can be computed by the same method,
which are straightforward but lengthy.
We omit the details here.

\subsection{Low-rank R\'edei matrix}
\label{s:low-rank}

Recall that we have already defined
a low-rank R\'edei matrix of size $(k-1)\times(k-1)$.
For general sizes, we have the similar definition.

\begin{defn}
An unrestricted R\'edei matrix $B\in\RRM'_{(a(k-1)+b)\times(c(k-1)+d)}$
is called a low-rank R\'edei matrix, if its entries are
linear combinations of $0$, $1$, $z_d'$,
$(z_d')^{\mathrm T}$ and $z_c'(z_d')^{\mathrm T}$,
$c,d\in\BQ(\Sigma,2)$.
\end{defn}

It's clear that if $B$ is of low-rank,
then the maximal possible rank $\rho$ of $B$
$$
\rho=\rho(B):=\max_{k\geq 1,\omega\in\Omega_k^{\Sigma,\mathbf s}}\rank(B(\omega))
$$
is a finite number.

A main result for a low-rank R\'edei matrix is that the probability
of its rank equal to its maximal possible rank tends to $1$ as $k\to\infty$.

\begin{prop}
\label{p:low-rank-max-rank}
Let $B$ be a low-rank R\'edei matrix.
Let $\omega\in\Omega_k^{\Sigma,\mathbf s}$ be such that
$(z_p'(\omega)\in\BF_2^{k-1})_{p\in\Sigma}$ are linearly independent.
Then $\rank(B(\omega))=\rho$.
In particular, for all $k$,
\begin{align*}
\BP\big(\rank(B(\omega))=\rho\mid\omega\in\Omega_k^{\Sigma,\mathbf s}\big)
&\geq\BP\big((z_p'(\omega)\in\BF_2^{k-1})_{p\in\Sigma}\text{ are linearly independent}
\mid\omega\in\Omega_k^{\Sigma,\mathbf s}\big) \\
&=\prod_{i=1}^{\#\Sigma}(1-2^{i-k})
\geq 1-\sum_{i=1}^{\#\Sigma}2^{i-k}
=1-2\cdot(2^{\#\Sigma}-1)\cdot 2^{-k}.
\end{align*}
\end{prop}

\begin{proof}
We may write $B=\left(\begin{smallmatrix}
P & Q \\ R & S
\end{smallmatrix}\right)$
with $P=(P_{ij})_{1\leq i\leq s_1,1\leq j\leq s_2}$
where each $P_{ij}\in\RRM'_{(k-1)\times(k-1)}$
is a fixed linear combination of $z_c'(z_d')^{\mathrm T}$,
$Q=(z_{q_{ij}}')_{1\leq i\leq s_1,1\leq j\leq t_2}$
and $R=((z_{r_{ij}}')^{\mathrm T})_{1\leq i\leq t_1,1\leq j\leq s_2}$
where each $q_{ij}$ and $r_{ij}$ are fixed elements of $\BQ(\Sigma,2)$,
and $S=(s_{ij})_{1\leq i\leq t_1,1\leq j\leq t_2}\in M_{t_1\times t_2}(\BF_2)$
is a fixed matrix.
We may impose the following assumptions.

\begin{itemize}
\item
We may assume that $P=(P_{ij})=0$.
In fact, we can add some new rows and columns to $B$
and obtain the following matrix
$$
\widetilde B=\begin{pmatrix}
P & Q & \begin{pmatrix}
Z_\Sigma' \\
& \ddots \\
& & Z_\Sigma'
\end{pmatrix} \\
R & S \\
& & I
\end{pmatrix},
$$
here $Z_\Sigma':=(z_p')_{p\in\Sigma}\in\RRM'_{(k-1)\times\#\Sigma}$,
and $I\in M_{(s_1\cdot\#\Sigma)\times(s_1\cdot\#\Sigma)}(\BF_2)$ is the identity matrix.
It satisfies
$\rank(\widetilde B)=\rank(B)+s_1\cdot\#\Sigma$.
Utilizing newly added columns, by column operation it's easy to see that
we can eliminate all the $P_{ij}$'s in $\widetilde B$.
\item
By row and column operations, we may assume that
$Q$ and $R$ are of ``full rank'' in the sense that
$(q_{ij})_{1\leq i\leq s_1}\in\BQ(\Sigma,2)^{\oplus s_1}$, $j=1,\cdots,t_2$
are linearly independent, as well as
$(r_{ij})_{1\leq j\leq s_2}\in\BQ(\Sigma,2)^{\oplus s_2}$, $i=1,\cdots,t_1$
are linearly independent.
\end{itemize}

By column operation
we may assume that the matrix $(q_{ij})_{1\leq i\leq s_1,1\leq j\leq t_2}$
is ``block lower triangular'' in the sense that we have a
decomposition $[t_2]=\bigsqcup_{\ell=1}^{s_1}J_\ell$,
such that the submatrix $(q_{ij})_{j\in J_\ell}$ is zero for $i<\ell$,
and the elements in the submatrix $(q_{ij})_{j\in J_i}$ are linearly independent
in $\BQ(\Sigma,2)$.
Similarly, by row operation we may assume that the matrix
$(r_{ij})_{1\leq i\leq t_1,1\leq j\leq s_2}$
is ``block upper triangular'' in the sense that we have a
decomposition $[t_1]=\bigsqcup_{\ell=1}^{s_2}I_\ell$,
such that the submatrix $(r_{ij})_{i\in I_\ell}$ is zero for $\ell>j$,
and the elements in the submatrix $(r_{ij})_{i\in I_j}$ are linearly independent
in $\BQ(\Sigma,2)$.

Now it's clear that
$$
\sum_{i=1}^{s_1}\rank\big((z_{q_{ij}}')_{j\in J_i}\big)
+\sum_{j=1}^{s_2}\rank\big(((z_{r_{ij}}')^{\mathrm T})_{i\in I_j}\big)
\leq\rank(B)\leq t_1+t_2.
$$
On the other hand, if
$(z_p'(\omega)\in\BF_2^{k-1})_{p\in\Sigma}$ are linearly independent,
then
$\rank\big((z_{q_{ij}}')_{j\in J_i}\big)=\#J_i$ for all $i=1,\cdots,s_1$,
and $\rank\big(((z_{r_{ij}}')^{\mathrm T})_{i\in I_j}\big)=\#I_j$ for all $j=1,\cdots,s_2$,
hence the left hand side of the above inequality is also equal to
$t_1+t_2$.
This means that the maximal possible rank of $B$
is $\rho=\rho(B)=t_1+t_2$, and the desired result holds.
\end{proof}

From the proof one can see that the $\rho$
can be effectively computed, either by Gaussian elimination,
or by choosing any element $\omega\in\Omega_k^{\Sigma,\mathbf s}$ such that
$(z_p'(\omega)\in\BF_2^{k-1})_{p\in\Sigma}$ are linearly independent,
and compute the rank of $B(\omega)$.

\subsection{R\'edei matrix whose corank is expected to be zero}
\label{s:high-rank-non-square}

In this section we study $j>s$ cases, namely, the cases
that the $B_j'$ is not of low-rank.
In this case the $B_j'$, as a submatrix of $B$ of form in
Definition \ref{p:high-rank-alternating-defn},
produced by Swinnerton-Dyer's argument \cite{SD08}
in \S\ref{s:2-Selmer},
is not so easy to analyze, since it contains non-zero low-rank parts
which are not so easy to control, as in Theorem \ref{p:2-descent-redei-mat-2}.
In this case we use R\'edei matrices
in Theorem \ref{p:2-descent-redei-mat-0} to describe the
essential $\pi$-strict Selmer groups
instead.

\begin{prop}
\label{p:high-rank-non-square}
Let
$$
B=(P,Q)=\begin{pmatrix}
A'+D_{d_1}' & Q_1 \\
D_{d_2}' & Q_2 \\
\vdots & \vdots \\
D_{d_a}' & Q_a
\end{pmatrix}\in\RRM'_{a(k-1)\times((k-1)+d)}
$$
with $a\geq 2$,
such that at least one of
$d_1,\cdots,d_a\in\BQ(\Sigma,2)$ is $\neq 1$,
and $Q=(Q_1;\cdots;Q_a)\in\RRM'_{a(k-1)\times d}$ is of generically full rank,
namely, there exists $k\geq 1$ and $\omega\in\Omega_k^{\Sigma,\mathbf s}$
such that $\rank(Q(\omega))=d$.
Then there exists a constant $C>0$ depending only on $\#\Sigma$ and $d$, such that
for any $m\geq 1$ and $k\geq 1$,
the random variable $X_k:\Omega_k^{\Sigma,\mathbf s}\to\BZ_{\geq 0}$,
$\omega\mapsto\corank(B(\omega))$ satisfies
$$
P_{\geq m}^{(k)}:=\BP(X_k\geq m)\leq
C\cdot k\cdot 2^{-m}(3/4)^k.
$$
In particular, $P_m^{(\infty)}:=\lim_{k\to\infty}\BP(X_k=m)$
is $1$ if $m=0$, is $0$ otherwise.
\end{prop}

\begin{proof}
We may assume that $d_2\neq 1$.
If write $Q=(z_{q_{ij}}')_{1\leq i\leq a,1\leq j\leq d}$, then
we have
$$
B_\new=\left(
\begin{array}{c|c|c}
P & Q & (u_1;0;\cdots;0) \\
\hline
(u_2^\RT;0;\cdots;0) & (z_k^{(q_{ij})}) &
(a_{kk}+z_k^{(d_1)};z_k^{(d_2)};\cdots;z_k^{(d_a)})
\end{array}
\right),
$$
$0\leq\rank(B_\new)-\rank(B)\leq a+1$
and correspondingly, $-a\leq\corank(B_\new)-\corank(B)\leq 1$.
We may assume $Q=(Q_{11},0;Q_{21},Q_{22})$
with $Q_{11}\in\RRM'_{(k-1)\times d_1}$ and
$Q_{22}\in\RRM'_{(a-1)(k-1)\times(d-d_1)}$, such that
$Q_{11}$ and $Q_{22}$ are both of generically full rank.
This implies that $q_{11},\cdots,q_{1d_1}$ are linearly independent,
and $q_{1,d_1+1}=\cdots=q_{1d}=0$.

For simplicity of notation, let $\cdots$ be the condition $\corank(B)=m$,
and let $\Omega:=\{\omega\in\Omega_k^{\Sigma,\mathbf s}\mid
\corank(B(\omega))=m\}$.
We start with two conditional probabilities.
Firstly, for each $\omega\in\Omega$, let $W_\omega:=(*;0;\cdots;0)=\BF_2^{k-1}$, $V_\omega:=\operatorname{span}(B)$,
$\widetilde V:=(*;*;\cdots;*)=\BF_2^{a(k-1)}$, then by Lemma
\ref{how to compute probability}
\begin{align*}
\BP\big((u_1;0;\cdots;0)\in\operatorname{span}(B)\mid\cdots\big)
&=\sum_{h:\widetilde V\to\BF_2}\sum_{\substack{\omega\in\Omega\\
h|_{V_\omega}=0\\h|_{W_\omega}=0}}\frac{\BP(\omega)}{\#(\widetilde V/V_\omega)}
\leq\sum_{\substack{\omega\in\Omega\\
h:\BF_2^{(a-1)(k-1)}\to\BF_2\\h\cdot(D_{d_2}';\cdots;D_{d_a}')=0}}
\frac{1}{\#\Omega\cdot 2^{(a-1)(k-1)-d+m}} \\
&\leq\sum_{\substack{\omega\in\Omega_k^{\Sigma,\mathbf s}\\
h:\BF_2^{k-1}\to\BF_2\\h\cdot D_{d_2}'=0}}\frac{1}{\#\Omega\cdot 2^{k-1-d+m}}
=\frac{\#\Omega_k^{\Sigma,\mathbf s}}{\#\Omega}\cdot(3/4)^{k-1}\cdot 2^{d-m}.
\end{align*}
Secondly, for each $\omega$, let $W_\omega=(*;*;0)=\BF_2^{k-1+d_1}$,
$V_\omega=\operatorname{span}(B^\RT)$,
then by Lemma \ref{how to compute probability}
and Proposition \ref{p:low-rank-max-rank},
\begin{align*}
&\BP\big((u_2;(z_k^{(q_{1j})})_{1\leq j\leq d}^\RT)
\in\operatorname{span}(B^\RT)\mid\cdots\big)
=\frac{1}{\#\Omega}\sum_{\omega\in\Omega}\frac{\#V_\omega}
{\#(V_\omega+W_\omega)}
=\frac{1}{\#\Omega}\sum_{\omega\in\Omega}\frac{2^{k-1+d-m}}
{2^{k-1+d_1+\rank(Q_{22})}} \\
&\leq\frac{1}{\#\Omega}\left(
\sum_{\substack{\omega\in\Omega\text{ such that}\\
Z_\Sigma'\text{ is of full rank}}}2^{-m}
+\sum_{\substack{\omega\in\Omega_k^{\Sigma,\mathbf s}\text{ such that}\\
Z_\Sigma'\text{ is not of full rank}}}2^{d-m}
\right)
\leq 2^{-m}+\frac{\#\Omega_k^{\Sigma,\mathbf s}}{\#\Omega}\cdot
(2^{\#\Sigma}-1)\cdot 2^{1-k+d-m}.
\end{align*}

From these we can obtain that
\begin{align*}
P_{m\to m+1}^{(k)}&=\BP\left(
\begin{array}{l}
(u_1;0;\cdots;0)\in\operatorname{span}(B) \\
\text{and }
(u_2;(z_k^{(q_{1j})})_{1\leq j\leq d}^\RT;a_{kk}+z_k^{(d_1)})
\in\operatorname{span}(B^\RT;(u_1^\RT,0,\cdots,0)) \\
\text{and }
(0;(z_k^{(q_{ij})})_{1\leq j\leq d}^\RT;z_k^{(d_i)})
\in\operatorname{span}(B^\RT;(u_1^\RT,0,\cdots,0))
\text{ for }2\leq i\leq a
\end{array}
\middle|~\cdots\right) \\
&\leq\BP\big((u_1;0;\cdots;0)\in\operatorname{span}(B)\mid\cdots\big)
\leq\frac{\#\Omega_k^{\Sigma,\mathbf s}}{\#\Omega}\cdot(3/4)^{k-1}\cdot 2^{d-m},
\displaybreak[0]\\
P_{m\to m}^{(k)}&=\BP\left(
\begin{array}{l}
(u_1;0;\cdots;0)\notin\operatorname{span}(B) \\
\text{and }
(u_2;(z_k^{(q_{1j})})_{1\leq j\leq d}^\RT;a_{kk}+z_k^{(d_1)})
\in\operatorname{span}(B^\RT;(u_1^\RT,0,\cdots,0)) \\
\text{and }
(0;(z_k^{(q_{ij})})_{1\leq j\leq d}^\RT;z_k^{(d_i)})
\in\operatorname{span}(B^\RT;(u_1^\RT,0,\cdots,0))
\text{ for }2\leq i\leq a
\end{array}
\middle|~\cdots\right) \\
&\qquad{}+\BP\big((u_1;0;\cdots;0)\in\operatorname{span}(B)
\text{ and }\rank(B_\new)=\rank(B)+1\mid\cdots\big) \\
&\leq\BP\big((u_2;(z_k^{(q_{1j})})_{1\leq j\leq d}^\RT)
\in\operatorname{span}(B^\RT)\mid\cdots\big)
+\BP\big((u_1;0;\cdots;0)\in\operatorname{span}(B)\mid\cdots\big) \\
&\leq 2^{-m}+\frac{\#\Omega_k^{\Sigma,\mathbf s}}{\#\Omega}\cdot
(3/4)^{k-1}\cdot 2^{\#\Sigma+d-m}.
\end{align*}
Hence for any $m\geq 1$ and $k\geq 1$,
\begin{align*}
&P_{\geq m}^{(k+1)}
=\sum_{i\geq m}P_i^{(k+1)}
=\sum_{i\geq m-1}P_{i\to i+1}^{(k)}P_i^{(k)}
+\sum_{i\geq m}P_{i\to i}^{(k)}P_i^{(k)}
+\sum_{j\geq 1}\sum_{i\geq m+j}P_{i\to i-j}^{(k)}P_i^{(k)} \\
&\leq(3/4)^{k-1}\sum_{i\geq m-1}2^{d-i}
+\left(\sum_{i\geq m}2^{-i}P_i^{(k)}
+(3/4)^{k-1}\sum_{i\geq m}2^{\#\Sigma+d-i}\right)
+\sum_{i\geq m+1}P_i^{(k)} \\
&\leq(3/4)^{k-1}\cdot(2^{\#\Sigma}+2)\cdot 2^{d+1-m}
+2^{-m}P_{\geq m}^{(k)}+P_{\geq m+1}^{(k)}.
\end{align*}
Note that $P_m^{(1)}=1$ if $m=d$, and $P_m^{(1)}=0$ otherwise.
Denote $C_0:=(3/4)^{-1}\cdot(2^{\#\Sigma}+2)\cdot 2^{d+1}$.
Then by induction on $k$, it's easy to see that, for any $m\geq 1$ and $k\geq 1$,
$$
P_{\geq m}^{(k)}\leq C_0\cdot k\cdot 2^{-m}(3/4)^k\cdot\begin{cases}
4/3,&\text{if }m\geq 2, \\
8/3,&\text{if }m=1.
\end{cases}
$$
This completes the proof.
\end{proof}

Now we look at the example described in Theorem \ref{p:2-descent-redei-mat-0}.
Let $E:y^2=x(x-e_1)(x-e_2)$,
$P_1=(e_1,0)$, $P_2=(0,0)$,
$\pi:E[2]\to\BF_2$ is given by $P_1=(e_1,0)\mapsto 0$ and $P_2=(0,0)\mapsto 1$.
Let $\Sigma$ be containing $-1$ all primes dividing $2e_1e_2(e_1-e_2)$.
Let $\fX$ be a $\Sigma$-equivalence class of square-free integers,
and $q\in\BQ(\Sigma,2)$ be the $\Sigma$-part of $\fX$.
By Theorem \ref{p:2-descent-redei-mat-0},
for $m=qn\in\fX$ the $S_{\pi\sstr}(E^{(m)})$
and $S_{\pi\sstr}'(E^{(m)})$ depend only on
$\omega=\omega(n)\in\Omega_k^{\Sigma,\mathbf s}$,
by abuse of notation, denote them by $S_{\pi\sstr}(E^{(q\omega)})$
and $S_{\pi\sstr}'(E^{(q\omega)})$.
Suppose $e_1(e_1-e_2)$ is not a perfect square.
Then we can apply Proposition \ref{p:high-rank-non-square}
to the R\'edei matrices describing
$S_{\pi\sstr}(E^{(m)})$
and $S_{\pi\sstr}'(E^{(m)})$,
which are $B_1$ and $B_1'$ in Theorem \ref{p:2-descent-redei-mat-0}
with a column removed.
The conclusion is, there exists
a constant $C>0$ depending only on $\#\Sigma$, such that for any $k\geq 1$,
\begin{equation}
\label{e:strict-Sel-upper-bound}
\BP\big(S_{\pi\sstr}(E^{(q\omega)})\neq 0
\mid\omega\in\Omega_k^{\Sigma,\mathbf s}\big)\leq C\cdot k\cdot(3/4)^k,
\end{equation}
the same result goes for $S_{\pi\sstr}'(E^{(q\omega)})$.
For other $\pi$ there are similar results.

\subsection{Change of $\Sigma$ and invariance of the parameter}
\label{s:inv of param}

In this section we prove that,
if $\fX$ is a $\Sigma$-equivalence class of quadratic twists of elliptic curves
over $\BQ$ with full rational $2$-torsion points,
$\Sigma\subset \Sigma'$ and $\fX'\subset \fX$ is a $\Sigma'$-equivalence class,
then the class $\fX'$ has the same parameter $\mathbf t$
as that of $\fX$.
This is reduced to the corresponding result on
``high-rank alternating R\'edei matrix of level $2$''
in Definition \ref{p:high-rank-alternating-defn}.
The Markov analysis is not needed in the proof; we only need to use
the basic properties of low-rank R\'edei matrices.
By induction, we may add new primes to $\Sigma'$ one by one,
so we only need to consider $\Sigma'=\Sigma\sqcup\{q\}$ case.

Let $B=(B_{ij})\in\RRM'_{(2(k-1)+t)\times(2(k-1)+t)}$
be in Definition \ref{p:high-rank-alternating-defn},
associated to the $\Sigma$-equivalence class $S^{\mathbf s}(\Sigma)$ of $S(\Sigma)$
given by $\mathbf s=(s_p)_{p\in\Sigma}\in\BF_2^{\#\Sigma}$.
There are $4$ different $\Sigma'$-equivalence classes
$\fX'$ contained in $S^{\mathbf s}(\Sigma)$:
\begin{itemize}
\item
For $s_q\in\BF_2$, the two classes
$$
\fX'=\left\{
n~\middle|~n\in S(\Sigma')\text{ such that }
\lrdd{p}{n}=s_p\text{ for all }p\in\Sigma'
\right\}.
$$
\item
For $s_q\in\BF_2$, the two classes
$$
\fX'=\left\{
qn~\middle|~n\in S(\Sigma')\text{ such that }
\lrdd{p}{n}=s_p+\lrdd{p}{q}\text{ for all }p\in\Sigma\text{ and }\lrdd{q}{n}=s_q
\right\}.
$$
\end{itemize}
For each of the above $\fX'$, there is an associated
$\widetilde B=(\widetilde B_{ij})$ for the $\Sigma'$-equivalence
class $\{n\}$,
such that
\begin{itemize}
\item
for the first two classes,
$\widetilde B\in\RRM'_{(2(k-1)+t)\times(2(k-1)+t)}$
and $\widetilde B(n)=B(n)$;
\item
for the last two classes,
$\widetilde B\in\RRM'_{(2(k-1)+(t+2))\times(2(k-1)+(t+2))}$
and $\widetilde B(n)\sim B(qn)$,
here we need to fix the order of prime factors of $qn$
according to that of $n$:
if $\ell_1,\cdots,\ell_k$ is the prime factors of $n$,
then we fix the order of prime factors of $qn$
to be $\ell_1,\cdots,\ell_{k-1},q,\ell_k$.
\end{itemize}
For example, we may take the above $B$ and $\widetilde B$
to be the $\widetilde B$ in Corollary \ref{selmer mat unrestricted version}
associated to $\fX$ and $\fX'$, respectively,
namely, for $E\in\fX$, $\dim_{\BF_2}S(E)=\corank(B(n))$
where $n$ is the product of bad primes of $E$
outside $\Sigma$, and for $E\in\fX'$, $\dim_{\BF_2}S(E)=\corank(\widetilde B(n))$
where $n$ is the product of bad primes of $E$
outside $\Sigma'$.

\begin{prop}
\label{invariance of parameter}
For each of the above $\fX'$,
the $\widetilde B$ is also a
``high-rank alternating R\'edei matrix of level $2$'',
and has the same parameter $\mathbf t$ as that of $B$.
\end{prop}

\begin{proof}
For the first two classes, we have $\widetilde B=B$
and there is nothing to prove.
For the last two classes, it's easy to see that,
according to our choice of order of prime factors,
the $\widetilde B$ can be defined block-wise by
$$
\widetilde A'=\begin{pmatrix}
A'+D_q' & z_q' \\
(z_{q^*}')^\RT & \lrdd{q^*}{n}
\end{pmatrix}
\quad\text{and}\quad
\widetilde z_p'=\begin{pmatrix}
z_p' \\
\lrdd{p}{q}
\end{pmatrix}
\quad\text{for }p\in\Sigma.
$$
Here $q^*=(-1)^{(q-1)/2}q$.
We reorder the rows and columns to make
$\widetilde B\in\RRM'_{(2(k-1)+(t+2))\times(2(k-1)+(t+2))}$.

It's easy to see that $\widetilde B$ is still alternating,
and $\widetilde B_{ii}=B_{ii}$ for $i=1,2$.
Consider $\left(\begin{smallmatrix}
\widetilde B_{13} \\ \widetilde B_{23}
\end{smallmatrix}\right)\in\RRM'_{2(k-1)\times(t+2)}$.
Its first $t$ columns are that of $\left(\begin{smallmatrix}
B_{13} \\ B_{23}
\end{smallmatrix}\right)$,
and its last two columns comes from the additional column associated to
$\left(\begin{smallmatrix}
\widetilde B_{1j} \\ \widetilde B_{2j}
\end{smallmatrix}\right)$, $j=1,2$.
Considering the occurrence of $z_q'$, it's easy to see that
$\left(\begin{smallmatrix}
\widetilde B_{13} \\ \widetilde B_{23}
\end{smallmatrix}\right)$ is also of generically full rank.

Now we compare the maximal possible rank $\widetilde\rho_j$
of $\widetilde B_j'':=\left(\begin{smallmatrix}
\widetilde B_{jj} & \widetilde B_{j3} \\
\widetilde B_{3j} & \widetilde B_{33}
\end{smallmatrix}\right)$
to the maximal possible rank $\rho_j$
of $B_j'':=\left(\begin{smallmatrix}
B_{jj} & B_{j3} \\
B_{3j} & B_{33}
\end{smallmatrix}\right)$, $1\leq j\leq s$.
Note that the $t+2$ columns of $\widetilde B_{j3}$ come
from the $t$ columns of $B_{j3}$ and the additional columns associated
to $\widetilde B_{j1}$ and $\widetilde B_{j2}$.
If we put one of that additional column back to $\widetilde B_{jj}$,
then it's easy to see that $\widetilde B_j''(n)$
is equal to $B_j''(qn)$ with an additional row and column containing $z_q'$.
Therefore $\widetilde\rho_j\leq\rho_j+2$.
On the other hand, by Proposition \ref{p:low-rank-max-rank}
it's easy to see that ``$=$'' can be obtained.
Hence the parameter $\widetilde t_j$ of $\widetilde B$
is $\widetilde t_j=(t+2)-\widetilde\rho_j=t-\rho_j=t_j$, the same as that of $B$.
\end{proof}

\begin{remark}
\label{change Sigma remark}
Each of the above $\fX'$
gives maps $\psi_r:\Omega_r^{\Sigma',*}/S_r\to\Omega_{r+t}^{\Sigma,*}/S_{r+t}$
for $r\geq 0$, where $t=0$ for the first two classes,
$t=1$ for the last two classes, such that
$\psi_r(\omega(n))=\omega(n)$ for the first two classes,
$\psi_r(\omega(n))=\omega(qn)$ for the last two classes,
see \S\ref{s:change-set}.
Lemma \ref{p:add-primes} tells us that,
if for each $r$, a subset $B_r$ of $\Omega_r^{\Sigma,*}$ is given
(i.e.~an event),
which is $S_r$-invariant, let $B_r'$ be the preimage of $B_{r+t}$
in $\Omega_r^{\Sigma',*}$ under the map $\psi_r$,
then $\lim_{r\to\infty}\BP(B_r')=\lim_{r\to\infty}\BP(B_r)$
if they exist.
From this and the Markov analysis
(which gives the limit distribution of the corank of $B$)
one may also deduce that the parameter is invariant
under refinement of equivalence classes
(since different parameter gives different limit distribution).
\end{remark}

\subsection{Limit of moments}

For $\xi\in\BZ_{\geq 1}$,
we consider the limit of the $\xi$-th moment of the order of $\ker(B(\omega))$,
$\omega\in\Omega_k^{\Sigma,\mathbf s}$, as $k\to\infty$:
$$
\lim_{k\to\infty}\BE(2^{\xi\cdot X_k})
=\lim_{k\to\infty}\sum_{m\geq 0}2^{\xi\cdot m}\cdot\BP(X_k=m).
$$
The \eqref{e:redei-matrix-sum-prob-alt} in Theorem \ref{redei matrix main thm}
implies that
\begin{equation}
\label{e:redei limit moment equal}
\left|\BE(2^{\xi\cdot X_k})
-\sum_{m\geq 0}2^{\xi\cdot m}\cdot P_{t,\mathbf t}^\Alt(m)\right|
\leq C\cdot\alpha^k,
\text{ in particular }
\lim_{k\to\infty}\BE(2^{\xi\cdot X_k})
=\sum_{m\geq 0}2^{\xi\cdot m}\cdot P_{t,\mathbf t}^\Alt(m),
\end{equation}
which is the $\xi$-th moment of the corresponding matrix model $M_{2k+t,\mathbf t}^\Alt(\BF_2)$.

On the other hand, Heath-Brown's method \cite{HB94} allows us to compute the
limit of average order, we called ``$\omega$-average order''
(and hence the limit of $\xi$-th moment),
by a linear algebra way without Markov chain analysis.
The disadvantage is that it does not produce easily a simple formula
(in \cite{HB93}, \cite{HB94}
the type (A) is considered, although with a lengthy computation).
Instead, it's better to pass to $P_{t,\mathbf t}^\Alt$,
then apply the same method to the matrix model
(which produces $3+\sum_i2^{t_i}$ easily, see appendix \ref{average order matrix model}):
$$
\lim_{k\to\infty}\BE(2^{X_k})
=\lim_{k\to\infty}\sum_{m\geq 0}2^m\cdot\BP(X_k=m)
=\sum_{m\geq 0}2^m\cdot P_{t,\mathbf t}^\Alt(m)=3+\sum_i2^{t_i}.
$$
The advantage of Heath-Brown's method is
that the same method also applies to the natural average order,
and allows us to compare them.

Let $B=(B_{ij})_{\substack{1\leq i\leq a+1\\
1\leq j\leq c+1}}\in\RRM_{(ak+b)\times(ck+d)}$
be a restricted R\'edei matrix.
Consider the high-rank part $B_{ij,H}$ of $B_{ij}$
for $1\leq i\leq a$, $1\leq j\leq c$.
If $(B_{ij,H})\sim(\diag(A,\cdots,A);0)$, then such $B$
is called a ``high-rank R\'edei matrix''.
This coincides with the previous definition when $B$ is of size $k\times k$.
Clearly, matrix in Theorem \ref{p:2-descent-redei-mat}
(or Theorem \ref{p:2-descent-redei-mat-0},
more generally, in Remark \ref{p:2-descent-redei-mat-general})
is of high-rank.
Heath-Brown's method implies that,
for such matrix, the $\omega$-average order of kernel exists.

\begin{prop}
\label{p:high-rank-moment-limit}
Let $B\in\RRM_{(ak+b)\times(ck+d)}$
be a high-rank restricted R\'edei matrix.
For each $k\geq 1$ let $X_k:\Omega_k^{\Sigma,\mathbf s}\to\BZ_{\geq 0}$
be the random variable $\omega\mapsto\corank(B(\omega))$
and $E^{(k)}:=\BE(2^{X_k})$.
Then the $\omega$-average order $E^{(\infty)}:=\lim_{k\to\infty}E^{(k)}$
exists, has an explicit formula \eqref{e:high-rank-moment},
satisfies an upper bound
$$
1\leq E^{(\infty)}<3\cdot 2^{c(c+3)/2+d},
$$
and we have the convergence rate estimation:
let $k_0\geq 1$ be any fixed integer, denote $x_0:=(1-2^{-a})^{k_0}$,
then for any $k\geq k_0$,
\begin{align*}
\left|E^{(k)}-E^{(\infty)}\right|
&<3\cdot 2^{c(c+3)/2+d}\cdot
\left(\frac{\left(1+x_0\right)^{2^a}
-1-x_0^{2^a}}{x_0}+2^a\right)\cdot(1-2^{-a})^k \\
&<3\cdot 2^{c(c+3)/2+d}\cdot
(2^{2^a}+2^a-2)\cdot(1-2^{-a})^k.
\end{align*}
\end{prop}

\begin{proof}
Write $B=(P,Q;R^\RT,S)\in\RRM_{(ak+b)\times(ck+d)}$.
By elementary row and column operations,
we may keep $a,b,c,d$ unchanged and assume
that $P_H=(\diag(A,\cdots,A);0)$.

Let $V=\BF_2^{ck+d}$ and $W=\BF_2^{ak+b}$.
For $B\in M_{(ak+b)\times(ck+d)}(\BF_2)$ and $v\in V$ we have
$$
\frac{1}{\#W}\sum_{w\in W}(-1)^{w^\RT Bv}=\begin{cases}
0,&\text{if }v\notin\ker(B), \\
1,&\text{if }v\in\ker(B),
\end{cases}
\quad\text{hence}\quad
2^{\corank(B)}=\frac{1}{\#W}
\sum_{v\in V,w\in W}(-1)^{w^\RT Bv}.
$$
Therefore
\begin{equation}
\label{e:E2Xk-alg}
\BE(2^{X_k})
=\frac{1}{\#\Omega_k^{\Sigma,\mathbf s}\cdot\#W}
\sum_{v\in V,w\in W}
\sum_{\omega\in\Omega_k^{\Sigma,\mathbf s}}
(-1)^{w^\RT B(\omega)v}.
\end{equation}

When $k\in\BZ_{\geq 1}$, $v\in V=\BF_2^{ck+d}$ and
$w\in W=\BF_2^{ak+b}$ are fixed, define the map
$\psi=\psi_{B,v,w}:\Omega_k^{\Sigma,\mathbf s}\to\BF_2$,
$\omega\mapsto w^\RT B(\omega)v$.
To study the map $\psi_{B,v,w}$,
write $v=(v^{(1)};\cdots;v^{(c)};\widetilde v)$
and $w=(w^{(1)};\cdots;w^{(a)};\widetilde w)$,
where $v^{(1)},\cdots,v^{(c)},
w^{(1)},\cdots,w^{(a)}\in\BF_2^k$,
$\widetilde v\in\BF_2^d$,
and $\widetilde w\in\BF_2^b$
(note that $a\geq c$).
Note that to give $v^{(1)},\cdots,v^{(c)},
w^{(1)},\cdots,w^{(a)}$ is equivalent to give the map
$$
\iota:[k]:=\{1,\cdots,k\}\to\BF_2^{c+a},
\quad
i\mapsto(v_i^{(1)},\cdots,v_i^{(c)},
w_i^{(1)},\cdots,w_i^{(a)})^\RT.
$$
For $\lambda\in\BF_2^{c+a}$, define
$[k]_\lambda:=\iota^{-1}(\{\lambda\})
=\{i\in[k]\mid(v_i^{(1)},\cdots,v_i^{(c)},
w_i^{(1)},\cdots,w_i^{(a)})^\RT=\lambda\}$.
Similarly, for $1\leq\ell\leq c$ and $\lambda\in\BF_2^2$,
define $[k]_\lambda^{(\ell)}:=\{i\in[k]\mid(v_i^{(\ell)},w_i^{(\ell)})^\RT=\lambda\}$.

Let $\Lambda:=\Im(\iota)\subset\BF_2^{c+a}$ be the image of the map $\iota$.
We claim that once $\Lambda,\widetilde v,\widetilde w$
are fixed, the $\psi_{B,v,w}(\omega)$
depends only on $a_{\lambda\lambda'}:=\sum_{i\in[k]_\lambda,j\in[k]_{\lambda'}}a_{ij}$
for all $\lambda\neq\lambda'$ in $\Lambda$,
and $z_\lambda^{(p)}:=\sum_{i\in[k]_\lambda}z_i^{(p)}$
for all $\lambda\in\Lambda$ and $p\in\Sigma$.
More precisely, it is a linear combination
of $1$, $a_{\lambda\lambda'}$, $z_\lambda^{(p)}$
and $z_\lambda^{(p)}z_{\lambda'}^{(q)}$ with coefficients
depending only on $B$.
To prove this, note that the R\'edei matrix $A$ only
occurred in the diagonal of $P$, so $\psi_{B,v,w}(\omega)$
is equal to $\sum_{\ell\in[c]}(w^{(\ell)})^\RT A(\omega)v^{(\ell)}$
plus $\widetilde w^\RT S\widetilde v$
plus some linear combinations (with coefficients
depending only on $B$) of the following:
\begin{itemize}
\item[(i)]
$(w^{(i)})^\RT D_pv^{(j)}$ for some $i\in[a]$, $j\in[c]$, $p\in\Sigma$;
\item[(ii)]
$(w^{(i)})^\RT z_pz_q^\RT v^{(j)}$ for some $i\in[a]$, $j\in[c]$, $p,q\in\Sigma$;
\item[(iii)]
$(w^{(i)})^\RT z_p\widetilde v_j$ for some $i\in[a]$, $j\in[d]$, $p\in\Sigma$;
\item[(iv)]
$\widetilde w_iz_p^\RT v^{(j)}$ for some $i\in[b]$, $j\in[c]$, $p\in\Sigma$.
\end{itemize}
Note that
\begin{align*}
(w^{(i)})^\RT D_pv^{(j)}
&=\sum_{\ell=1}^kw_\ell^{(i)}z_\ell^{(p)}v_\ell^{(j)}
=\sum_{\lambda\in\Lambda}\sum_{\ell\in[k]_\lambda}
w_\ell^{(i)}z_\ell^{(p)}v_\ell^{(j)}
=\sum_{\lambda\in\Lambda}\lambda_{c+i}\lambda_jz_\lambda^{(p)}, \\
(w^{(i)})^\RT z_p
&=\sum_{\ell=1}^kw_\ell^{(i)}z_\ell^{(p)}
=\sum_{\lambda\in\Lambda}\sum_{\ell\in[k]_\lambda}w_\ell^{(i)}z_\ell^{(p)}
=\sum_{\lambda\in\Lambda}\lambda_{c+i}z_\lambda^{(p)}, \\
z_p^\RT v^{(j)}
&=\sum_{\ell=1}^kz_\ell^{(p)}v_\ell^{(j)}
=\sum_{\lambda\in\Lambda}\sum_{\ell\in[k]_\lambda}z_\ell^{(p)}v_\ell^{(j)}
=\sum_{\lambda\in\Lambda}\lambda_jz_\lambda^{(p)}.
\end{align*}
Therefore the above (i)--(iv) all depend only on
$z_\lambda^{(p)}$.
For $\sum_{\ell\in[c]}(w^{(\ell)})^\RT A(\omega)v^{(\ell)}$, we have
\begin{equation}
\label{e:wTAv}
(w^{(\ell)})^\RT A(\omega)v^{(\ell)}
=\sum_{i\in[k]_{01}^{(\ell)},j\in[k]_{10}^{(\ell)}}a_{ij}
+\sum_{i\in[k]_{11}^{(\ell)},j\in[k]_{00}^{(\ell)}}a_{ij}
+\sum_{i\in[k]_{11}^{(\ell)},j\in[k]_{01}^{(\ell)}}z_i^{(-1)}z_j^{(-1)}.
\end{equation}
Note that in each of the three sums in \eqref{e:wTAv},
the ranges of $i$ and $j$
are disjoint.
For each $\lambda\neq\lambda'$ in $\Lambda$,
\begin{itemize}
\item
the $a_{\lambda\lambda'}$ appeared in \eqref{e:wTAv} if and only if
$(\lambda_\ell,\lambda_\ell',\lambda_{c+\ell},\lambda_{c+\ell}')
=(0,1,1,0)$ or $(1,0,1,0)$,
\item
the $a_{\lambda'\lambda}$ appeared in \eqref{e:wTAv} if and only if
$(\lambda_\ell,\lambda_\ell',\lambda_{c+\ell},\lambda_{c+\ell}')
=(1,0,0,1)$ or $(0,1,0,1)$,
\item
the $z_\lambda^{(-1)}z_{\lambda'}^{(-1)}$ appeared in \eqref{e:wTAv} if and only if
$(\lambda_\ell,\lambda_\ell',\lambda_{c+\ell},\lambda_{c+\ell}')
=(1,0,1,1)$ or $(0,1,1,1)$.
\end{itemize}
Replace $a_{\lambda'\lambda}$ with
$a_{\lambda\lambda'}+z_\lambda^{(-1)}z_{\lambda'}^{(-1)}$,
sum them over $1\leq\ell\leq c$, we obtain that
$$
\sum_{\ell\in[c]}(w^{(\ell)})^\RT A(\omega)v^{(\ell)}
=\sum_{\lambda<\lambda'}\left(
a_{\lambda\lambda'}\cdot
\sum_{\ell=1}^c(\lambda_\ell+\lambda_\ell')(\lambda_{c+\ell}+\lambda_{c+\ell}')
+z_\lambda^{(-1)}z_{\lambda'}^{(-1)}\cdot
\sum_{\ell=1}^c(\lambda_\ell+\lambda_\ell')\lambda_{c+\ell}'
\right),
$$
here we fix any complete order $<$ on $\Lambda$.
This proves our claim.

The occurrence of $a_{\lambda\lambda'}$ or not in the expression of
$\psi_{B,v,w}(\omega)$ is important to the properties of $\psi$.
For example, if $a_{\lambda\lambda'}$ occurred
for some $\lambda\neq\lambda'$ in $\Lambda$,
fix $i\neq j$ with $i\in[k]_\lambda$ and $j\in[k]_{\lambda'}$,
utilizing the bijection
$$
\Omega_k^{\Sigma,\mathbf s}\xrightarrow\sim\BF_2^{k(k-1)/2}\times M_{(k-1)\times 1}(\BF_2)^{\#\Sigma},
\qquad
\omega=\big((a_{ij}),(z_p)_{p\in\Sigma}\big)\mapsto
\big((a_{ij})_{1\leq i<j\leq k},(z_p')_{p\in\Sigma}\big)
$$
we let $a_{ij}$ or $a_{ji}$ varies while let others unchanged,
then it's easy to see that
$\sum_{\omega\in\Omega_k^{\Sigma,\mathbf s}}(-1)^{\psi(\omega)}=0$,
so it has no contribution in \eqref{e:E2Xk-alg}.
According to \cite{HB94},
\cite{Kane13}, two elements $\lambda,\lambda'$ of $\BF_2^{c+a}$
are called \emph{linked} if
$\sum_{\ell=1}^c(\lambda_\ell+\lambda_\ell')(\lambda_{c+\ell}+\lambda_{c+\ell}')=1$,
or equivalently, when $\lambda,\lambda'\in\Lambda$,
$a_{\lambda\lambda'}$ occurred in the expression of
$\psi_{B,v,w}(\omega)$.
A subset $\Lambda$ of $\BF_2^{c+a}$ is called an \emph{unlinked subset}
if any elements $\lambda,\lambda'$ of $\Lambda$ are unlinked.

\begin{lem}[\cite{HB94}, Lemma 7, \cite{Kane13}, p.~1274]
\label{p:HB94-lem7}
Let $\Lambda$ be an unlinked subset of $\BF_2^{c+a}$.
Then its projection to the first $2c$ components
is contained in a translate of a Lagrangian of $\BF_2^{2c}$
with respect to the symplectic form
$\left(\begin{smallmatrix}
0 & I \\ I & 0
\end{smallmatrix}\right)$.
In particular, $\#\Lambda\leq 2^a$.
\end{lem}

Now we know that the expression \eqref{e:E2Xk-alg}
only involves those $\psi=\psi_{B,v,w}$ such that the corresponding
$\iota:[k]\to\BF_2^{c+a}$, $\widetilde v$ and $\widetilde w$
satisfy that $\Lambda:=\Im(\iota)$ is an unlinked subset
(which must have cardinality $\leq 2^a$).
We know that for such $\psi$, the $a_{\lambda\lambda'}$ does not occurred in the expression
of $\psi_{B,v,w}(\omega)$ for any $\lambda,\lambda'$,
hence such map $\psi=\psi_{B,v,w}$ factors through the surjective map
\begin{align*}
\Omega_k^{\Sigma,\mathbf s}&\twoheadrightarrow
\Omega_\Lambda^{\Sigma,\mathbf s}
:=\left\{\mathbf z=(z_\lambda^{(p)})_{\lambda\in\Lambda,p\in\Sigma}
\in\BF_2^{\#\Lambda\cdot\#\Sigma}~\middle|~
\sum_{\lambda\in\Lambda}z_\lambda^{(p)}=s_p\text{ for all }p\in\Sigma\right\}, \\
\omega=\big((a_{ij}),(z_i^{(p)})\big)&\mapsto
\left(\sum_{i\in[k]_\lambda}z_i^{(p)}\right)_{\lambda\in\Lambda,p\in\Sigma}.
\end{align*}

For each $\mathbf z\in\Omega_\Lambda^{\Sigma,\mathbf s}$,
its preimage in $\Omega_k^{\Sigma,\mathbf s}$
has size $2^{k(k-1)/2+(k-\#\Lambda)\cdot\#\Sigma}$.
When $\Lambda$ is fixed with $\#\Lambda=m$, the number of maps $\iota:[k]\to\BF_2^{c+a}$
with $\Im(\iota)=\Lambda$ is equal to
the number of surjections $\#\{[k]\twoheadrightarrow[m]\}$
from $[k]$ to $[m]$, which clearly satisfies
$m^k-m(m-1)^k\leq\#\{[k]\twoheadrightarrow[m]\}\leq m^k$.
(In fact $\#\{[k]\twoheadrightarrow[m]\}=\sum_{i=1}^m(-1)^{m-i}\binom{m}{i}i^k$.)
Therefore \eqref{e:E2Xk-alg} implies that
\begin{align*}
\BE(2^{X_k})
&=\frac{1}{\#\Omega_k^{\Sigma,\mathbf s}\cdot 2^{ak+b}}
\sum_{\Lambda\text{ unlinked}}
\sum_{\substack{\iota:[k]\to\BF_2^{c+a}\\
\Im(\iota)=\Lambda}}
\sum_{\widetilde v,\widetilde w}
\sum_{\omega\in\Omega_k^{\Sigma,\mathbf s}}
(-1)^{\psi_{B,v,w}(\omega)} \\
&=\frac{1}{2^{ak+b}}
\sum_{m=1}^{2^a}
2^{(1-m)\cdot\#\Sigma}
\cdot\#\{[k]\twoheadrightarrow[m]\}\cdot
\sum_{\substack{\Lambda\text{ unlinked}\\
\#\Lambda=m}}
\sum_{\widetilde v,\widetilde w}
\sum_{\mathbf z\in\Omega_\Lambda^{\Sigma,\mathbf s}}
(-1)^{\psi_{B,v,w}(\mathbf z)}.
\end{align*}
Its main term $E^{(\infty)}$
comes from the $m=2^a$ term in the above sum
and by replacing $\#\{[k]\twoheadrightarrow[m]\}$
with $m^k$. Namely, the explicit formula of $E^{(\infty)}$ is
\begin{equation}
\label{e:high-rank-moment}
E^{(\infty)}
=2^{(1-2^a)\cdot\#\Sigma-b}
\sum_{\substack{\Lambda\text{ unlinked}\\
\#\Lambda=2^a}}
\sum_{\widetilde v,\widetilde w}
\sum_{\mathbf z\in\Omega_\Lambda^{\Sigma,\mathbf s}}
(-1)^{\psi_{B,v,w}(\mathbf z)}.
\end{equation}
The number of different Lagrangian subspaces of $\BF_2^{2c}$
is $\#\big(\Sp_{2c}(\BF_2)/(\Sp_{2c}(\BF_2)\cap(\begin{smallmatrix}
* & * \\ 0 & *
\end{smallmatrix}))\big)=\prod_{i=1}^c(2^i+1)$.
Therefore an upper bound of $E^{(\infty)}$ is
$$
E^{(\infty)}\leq
2^{(1-2^a)\cdot\#\Sigma-b}
\cdot 2^c\cdot\prod_{i=1}^c(2^i+1)
\cdot 2^{b+d}
\cdot 2^{(2^a-1)\cdot\#\Sigma}
=2^{c(c+3)/2+d}\cdot\prod_{i=1}^c(1+2^{-i})
<3\cdot 2^{c(c+3)/2+d}.
$$
Let $k_0\geq 1$ be any fixed integer such that $k\geq k_0$.
For simplicity of notation, denote
$x:=(1-2^{-a})^k$ and $x_0:=(1-2^{-a})^{k_0}$, so that $x\leq x_0<1$.
Note that for any $1\leq m\leq 2^a-1$, $(1-2^{-a}m)^k\leq(1-2^{-a})^{mk}=x^m$.
So we may get an estimate of error term
\begin{align*}
&\left|\BE(2^{X_k})-E^{(\infty)}\right|
\leq\frac{1}{2^{ak+b}}
\sum_{m=1}^{2^a-1}
2^{(1-m)\cdot\#\Sigma}
\cdot m^k\cdot\binom{2^a}{m}
\cdot 2^c\cdot\prod_{i=1}^c(2^i+1)
\cdot 2^{b+d}
\cdot 2^{(m-1)\cdot\#\Sigma} \\
&\qquad{}+\frac{1}{2^{ak+b}}
\cdot 2^{(1-2^a)\cdot\#\Sigma}
\cdot 2^a(2^a-1)^k
\cdot 2^c\cdot\prod_{i=1}^c(2^i+1)
\cdot 2^{b+d}
\cdot 2^{(2^a-1)\cdot\#\Sigma} \\
&=2^{c(c+3)/2+d}\cdot\prod_{i=1}^c(1+2^{-i})\cdot
\left(
\sum_{m=1}^{2^a-1}
(1-2^{-a}m)^k\cdot\binom{2^a}{m}
+2^ax\right) \\
&<3\cdot 2^{c(c+3)/2+d}\cdot
\left(
\left(1+x\right)^{2^a}
-1-x^{2^a}
+2^ax\right)
\leq 3\cdot 2^{c(c+3)/2+d}\cdot
\left(\frac{\left(1+x_0\right)^{2^a}
-1-x_0^{2^a}}{x_0}+2^a\right)x.
\end{align*}
This completes the proof.
\end{proof}

\section{Natural density and average order}

Recall that for $r\geq 1$, $S_r(\Sigma)$ is
the set of positive square-free integers
prime to $\Sigma$ and has exactly $r$ prime factors,
and $\Omega_r^{\Sigma,*}$ is the corresponding
probability space of Legendre symbols
defined in \eqref{e:Omega-k-Sigma-*}.
For all $n\in S_r(\Sigma)$
fix an order of prime factors of $n$.
Then there is a natural map $\omega:S_r(\Sigma)\to\Omega_r^{\Sigma,*}$.
If $N$ is a positive real number, define
$S_r(N,\Sigma):=\{n\in S_r(\Sigma)\mid n<N\}$,
similarly define $S_{\geq 1}(N,\Sigma)$.

Let $I$ be a countable set. For each $r\geq 1$,
let $B_r:\Omega_r^{\Sigma,*}/S_r\to I$ be a random variable.
For $n\in S_r(\Sigma)$,
we can define $B_r(n):=B_r(\omega(n))$.
If there is no risk of confusion, the $B_r(\omega)$ and $B_r(n)$ are simply denoted by
$B(\omega)$ and $B(n)$.
Namely, the $B=(B_r)_{r\geq 1}$ also gives $\bigsqcup_{r=1}^\infty\Omega_r^{\Sigma,*}/S_r\to I$
and $S(\infty,\Sigma)\to I$.
For each $i\in I$,
there are three densities associated to $B$.
\begin{itemize}
\item
For each $r\geq 1$,
the (algebraic) $\omega$-density
$$
P^{(r)}(i):=\BP(B_r(\omega)=i\mid\omega\in\Omega^{\Sigma,*}_r)
\qquad\text{and}\qquad
P^{(\infty)}(i):=\lim_{r\to\infty}P^{(r)}(i).
$$
\item
For each $r\geq 1$,
the natural $\omega$-density
$$
P_\nat^{(r)}(i):=\lim_{N\to\infty}\BP\big(B_r(n)=i\mid n\in S_r(N,\Sigma)\big)
\qquad\text{and}\qquad
P_\nat^{(\infty)}(i):=\lim_{r\to\infty}P_\nat^{(r)}(i).
$$
\item
The natural density
$$
P_\nat(i):=\lim_{N\to\infty}\BP\big(B(n)=i\mid n\in S_{\geq 1}(N,\Sigma)\big).
$$
\end{itemize}

By the work of Gerth \cite{Ger84},
the algebraic $\omega$-density $P^{(r)}(i)$
is equal to the natural $\omega$-density $P_\nat^{(r)}(i)$.
The main result of this section gives a uniform behavior
of the limit of algebraic $\omega$-density $P^{(\infty)}(i)$
(if exists) related to the natural density $P_\nat(i)$,
which is a key ingredient of the proof of Theorem \ref{main theorem}.

\begin{thm}[$\omega$-density implies natural density]
\label{p:natural-density}
Let $I$ and $B=(B_r)_{r\geq 1}$ be as above.
For each $r\geq 1$, let $I_r\subset I$ be a finite subset.
Suppose that for each $i\in I$, the limit $P^{(\infty)}(i)$ exists.
Then there are absolute constants $C_1>0$ and $c>0$,
such that for any fixed
$\epsilon>0$, the following inequality
holds for sufficiently large $N$ (depending on $\epsilon$):
\begin{multline*}
\sum_{i\in I}
\left|\BP\big(B(n)=i\mid n\in S_{\geq 1}(N,\Sigma)\big)-P^{(\infty)}(i)\right| \\
\leq I_\epsilon(\log\log N)\cdot C_1\cdot\exp(-c(\log\log\log N)^{1/2})
+R_\epsilon(\log\log N),
\end{multline*}
where
$$
I_\epsilon(x):=1+\max_{(1-\epsilon)x\leq r\leq(1+\epsilon)x}\#I_r
$$
and
$$
R_\epsilon(x):=\max_{(1-\epsilon)x\leq r\leq(1+\epsilon)x}
\left(\sum_{i\in I_r}\left|P^{(r)}(i)-P^{(\infty)}(i)\right|
+\sum_{i\notin I_r}\left(P^{(r)}(i)+P^{(\infty)}(i)\right)\right).
$$
In particular, if $\lim_{x\to\infty}I_\epsilon(x)\exp(-c(\log x)^{1/2})=0$
and $\lim_{x\to\infty}R_\epsilon(x)=0$,
then $P_\nat(i)=P^{(\infty)}(i)$ for all $i\in I$.
\end{thm}

This result is due to an earlier work of Smith \cite{Smith2}.
The following is the main ingredients of the proof.
\begin{itemize}
\item
We may enlarge $\Sigma$ such that it contains $-1$
and all primes $<D$ for some fixed real number $D>3$,
which is the case considered in \cite{Smith2}.
This is reviewed in \S\ref{s:change-set}.
\item
By the result of Hardy and Ramanujan, we only need to consider integers $n$
such that the number $r$ of prime factors of $n$ is close to $\log\log N$;
these numbers are of density $1-o(1)$ as $N\to\infty$.
This argument is also used in Kane's method \cite{Kane13}.
This is reviewed in \S\ref{s:HR}.
\item
Smith's box method \cite{Smith2} improves Kane's method \cite{Kane13}
by introducing several additional conditions to integers $n$
besides the number of prime factors close to $\log\log N$,
such that their density is still $1-o(1)$,
and the problem is reduced to study
$\BP\big(B_r(n)=i\mid n\in X\big)$
for subsets $X$ of $S_r(\Sigma)$ satisfying some good condition
(called ``box'') and which contains at least one
such ``good'' integers.
This is reviewed in \S\ref{s:box-method}.
\end{itemize}

The error term in Theorem \ref{p:natural-density} is not sufficient
to deduce the result on the average order (Theorem \ref{main average})
from the distribution of ranks (Theorem \ref{main theorem}).
For this, we go back to the method of Heath-Brown \cite{HB93}, \cite{HB94}
and Kane \cite{Kane13}.
Recall that $r\sim\log\log N$ means $|r-\log\log N|\leq(\log\log N)^{2/3}$.

\begin{thm}[$\omega$-average order implies natural and ``essential'' average order]
\label{p:natural-average-order}
Let $B\in\RRM_{(ak+b)\times(ck+d)}$
be a high-rank restricted R\'edei matrix.
For each $k\geq 1$ let $X_k:\Omega_k^{\Sigma,\mathbf s}\to\BZ_{\geq 0}$,
$\omega\mapsto\corank(B(\omega))$,
and let $E^{(\infty)}:=\lim_{k\to\infty}\BE(2^{X_k})$
be the $\omega$-average order
as in Proposition \ref{p:high-rank-moment-limit}.
Then there exits a constant $C>0$ depending only on $B$,
such that for any $N>30$, the natural average order satisfies
$$
\left|\BE\big(2^{\corank(B(n))}\mid n\in S^{\mathbf s}(N,\Sigma)\big)
-E^{(\infty)}\right|<C\cdot\frac{(\log\log N)^{2^{c+1}}}{(\log N)^{1/2^a}},
$$
and the ``essential'' average order satisfies
$$
\left|\BE\big(2^{\corank(B(n))}\mid n\in S^{\mathbf s}(N,\Sigma),r(n)\sim\log\log N\big)-E^{(\infty)}\right|
<C\cdot\frac{\log\log\log N}{(\log\log N)^{1/3}}.
$$
\end{thm}

\subsection{Natural density and Smith's box method}
\label{s:natural-density}

In this section we prove Theorem \ref{p:natural-density}.
We change the meaning of $r\sim\log\log N$ in this section,
from $|r-\log\log N|\leq(\log\log N)^{2/3}$ to
$|r-\log\log N|\leq\frac14(\log\log N)^{2/3}$.
(Here the $\frac14$ is arbitrary; any positive constant $<1$ should also work.)

\subsubsection{Box method}
\label{s:box-method}

We review box method of Smith \cite{Smith2}.
The main result is Proposition \ref{p:Smi2-6.11a}.

Let $N>30$, $D_1>D>3$, $C_0>1$ be real numbers such that
$\log N>(\log D)^4$.
Define $\Sigma_D$ to be the set
consists of $-1$ and all primes $<D$.
We introduce the notation $r\sim_D\log\log N$
which means that $|r-\log(\frac{\log N}{\log D})|
\leq(\log(\frac{\log N}{\log D}))^{2/3}$.
As in \cite{Smith2}, Definition 5.3,
we only consider $r$ fixed such that $r\sim_D\log\log N$.
For an integer $n\in S_r(N,\Sigma_D)$,
let $\ell_1<\cdots<\ell_r$ be its prime factors ordered by size.
The following additional conditions for $n$ is introduced
(see \cite{Smith2}, Definition 5.3 and Proposition 6.10):
\begin{itemize}
\item
$n$ is called \emph{comfortably spaced above} $D_1$
if $\ell_i>D_1$ ($1\leq i<r$) implies
$\ell_i>2D_1$ and $\ell_{i+1}>2\ell_i$.
\item
$n$ is called \emph{$C_0$-regular} if the range of $\ell_i$
is specified in certain way for all $i\leq r/3$.
\item
$n$ is called \emph{extravagantly spaced} if
$\ell_i$ is larger than $\ell_{i-1}$ in certain way
for some $\sqrt{r}/2<i<r/2$.
\item
$n$ is called \emph{Siegel-free above $D_1$} if it is not contained
in the set $V$ of \cite{Smith2}, Proposition 6.10.
\end{itemize}
The $n$ is called ``good'' if it satisfies all of the above conditions.
Take $D_1=D^{(\log\log N)^{c_{13}}}$
and $C_0=c_{14}\cdot\log\log\log N$
for some positive absolute constants $c_{13},c_{14}$ in \cite{Smith2}.
The first important result is that, the ``good'' integers has density $1-o(1)$
as $N\to\infty$.

\begin{prop}
[\cite{Smith2}, Theorem 5.4 and Proposition 6.10]
\label{p:Smi2-5.4}
Let $W$ be the set of
``good'' integers $n\in S_r(N,\Sigma_D)$.
Then
$$
\BP\big(n\in W\mid n\in S_r(N,\Sigma_D)\big)>1-C_1\cdot
\exp(-c(\log\log\log N)^{1/2})
$$
for some absolute constants $C_1>0,c>0$.
\end{prop}

Smith considered certain subsets $X$ of $S_r(\Sigma_D)$, called ``box''.

\begin{defn}[\cite{Smith2}, Definition 6.8]
A ``box'' $X$ is a subset of $S_r(\Sigma_D)$ given by the following data:
\begin{itemize}
\item
an integer $0\leq k\leq r$,
\item
primes $D<\ell_1<\cdots<\ell_k<D_1$,
\item
real numbers $D_1<t_{k+1}<t_{k+1}'<\cdots<t_r<t_r'$,
where $t_i':=(1+\frac{1}{\exp(i-k)\log D_1})t_i$ for $i\geq k+1$,
\end{itemize}
such that
$$
X:=\left\{\prod_{i=1}^r\ell_i~\middle|~t_i<\ell_i<t_i'
\text{ prime for }i\geq k+1\right\}.
$$
\end{defn}

The second important result is that,
if a ``box'' contains at lease one ``good'' integers,
then the equidistribution of Legendre symbols holds on it.

\begin{prop}
[\cite{Smith2}, Theorem 6.4 and Corollary 6.11]
\label{p:Smi2-6.4}
Let $X\subset S_r(\Sigma_D)$ be a box, such that
it contains an element $<N$ which is ``good''.
Then there are positive absolute constants $c_8,c_{12}$,
such that for any $S_r$-stable subset $B_r$ of $\Omega_r^{\Sigma_D,*}$,
$$
\left|\BP\big(\omega(n)\in B_r\mid
n\in X\big)-\BP(B_r)\right|\leq r^{-c_{12}}+D_1^{-c_8}.
$$
\end{prop}

The third important result is that,
to study the density of a subset of $S_r(\Sigma_D)$,
one only needs to study its density restricted to each good boxes.

\begin{prop}[\cite{Smith2}, Proposition 6.9]
\label{p:Smi2-6.9}
Let $\delta\in\BR$ and $\epsilon>0$.
Let $W\subset S_r(N,\Sigma_D)$ be a subset, whose elements
are all comfortably spaced above $D_1$. Suppose
$$
\BP\big(n\in W\mid n\in S_r(N,\Sigma_D)\big)>1-\epsilon.
$$
Let $V$ be a subset of $S_r(\Sigma_D)$.
Suppose that for any box $X\subset S_r(\Sigma_D)$
satisfying $X\cap W\neq\varnothing$, we have
$$
\left|\BP\big(n\in V\mid n\in X\big)-\delta\right|<\epsilon.
$$
Then
$$
\left|\BP\big(n\in V\mid n\in S_r(N,\Sigma_D)\big)-\delta\right|
\leq C_1\cdot(\epsilon+(\log D_1)^{-1})
$$
for some absolute constant $C_1>0$.
\end{prop}

Combine all the above results, we can prove the equidistribution
of Legendre symbols on $S_r(\Sigma_D)$.

\begin{prop}
\label{p:Smi2-6.11a}
For any $S_r$-stable subset $B_r$ of $\Omega_r^{\Sigma_D,*}$,
$$
\left|\BP\big(\omega(n)\in B_r\mid n\in S_r(N,\Sigma_D)\big)-\BP(B_r)\right|
\leq C_1\cdot\exp(-c(\log\log\log N)^{1/2})
$$
for some absolute constants $C_1>0,c>0$.
\end{prop}

\begin{proof}
In Proposition \ref{p:Smi2-6.9}, take $W$
to be all ``good'' integers $<N$,
take $V$ to be all $n\in S_r(\Sigma_D)$ such that $\omega(n)\in B_r$.
By Proposition \ref{p:Smi2-5.4}, we may take $\epsilon=C_1\cdot
\exp(-c(\log\log\log N)^{1/2})$ and $\delta=\BP(B_r)$.
Then the problem is reduced to Proposition \ref{p:Smi2-6.4}.
\end{proof}

\subsubsection{Change the set of primes}
\label{s:change-set}

Proposition \ref{p:Smi2-6.11a} requires that
$r\sim_D\log\log N$ and the set of primes is $\Sigma_D$.
For a general $\Sigma$, we can fix a real number $D>3$
such that $\Sigma\subset\Sigma_D$.
The main result in this section is Proposition \ref{p:Smi2-6.11a-fake},
which is similar to Proposition \ref{p:Smi2-6.11a} but for the set $\Sigma$
and requires that $r\sim\log\log N$.

First we consider the following settings.
Let $\Sigma\subset\Sigma'$ be finite sets consisting of $-1$
and some primes of $\BQ$.
Let $q=q_1\cdots q_t$ be a positive square-free integer whose prime factors
are in $\Sigma'\setminus\Sigma$, with $t\geq 0$ be its number of prime factors.
For each $r\geq 1$, let $\psi_r:\Omega^{\Sigma',*}_r/S_r\to\Omega^{\Sigma,*}_{r+t}/S_{r+t}$
be the unique map which makes the following diagram commutes:
$$
\xymatrix@C=3em{
~S_r(\Sigma')~\ar@{->>}[d]^-{\omega}\ar@{^(->}[r]^-{n\mapsto qn} & S_{r+t}(\Sigma)\ar@{->>}[d]^-{\omega} \\
\Omega^{\Sigma',*}_r/S_r\ar[r]^-{\psi_r} & \Omega^{\Sigma,*}_{r+t}/S_{r+t}.
}
$$
In fact, the map $\psi_r$ comes from a map
$\widetilde\psi_r:\Omega^{\Sigma',*}_r\to\Omega^{\Sigma,*}_{r+t}$
(not unique) and one of the possible choice is
$$
\big((a_{ij})_{1\leq i<j\leq r},
(z_i^{(p)})_{1\leq i\leq r,p\in\Sigma'}\big)\mapsto\big((b_{ij})_{1\leq i<j\leq r+t},
(w_i^{(p)})_{1\leq i\leq r+t,p\in\Sigma}\big),
$$
where
$$
b_{ij}=\begin{cases}
a_{ij},&\text{if }1\leq i<j\leq r, \\
z_i^{(q_{j-r})},&\text{if }1\leq i\leq r\text{ and }r+1\leq j\leq r+t, \\
\lrdd{q_{j-r}}{q_{i-r}},&\text{if }r+1\leq i<j\leq r+t,
\end{cases}
\quad\text{and}\quad
w_i^{(p)}=\begin{cases}
z_i^{(p)},&\text{if }1\leq i\leq r, \\
\lrdd{p}{q_{i-r}},&\text{if }r+1\leq i\leq r+t.
\end{cases}
$$

Let $B=(B_r)_{r=1}^\infty$ where
$B_r$ is a $S_r$-stable subset of $\Omega^{\Sigma,*}_r$.
Define $B'=(B_r')_{r=1}^\infty$
where $B_r'\subset\Omega^{\Sigma',*}_r$
is the preimage of $B_{r+t}$ under the map
$\psi_r:\Omega^{\Sigma',*}_r/S_r\to\Omega^{\Sigma,*}_{r+t}/S_{r+t}$.
In particular, $B_r'$ is also $S_r$-stable, and if $n\in S_r(\Sigma')$,
then $\omega(n)\in B_r'$ if and only if $\omega(qn)\in B_{r+t}$.

Let $P^{(r)}:=\BP(B_r)$
and $Q^{(r)}:=\BP(B_r')$.
The following result tells us that the sequence
$(Q^{(r)})_{r=1}^\infty$ is very close to $(P^{(r)})_{r=1}^\infty$.

\begin{lem}
\label{p:add-primes}
Under the above settings, for any $r\geq 1$ we have
$$
|Q^{(r)}-P^{(r+t)}|\leq 2^{\#\Sigma}\cdot(2^t-1)\cdot
\frac{\log(r+1)+13/12}{\sqrt{r+1}}.
$$
In particular, $\lim_{r\to\infty}Q^{(r)}=\lim_{r\to\infty}P^{(r)}$
if they exist.
\end{lem}

Note that the right hand side decreases as $r\geq 1$ increases.

\begin{proof}
By induction we only need to prove the case that $\#(\Sigma'\setminus\Sigma)=1$.
There are two cases: either $q=1$ or $q$ is the unique element of
$\Sigma'\setminus\Sigma$.

If $q=1$, then the map $\widetilde\psi_r:\Omega^{\Sigma',*}_r\to\Omega^{\Sigma,*}_r$
is simply
$\big((a_{ij}),(z_p)_{p\in\Sigma'}\big)\mapsto
\big((a_{ij}),(z_p)_{p\in\Sigma}\big)$,
which is surjective, and for any element $\omega$ of $\Omega^{\Sigma,*}_r$,
the size of its preimage in $\Omega^{\Sigma',*}_r$ is $2^r$.
Hence it's clear that $Q^{(r)}=P^{(r)}$.

If $q$ is the unique element of
$\Sigma'\setminus\Sigma$, then the map
$\widetilde\psi_r:\Omega^{\Sigma',*}_r\to\Omega^{\Sigma,*}_{r+1}$
is $\big((a_{ij})_{1\leq i<j\leq r},
(z_i^{(p)})_{1\leq i\leq r,p\in\Sigma'}\big)\mapsto\big((b_{ij})_{1\leq i<j\leq r+1},
(w_i^{(p)})_{1\leq i\leq r+1,p\in\Sigma}\big)$, where
$$
b_{ij}=\begin{cases}
a_{ij},&\text{if }1\leq i<j\leq r, \\
z_i^{(q)},&\text{if }1\leq i\leq r\text{ and }j=r+1,
\end{cases}
\qquad\text{and}\qquad
w_i^{(p)}=\begin{cases}
z_i^{(p)},&\text{if }1\leq i\leq r, \\
\lrdd{p}{q},&\text{if }i=r+1.
\end{cases}
$$
It's clear that this map is injective, and whose image is the set
of $\big((b_{ij}),(w_i^{(p)})\big)$ such that
$w_{r+1}^{(p)}=\lrdd{p}{q}$ for all $p\in\Sigma$.
By this way we view $\Omega^{\Sigma',*}_r$ as a subset of $\Omega^{\Sigma,*}_{r+1}$.

For a fixed $r$, we have $P^{(r+1)}=\#B_{r+1}/\#\Omega^{\Sigma,*}_{r+1}$ as well as
$Q^{(r)}=\#(B_{r+1}\cap\Omega^{\Sigma',*}_r)/\#\Omega^{\Sigma',*}_r$.
Note that $B_{r+1}$ is a union of $S_{r+1}$-orbits.
Note that $\#\Omega^{\Sigma',*}_r=2^{r(r-1)/2+r\cdot(\#\Sigma+1)}$
and $\#\Omega^{\Sigma,*}_{r+1}=2^{r(r+1)/2+(r+1)\cdot\#\Sigma}
=2^{\#\Sigma}\#\Omega^{\Sigma',*}_r$.
In the following we estimate the difference of these two probabilities.

Write $\Omega^{\Sigma,*}_{r+1}=\bigsqcup_{s=0}^{r+1}\Omega^{\Sigma,*}_{r+1,s}$
where $\Omega^{\Sigma,*}_{r+1,s}$ consists of elements
$\big((b_{ij}),(w_i^{(p)})\big)$ of $\Omega^{\Sigma,*}_{r+1}$ such that
$\#\{1\leq i\leq r+1\mid w_i^{(p)}=\lrdd{p}{q}\text{ for all }p\in\Sigma\}=s$.
Then $\Omega^{\Sigma,*}_{r+1,s}$ is also a union of $S_{r+1}$-orbits, and
$$
\#\Omega^{\Sigma,*}_{r+1,s}=2^{r(r+1)/2}\binom{r+1}{s}(2^{\#\Sigma}-1)^{r+1-s}.
$$
It's known that (see for example \cite{MU17}, Theorem 4.12)
for $N\in\BZ_{\geq 1}$, $0<p<1$ and any $\delta>0$
(we will choose $\delta$ later), we have
\begin{equation}
\label{e:tail-bound-2}
\sum_{|n-Np|\geq N\delta}\binom{N}{n}p^n(1-p)^{N-n}\leq
2\exp(-2N\delta^2),
\end{equation}
hence
$$
R:=\sum_{|s-(r+1)/2^{\#\Sigma}|\geq(r+1)\delta}
\frac{\#\Omega^{\Sigma,*}_{r+1,s}}{\#\Omega^{\Sigma,*}_{r+1}}\leq
2\exp(-2(r+1)\delta^2).
$$

On the other hand,
if $\omega$ is an element of $\Omega^{\Sigma,*}_{r+1,s}$,
consider the $S_{r+1}$-orbit $[\omega]$, it's easy to see that
$$
\frac{\#([\omega]\cap\Omega^{\Sigma',*}_r)}{\#[\omega]}
=\frac{s}{r+1}.
$$
Therefore, if we
write $B_{r+1,s}:=B_{r+1}\cap\Omega^{\Sigma,*}_{r+1,s}$, then we have
$$
\#(B_{r+1}\cap\Omega^{\Sigma',*}_r)=\sum_{s=0}^{r+1}\#(B_{r+1,s}\cap\Omega^{\Sigma',*}_r)
=\sum_{s=0}^{r+1}\frac{s}{r+1}\#B_{r+1,s}.
$$
Hence
\begin{align*}
\#(B_{r+1}\cap\Omega^{\Sigma',*}_r)
&\leq
\sum_{|s-(r+1)/2^{\#\Sigma}|<(r+1)\delta}
\frac{s}{r+1}\#B_{r+1,s}
+R\cdot\#\Omega^{\Sigma,*}_{r+1} \\
&\leq(2^{-\#\Sigma}+\delta)\cdot\#B_{r+1}
+2\exp(-2(r+1)\delta^2)\cdot\#\Omega^{\Sigma,*}_{r+1}
\end{align*}
as well as
\begin{align*}
\#(B_{r+1}\cap\Omega^{\Sigma',*}_r)
&\geq
\sum_{|s-(r+1)/2^{\#\Sigma}|<(r+1)\delta}
\frac{s}{r+1}\#B_{r+1,s}
\geq(2^{-\#\Sigma}-\delta)\cdot(\#B_{r+1}-R\cdot\#\Omega^{\Sigma,*}_{r+1}) \\
&\geq(2^{-\#\Sigma}-\delta)\cdot\#B_{r+1}
-2\exp(-2(r+1)\delta^2)\cdot\#\Omega^{\Sigma,*}_{r+1},
\end{align*}
therefore
\begin{align*}
|Q^{(r)}-P^{(r+1)}|
&=\left|
\frac{\#(B_{r+1}\cap\Omega^{\Sigma',*}_r)-2^{-\#\Sigma}
\cdot\#B_{r+1}}{\#\Omega^{\Sigma',*}_r}
\right|
\leq\frac{\delta\cdot\#B_{r+1}+2\exp(-2(r+1)\delta^2)\cdot\#\Omega^{\Sigma,*}_{r+1}}
{\#\Omega^{\Sigma',*}_r} \\
&\leq
2^{\#\Sigma}\big(\delta+2\exp(-2(r+1)\delta^2)\big).
\end{align*}
Taking $\delta=\log(r+1)/\sqrt{r+1}$ it's easy to obtain
$|Q^{(r)}-P^{(r+1)}|\leq 2^{\#\Sigma}(\log(r+1)+13/12)/\sqrt{r+1}$.
\end{proof}

\begin{prop}
\label{p:Smi2-6.11a-fake}
Let $N$ be a sufficiently large (depending on $\Sigma$ and $c$ below) real number,
and let $r$ be a positive integer satisfying $r\sim\log\log N$.
Let $B_r$ be any $S_r$-stable subset of $\Omega^{\Sigma,*}_r$.
Then
$$
\left|\BP\big(\omega(n)\in B_r\mid n\in S_r(N,\Sigma)\big)-\BP(B_r)\right|
\leq C_1\cdot\exp(-c(\log\log\log N)^{1/2})
$$
for some absolute constants $C_1>0$ and $c>0$.
\end{prop}

\begin{proof}
Take $D>3$ to be a fixed real number
such that all primes in $\Sigma$ are $<D$.
When $r$ is sufficiently large, we have
$$
S_r(N,\Sigma)=\bigsqcup_qq\cdot S_{r-r(q)}(N/q,\Sigma_D),
$$
where $q$ runs over all positive square-free integers whose prime factors
are in $\Sigma_D\setminus\Sigma$.
For each such $q$, define
$B_q'\subset\Omega_{r-r(q)}^{\Sigma_D,*}$
to be the $B'$ as in the beginning of this section (i.e.~preimage of $B_r$
under certain map),
then for any $n\in S_{r-r(q)}(N/q,\Sigma_D)$,
$\omega(n)\in B_q'$ if and only if $\omega(qn)\in B_r$.
Note that $q$ and $r(q)$ have upper bounds depending only on $D$,
hence it's easy to see that when $N$ is sufficiently large (depending only on $D$),
$r\sim\log\log N$ implies $r-r(q)\sim_D\log\log(N/q)$
for all above $q$.
By Lemma \ref{p:add-primes} it's easy to see that
$$
\left|\BP(B_q')-\BP(B_r)\right|\leq 2^D\cdot
\frac{\log(r+1)+13/12}{\sqrt{r+1}},
$$
on the other hand, by Proposition \ref{p:Smi2-6.11a} it's easy to see that
(here we need to slightly decrease $c$ by an absolute constant,
and taking $N$ sufficiently large, depending only on $c$ and $D$)
$$
\left|\BP\big(\omega(n)\in B_q'\mid n\in S_{r-r(q)}(N/q,\Sigma_D)\big)-\BP(B_q')\right|
\leq C_1\cdot\exp(-c(\log\log\log N)^{1/2}),
$$
hence
\begin{align*}
\left|\BP\big(\omega(n)\in B_r\mid n\in S_r(N,\Sigma)\big)-\BP(B_r)\right|
&\leq\max_q\left|\BP\big(\omega(n)\in B_q'\mid n\in S_{r-r(q)}(N/q,\Sigma_D)\big)-\BP(B_r)\right| \\
&\leq C_1\cdot\exp(-c(\log\log\log N)^{1/2})
+2^D\cdot\frac{\log(r+1)+13/12}{\sqrt{r+1}}.
\end{align*}
The last term in above inequality is
$o(\exp(-c(\log\log\log N)^{1/2}))$ as $N\to\infty$.
Therefore the desired result holds by enlarge $C_1$ by an absolute constant.
\end{proof}

\subsubsection{Change the number of prime divisors}
\label{s:HR}

The Proposition \ref{p:Smi2-6.11a-fake} requires that the number $r$
of prime factors of $n$ are close to $\log\log N$.
But we have the following result of Hardy and Ramanujan,
which implies that almost all numbers $n$ satisfy this property.
From this we can finally prove Theorem \ref{p:natural-density}.

\begin{prop}[Hardy and Ramanujan, \cite{Kane13}, Lemma 6]
\label{p:HR}
There are some absolute constants $C_2>0,C_3>0$
such that for any $r\geq 1$ and any $N\geq 3$, we have
$$
\#S_r(N,\Sigma)\leq
\#S_r(N,\varnothing)\leq
\frac{C_3N}{\log N}\cdot\frac{(\log\log N+C_2)^{r-1}}{(r-1)!}.
$$
\end{prop}

\begin{proof}[Proof of Theorem \ref{p:natural-density}]
First let's give an upper bound of
$$
Q:=\sum_{r\not\sim\log\log N}\frac{\#S_r(N,\Sigma)}{\#S(N,\Sigma)}
=1-\sum_{r\sim\log\log N}\frac{\#S_r(N,\Sigma)}{\#S(N,\Sigma)}.
$$
It's known that (see for example \cite{MU17}, Theorem 5.4)
for $\lambda>0$, we have
\begin{equation}
\label{e:tail-bound-1}
\sum_{n\geq x}\frac{\lambda^n}{n!}\leq\left(\frac{e\lambda}{x}\right)^x
\text{ for }x>\lambda,
\qquad\text{and}\qquad
\sum_{0\leq n\leq x}\frac{\lambda^n}{n!}\leq\left(\frac{e\lambda}{x}\right)^x
\text{ for }0<x<\lambda.
\end{equation}
Consider the functions
$f(x)=\log\left((\frac{e\lambda}{x})^x\right)=x(\log\lambda+1-\log x)$
and $g(x)=\lambda-(x-\lambda)^2/4\lambda$.
Then $f(\lambda)=g(\lambda)=\lambda$, and it's easy to see that
$f(x)\leq g(x)$ when $0<x\leq 2\lambda$.
In particular, for $\lambda>1$, let $0<\alpha<1$ and $0<\beta<1$ be fixed, we have
\begin{equation}
\label{e:tail-bound-1a}
f(\lambda\pm\lambda^\alpha\beta)\leq
g(\lambda\pm\lambda^\alpha\beta)=\lambda-\lambda^{2\alpha-1}\beta^2/4.
\end{equation}
Take $\lambda=\log\log N+C_2$ where $C_2$
is in Proposition \ref{p:HR}.
Then for sufficiently large $N$,
$$
\log\log N-\frac14(\log\log N)^{2/3}+2
<\lambda-\lambda^{2/3}/5
<\lambda+\lambda^{2/3}/5
<\log\log N+\frac14(\log\log N)^{2/3}-2.
$$
On the other hand, there is a constant $C_4>0$ depending only on $\Sigma$ such that
$\#S(N,\Sigma)>C_4N$ for sufficiently large $N$.
Hence by Proposition \ref{p:HR}
and by taking $\alpha=2/3$ and $\beta=1/5$ in \eqref{e:tail-bound-1a},
\begin{align*}
Q&=\sum_{r\not\sim\log\log N}\frac{\#S_r(N,\Sigma)}{\#S(N,\Sigma)}
\leq\frac{C_3}{C_4\log N}\sum_{r\not\sim\log\log N}
\frac{\lambda^{r-1}}{(r-1)!}
\leq\frac{2C_3}{C_4\log N}\exp(\lambda-\lambda^{1/3}/100) \\
&\leq\frac{2C_3\exp(C_2)}{C_4\exp((\log\log N)^{1/3}/100)}
=o(\exp(-c(\log\log\log N)^{1/2}))\quad\text{as}\quad N\to\infty.
\end{align*}

Now we can prove Theorem \ref{p:natural-density}.
Denote the left hand side of Theorem \ref{p:natural-density} by $P$, then
$$
P=\sum_{i\in I}
\left|\sum_{r\geq 1}\frac{\#S_r(N,\Sigma)}{\#S(N,\Sigma)}
\Big(\BP\big(B(n)=i\mid n\in S_r(N,\Sigma)\big)-P^{(\infty)}(i)\Big)\right|
\leq P_A+P_B,
$$
where
$$
P_A:=\sum_{i\in I}
\sum_{r\sim\log\log N}\frac{\#S_r(N,\Sigma)}{\#S(N,\Sigma)}
\Big|\BP\big(B(n)=i\mid n\in S_r(N,\Sigma)\big)-P^{(\infty)}(i)\Big|
$$
and
$$
P_B:=\sum_{i\in I}\sum_{r\not\sim\log\log N}\frac{\#S_r(N,\Sigma)}{\#S(N,\Sigma)}
\Big(\BP\big(B(n)=i\mid n\in S_r(N,\Sigma)\big)+P^{(\infty)}(i)\Big)=2Q.
$$

By applying Proposition \ref{p:Smi2-6.11a-fake}
(on the subset $\{\omega\in\Omega_r^{\Sigma,*}\mid B_r(\omega)=i\}$
for each $i\in I_r$, and
the subset $\{\omega\in\Omega_r^{\Sigma,*}\mid B_r(\omega)\notin I_r\}$)
we have
\begin{align*}
P_A&\leq\sum_{r\sim\log\log N}\frac{\#S_r(N,\Sigma)}{\#S(N,\Sigma)}
\Bigg(\sum_{i\in I_r}\Big|\BP\big(B(n)=i\mid n\in S_r(N,\Sigma)\big)-P^{(r)}(i)\Big|
+\sum_{i\in I_r}\left|P^{(r)}(i)-P^{(\infty)}(i)\right| \\
&\qquad{}+\sum_{i\notin I_r}
\Big(\BP\big(B(n)=i\mid n\in S_r(N,\Sigma)\big)+P^{(\infty)}(i)\Big)\Bigg) \displaybreak[0]\\
&\leq\sum_{r\sim\log\log N}\frac{\#S_r(N,\Sigma)}{\#S(N,\Sigma)}
\Bigg((\#I_r+1)\cdot C_1\cdot\exp(-c(\log\log\log N)^{1/2})
+\sum_{i\in I_r}\left|P^{(r)}(i)-P^{(\infty)}(i)\right| \\
&\qquad{}+\sum_{i\notin I_r}\left(P^{(r)}(i)+P^{(\infty)}(i)\right)\Bigg) \displaybreak[0]\\
&\leq I_\epsilon(\log\log N)\cdot C_1\cdot\exp(-c(\log\log\log N)^{1/2})
+R_\epsilon(\log\log N),
\end{align*}
here we note that when $N$ is sufficiently large (depending on $\epsilon$),
$r\sim\log\log N$ implies $(1-\epsilon)\log\log N<r<(1+\epsilon)\log\log N$.
On the other hand, the $Q$ computed in the above is $o(\exp(-c(\log\log\log N)^{1/2}))$
as $N\to\infty$.
Thus the desired result holds by enlarge $C_1$
by an absolute constant.
\end{proof}

\subsection{Natural average order and Heath-Brown's method}
\label{s:natural average}

The proof of the first part of Theorem \ref{p:natural-average-order} is
due to Heath-Brown \cite{HB93}, \cite{HB94},
and is based on the proof of Proposition \ref{p:high-rank-moment-limit}.

Write $B=(P,Q;R^\RT,S)\in\RRM_{(ak+b)\times(ck+d)}$.
We may keep $a,c$ unchanged and assume
that $P_H=(\diag(A,\cdots,A);0)$,
and that $P_L=0$, namely, $B$ does not contain $z_cz_d^\RT$
(similar to the proof of Proposition \ref{p:low-rank-max-rank}).
For each integer $n$ in $S^{\mathbf s}(N,\Sigma):=\{n\in S^{\mathbf s}(\Sigma)\mid n<N\}$,
let $r:=r(n)$ be the number of prime factors of $n$,
fix an order $\ell_1,\cdots,\ell_r$ of the prime factors of $n$,
let $V=\BF_2^{cr+d}$ and $W=\BF_2^{ar+b}$ (depending on $n$),
then $B(n)$ is a well-defined matrix, and we have
$$
\BE\big(2^{\corank(B(n))}\mid n\in S^{\mathbf s}(N,\Sigma)\big)
=\frac{1}{\#S^{\mathbf s}(N,\Sigma)}
\sum_{n\in S^{\mathbf s}(N,\Sigma)}
\frac{1}{\#W}
\sum_{v\in V,w\in W}
(-1)^{w^\RT B(n)v}.
$$

Recall that we defined the map
$\psi=\psi_{B,v,w}:\Omega_r^{\Sigma,\mathbf s}\to\BF_2$,
$\omega\mapsto w^\RT B(\omega)v$.
It induces the map $\psi=\psi_{B,v,w}:S_r^{\mathbf s}(N,\Sigma)\to\BF_2$,
$n\mapsto w^\RT B(n)v$.
Recall the map $\iota:[r]\to\BF_2^{c+a}$,
and that once $\Im(\iota),\widetilde v,\widetilde w$
are fixed, the $\psi_{B,v,w}(\omega)$ depends only on $a_{\lambda\lambda'}$
for $\lambda\neq\lambda'$ in $\Im(\iota)$,
and $z_\lambda^{(p)}$ for $\lambda\in\Im(\iota)$ and $p\in\Sigma$.
By the same reason, $\psi_{B,v,w}(n)$ depends only on
$$
a_{\lambda\lambda'}=\sum_{i\in[r]_\lambda,j\in[r]_{\lambda'}}a_{ij}
=\sum_{i\in[r]_\lambda,j\in[r]_{\lambda'}}\lrdd{\ell_j}{\ell_i}
=\lrdd{n_{\lambda'}}{n_\lambda}
\text{ for }\lambda\neq\lambda'\text{ in }\Im(\iota),
$$
and
$$
z_\lambda^{(p)}=\sum_{i\in[r]_\lambda}z_i^{(p)}
=\sum_{i\in[r]_\lambda}\lrdd{p}{\ell_i}
=\lrdd{p}{n_\lambda}
\text{ for }\lambda\in\Im(\iota)\text{ and }p\in\Sigma,
$$
here $n_\lambda:=\prod_{i\in[r]_\lambda}\ell_i$
so that $n=\prod_{\lambda\in\Im(\iota)}n_\lambda$.
We also know that, if $\Im(\iota)$ is an unlinked subset,
then $\psi_{B,v,w}(\omega)$ depends only on
$z_\lambda^{(p)}$ for $\lambda\in\Im(\iota)$ and $p\in\Sigma$,
namely $\psi_{B,v,w}$ factors through
$\Omega_r^{\Sigma,\mathbf s}\twoheadrightarrow\Omega_{\Im(\iota)}^{\Sigma,\mathbf s}$.
Similarly, in this case $\psi_{B,v,w}(n)$
depends only on $\lrdd{p}{n_\lambda}$
for $\lambda\in\Im(\iota)$ and $p\in\Sigma$.

Divide the sum over all $n\in S^{\mathbf s}(N,\Sigma)$
into the sum over intervals of each $n_\lambda$ as in \cite{HB94},
we have
$$
\BE\big(2^{\corank(B(n))}\mid n\in S^{\mathbf s}(N,\Sigma)\big)
=\frac{1}{\#S^{\mathbf s}(N,\Sigma)}\sum_{\mathbf A}S(\mathbf A),
$$
where the sum runs over $\mathbf A=(A_\lambda)_{\lambda\in\BF_2^{c+a}}$
with $A_\lambda\in 2^{\BZ_{\geq -1}}=\{\frac12,1,2,4,\cdots\}$
such that $\prod_\lambda A_\lambda<N$, and
$$
S(\mathbf A):=\sum_{\substack{A_\lambda<n_\lambda\leq 2A_\lambda\\
n=\prod_\lambda n_\lambda\in S^{\mathbf s}(N,\Sigma)}}
\frac{1}{\#W}
\sum_{\widetilde v,\widetilde w}
(-1)^{w^\RT B(n)v}.
$$
The number of different $\mathbf A$'s are
$\leq 2^{2^{c+a}}\cdot(\log N)^{2^{c+a}}$.
The terms in $S(\mathbf A)$
satisfy $n_\lambda\neq 1$ (namely $\lambda$ is contained in the image of $\iota$)
if and only if $A_\lambda\geq 1$,
namely, the image of $\iota$ is determined by $\mathbf A$.

These $\mathbf A$'s are divided into three types.
\begin{itemize}
\item[(1)]
There are linked $\lambda,\lambda'$ such that
$A_\lambda\geq\exp(\kappa(\log\log N)^2)$
and $A_{\lambda'}\geq 1$.
Here $\kappa>0$ is some constant depending only on $B$ to be chosen later.
\item[(2)]
$\#\{\lambda\mid A_\lambda\geq\exp(\kappa(\log\log N)^2)\}\leq 2^a-1$.
\item[(3)]
The remaining part.
\end{itemize}

It turns out that the first two types of $\mathbf A$ contribute error terms
(\cite{HB94}, Lemma 6),
and the last type of $\mathbf A$ contributes main term.
More precisely, for the first type of $\mathbf A$, we have the following.

\begin{lem}
There are $\kappa>0$ and $C>0$ depending only on $B$ such that
$$
\sum_{\substack{\mathbf A\text{ such that}\\
\exists\lambda,\lambda'\text{ linked s.t. }
A_\lambda\geq\exp(\kappa(\log\log N)^2),A_{\lambda'}\geq 1
}}
|S(\mathbf A)|
\leq C\cdot N\cdot(\log N)^{-1}.
$$
\end{lem}

\begin{proof}
Let $\lambda,\lambda'\in\BF_2^{c+a}$ be linked
and recall the explicit formula of $\psi_{B,v,w}(\omega)$.
We have
$1=\sum_{\ell=1}^c(\lambda_\ell+\lambda_\ell')(\lambda_{c+\ell}+\lambda_{c+\ell}')
=\sum_{\ell=1}^c(\lambda_\ell+\lambda_\ell')\lambda_{c+\ell}
+\sum_{\ell=1}^c(\lambda_\ell+\lambda_\ell')\lambda_{c+\ell}'$,
namely, exactly one of
$\sum_{\ell=1}^c(\lambda_\ell+\lambda_\ell')\lambda_{c+\ell}$
and $\sum_{\ell=1}^c(\lambda_\ell+\lambda_\ell')\lambda_{c+\ell}'$
is $1$. Therefore, switching $\lambda$ and $\lambda'$ if necessary
(switching or not only depends on $\lambda$ and $\lambda'$,
and is independent of $B,v,w$), we may write
$$
\psi_{B,v,w}(\omega)=a_{\lambda\lambda'}
+\psi_{B,v,w}^{(1)}(\omega)+\psi_{B,v,w}^{(2)}(\omega)
$$
for some $\psi_{B,v,w}^{(1)}:\Omega_r^{\Sigma,\mathbf s}\to\BF_2$
(resp.~$\psi_{B,v,w}^{(2)}:\Omega_r^{\Sigma,\mathbf s}\to\BF_2$)
which depends only on $a_{\mu\mu'}$
for all $\mu\neq\mu'$ in $\Im(\iota)\setminus\{\lambda'\}$
(resp.~$\Im(\iota)\setminus\{\lambda\}$),
and $z_\mu^{(p)}$ for all
$\mu\in\Im(\iota)\setminus\{\lambda'\}$
(resp.~$\Im(\iota)\setminus\{\lambda\}$)
and $p\in\Sigma$.
This implies that
$$
\psi_{B,v,w}(n)=\lrdd{n_{\lambda'}}{n_\lambda}
+\psi_{B,v,w}^{(1)}(n)+\psi_{B,v,w}^{(2)}(n)
$$
where $\psi_{B,v,w}^{(1)}(n)$ is independent of $n_{\lambda'}$
and $\psi_{B,v,w}^{(2)}(n)$ is independent of $n_\lambda$.
Utilizing this decomposition,
the \cite{HB93}, Lemma 7 applies (whose proof uses
\cite{HB94}, Lemma 3 and Lemma 4), namely,
\begin{itemize}
\item
For any $\delta>0$,
there are $\kappa>0$ and $C>0$ depending only on $B$ and $\delta$ such that
for any linked $\lambda,\lambda'\in\BF_2^{c+a}$, and any $\mathbf A$
such that $A_\lambda\geq\exp(\kappa(\log\log N)^2)$ and $A_{\lambda'}\geq 1$,
we have $|S(\mathbf A)|\leq C\cdot N\cdot(\log N)^{-\delta}$.
\end{itemize}
Now taking $\delta=2^{c+a}+1$ the desired result follows.
\end{proof}

For the second type of $\mathbf A$, we have the following.

\begin{lem}
There are $\kappa>0$ and $C>0$ depending only on $B$ such that
$$
\sum_{\substack{\mathbf A\text{ such that}\\
\#\{\lambda\mid A_\lambda\geq\exp(\kappa(\log\log N)^2)\}\leq 2^a-1
}}
|S(\mathbf A)|
\leq C\cdot N\cdot(\log N)^{-2^{-a}}\cdot(\log\log N)^{2^{c+1}}.
$$
\end{lem}

\begin{proof}
We have (\cite{HB94}, pp.~340--341)
$$
\sum_{\substack{\mathbf A\text{ such that}\\
\#\{\lambda\mid A_\lambda\geq\exp(\kappa(\log\log N)^2)\}\leq 2^a-1}}
|S(\mathbf A)|
\leq
2^d\cdot\sum_{\substack{n_\lambda\text{ such that}\\
\#\{\lambda\mid n_\lambda\geq 2\exp(\kappa(\log\log N)^2)\}\leq 2^a-1\\
n=\prod_\lambda n_\lambda\in S^{\mathbf s}(N,\Sigma)}}
\frac{1}{2^{a\cdot\sum_\lambda r(n_\lambda)}}.
$$
For each $(n_\lambda)_{\lambda\in\BF_2^{c+a}}$ appeared in the above sum,
let $n'$ (resp.~$n''$) be the product of $n_\lambda$
such that $n_\lambda$ is $<$ (resp.~$\geq$) $2\exp(\kappa(\log\log N)^2)$.
Then $n'\leq N':=(2\exp(\kappa(\log\log N)^2))^{2^{c+a}}$
and $n''\leq N/n'$.
Conversely, if $n'$ and $n''$ in these ranges are given,
then one can determine
$(n_\lambda)_{\lambda\in\BF_2^{c+a}}$
by determine a map which maps each prime factor of $n'$ and $n''$ to
an element of $\BF_2^{c+a}$.
One may choose a subset of $\BF_2^{c+a}$ of size $2^a-1$,
for each prime factor of $n''$ it maps into this subset,
and for each prime factor of $n'$
it maps to anywhere in $\BF_2^{c+a}$.
There are repeated and invalid choices, hence the above sum is
\begin{align*}
&\leq\sum_{n'\leq N'}\sum_{n''\leq N/n'}
2^{-a(r(n')+r(n''))}\cdot\binom{2^{c+a}}{2^a-1}
\cdot(2^a-1)^{r(n'')}(2^{c+a})^{r(n')} \\
&=\binom{2^{c+a}}{2^a-1}
\sum_{n'\leq N'}2^{c\cdot r(n')}\sum_{n''\leq N/n'}(1-2^{-a})^{r(n'')}
\leq O(1)\cdot\sum_{n'\leq N'}2^{c\cdot r(n')}\cdot N/n'
\cdot(\log N/n')^{-2^{-a}} \\
&\leq O(1)\cdot N\cdot(\log N)^{-2^{-a}}
\sum_{n'\leq N'}\frac{2^{c\cdot r(n')}}{n'}
\leq O(1)\cdot N\cdot(\log N)^{-2^{-a}}
\cdot(\log N')^{2^c} \\
&\leq O(1)\cdot N\cdot(\log N)^{-2^{-a}}
\cdot(\log\log N)^{2^{c+1}}.
\end{align*}
This completes the proof.
\end{proof}

Suppose $\mathbf A$ is of the last type.
Namely,
\begin{itemize}
\item
Let $\Lambda:=\{\lambda\mid A_\lambda\geq\exp(\kappa(\log\log N)^2)\}$.
Then $\#\Lambda\geq 2^a$.
\item
For any linked $\lambda,\lambda'$,
either $A_\lambda<\exp(\kappa(\log\log N)^2)$
or $A_{\lambda'}<1$.
\end{itemize}
This implies that for any $\lambda,\lambda'\in\Lambda$,
they are not linked, hence by Lemma \ref{p:HB94-lem7}, $\#\Lambda=2^a$,
$\Lambda$ is a maximal unlinked subset,
so for any $\lambda'\notin\Lambda$ it must be linked
with some elements in $\Lambda$, which implies $A_{\lambda'}<1$.
In other words, we have $\Lambda=\{\lambda\mid A_\lambda\geq 1\}=\Im(\iota)$
which is a maximal unlinked subset.
Recall that this implies that once $\widetilde v,\widetilde w$ are fixed,
the $\psi_{B,v,w}(n)=w^\RT B(n)v$
depends only on
$\mathbf z=\left(\lrdd{p}{n_\lambda}\right)_{\lambda\in\Lambda,p\in\Sigma}
\in\Omega_\Lambda^{\Sigma,\mathbf s}$.
We have
$$
S(\mathbf A)=
2^{-b}
\sum_{\mathbf z=(z_\lambda^{(p)})\in\Omega_\Lambda^{\Sigma,\mathbf s}}
\sum_{\widetilde v,\widetilde w}
(-1)^{\psi_{B,v,w}(\mathbf z)}
\underbrace{\sum_{\substack{A_\lambda<n_\lambda\leq 2A_\lambda,\lambda\in\Lambda\\
n=\prod_{\lambda\in\Lambda}n_\lambda\in S^{\mathbf s}(N,\Sigma)\\
\lrdd{p}{n_\lambda}=z_\lambda^{(p)},\forall\lambda,p}}
\frac{1}{2^{a\cdot\sum_\lambda r(n_\lambda)}}}_{(*)}.
$$
The $(*)$ is equal to
$$
2^{-2^a\cdot\#\Sigma}
\sum_{\mathbf y=(y_\lambda^{(p)})\in\Omega_\Lambda^{\Sigma,*}}
(-1)^{\sum_{p,\lambda}y_\lambda^{(p)}z_\lambda^{(p)}}
\underbrace{\sum_{\substack{A_\lambda<n_\lambda\leq 2A_\lambda,\lambda\in\Lambda\\
n=\prod_{\lambda\in\Lambda}n_\lambda\in S^{\mathbf s}(N,\Sigma)}}
\left(\prod_{\lambda\in\Lambda}\lrd{d_\lambda}{n_\lambda}\right)
\frac{1}{2^{a\cdot\sum_\lambda r(n_\lambda)}}}_{(**)},
$$
where $\Omega_\Lambda^{\Sigma,*}
:=\{(y_\lambda^{(p)})_{\lambda\in\Lambda,p\in\Sigma}\}
\cong\BF_2^{\#\Lambda\cdot\#\Sigma}$
and $d_\lambda:=\prod_{p\in\Sigma}p^{y_\lambda^{(p)}}\in\BQ(\Sigma,2)$.
If there are $\lambda,\lambda'$ such that $(y_\lambda^{(p)})_{p\in\Sigma}
\neq(y_{\lambda'}^{(p)})_{p\in\Sigma}$, namely $d_\lambda\neq d_{\lambda'}$,
then $(**)$ is an error term
(\cite{HB94}, Lemma 12),
and the sum of absolute values of them is
(\cite{HB94}, p.~348)
$$
\leq C\cdot N\cdot(\log N)^{2^a-c'\kappa^{1/2}}
$$
for some constant $c'>0$ depending only on $\Sigma$, $a$,
and some constant $C>0$ depending only on $\Sigma$, $a$ and $\kappa$.
If the $(y_\lambda^{(p)})_{p\in\Sigma}$ are equal for all $\lambda$,
say all equal to $(y^{(p)})_{p\in\Sigma}$, then $(**)$ is
$$
\sum_{\substack{A_\lambda<n_\lambda\leq 2A_\lambda,\lambda\in\Lambda\\
n=\prod_{\lambda\in\Lambda}n_\lambda\in S^{\mathbf s}(N,\Sigma)}}
(-1)^{\sum_py^{(p)}s_p}
\frac{1}{2^{a\cdot\sum_\lambda r(n_\lambda)}}
$$
and the outer $(-1)^{\sum_{p,\lambda}y_\lambda^{(p)}z_\lambda^{(p)}}$
is also $(-1)^{\sum_py^{(p)}s_p}$.
Hence the main term of $(*)$ is
independent of $\mathbf z$, $\widetilde v$ and $\widetilde w$;
sum all of them together we obtain that
\begin{multline*}
S(\mathbf A)
=2^{(1-2^a)\cdot\#\Sigma-b}
\left(\sum_{\mathbf z\in\Omega_\Lambda^{\Sigma,\mathbf s}}
\sum_{\widetilde v,\widetilde w}
(-1)^{\psi_{B,v,w}(\mathbf z)}\right)
\left(\sum_{\substack{A_\lambda<n_\lambda\leq 2A_\lambda,\lambda\in\Lambda\\
n=\prod_{\lambda\in\Lambda}n_\lambda\in S^{\mathbf s}(N,\Sigma)}}
\frac{1}{2^{a\cdot\sum_\lambda r(n_\lambda)}}\right) \\
+O(N\cdot(\log N)^{2^a-c'\kappa^{1/2}}).
\end{multline*}
Let $\Lambda$ be fixed and let $A_\lambda\geq\exp(\kappa(\log\log N)^2)$
varies for $\lambda\in\Lambda$.
There are $\leq 2^{2^a}\cdot(\log N)^{2^a}$ choices.
By \cite{HB94}, p.~348,
the sum of all such $S(\mathbf A)$ is
\begin{multline*}
2^{(1-2^a)\cdot\#\Sigma-b}
\left(\sum_{\mathbf z\in\Omega_\Lambda^{\Sigma,\mathbf s}}
\sum_{\widetilde v,\widetilde w}
(-1)^{\psi_{B,v,w}(\mathbf z)}\right)
\left(\#S^{\mathbf s}(N,\Sigma)+O(N\cdot(\log N)^{-2^{-a}}\cdot(\log\log N)^2)\right) \\
+O(N\cdot(\log N)^{2\cdot 2^a-c'\kappa^{1/2}}).
\end{multline*}
Finally, let $\Lambda$ varies, combined with the terms of the first
two types of $\mathbf A$,
we obtain that
\begin{multline*}
\BE\big(2^{\corank(B(n))}\mid n\in S^{\mathbf s}(N,\Sigma)\big)
=2^{(1-2^a)\cdot\#\Sigma-b}
\sum_{\substack{\Lambda\text{ unlinked}\\
\#\Lambda=2^a}}
\sum_{\mathbf z\in\Omega_\Lambda^{\Sigma,\mathbf s}}
\sum_{\widetilde v,\widetilde w}
(-1)^{\psi_{B,v,w}(\mathbf z)} \\
+O((\log N)^{-2^{-a}}\cdot(\log\log N)^2)
+O((\log N)^{2\cdot 2^a-c'\kappa^{1/2}})
+O((\log N)^{-2^{-a}}\cdot(\log\log N)^{2^{c+1}}).
\end{multline*}
Taking $\kappa$ sufficiently large, only the last error term remains.
The main term is the same as the formula \eqref{e:high-rank-moment}
of $E^{(\infty)}$ in Proposition \ref{p:high-rank-moment-limit}.
Hence the first part of Theorem \ref{p:natural-average-order} holds.

\subsection{Essential average order and Kane's method}
\label{s:essential average}

The proof of the second part of Theorem \ref{p:natural-average-order} is
due to Kane \cite{Kane13},
and is also based on the proof of Proposition \ref{p:high-rank-moment-limit}.

Again let $B=(P,Q;R^\RT,S)\in\RRM_{(ak+b)\times(ck+d)}$
be such that $P_H=(\diag(A,\cdots,A);0)$
and $P_L=0$.
Study the property of the map
$\psi=\psi_{B,v,w}:\Omega_r^{\Sigma,\mathbf s}\to\BF_2$.
Recall that $\psi$ is determined by $\iota:[r]\to\BF_2^{c+a}$,
$\widetilde v\in\BF_2^d$ and $\widetilde w\in\BF_2^b$.
Recall that $\Lambda:=\Im(\iota)$.
Define
\begin{align*}
m(\psi)&:=\#\{i\mid
\text{there exists }j\neq i\text{ such that }
a_{ij}\text{ or }a_{ji}\text{ occurred in the expression of }\psi(\omega)\} \\
&=\sum_{\substack{\lambda\in\Lambda\\
\text{s.t. }\lambda,\lambda'\text{ linked for some }\lambda'\in\Lambda}}\#[r]_\lambda
=r-\sum_{\substack{\lambda\in\Lambda\\
\text{s.t. }\lambda,\lambda'\text{ unlinked for all }\lambda'\in\Lambda}}\#[r]_\lambda.
\end{align*}
Recall that if $m(\psi)\neq 0$, then
$\sum_{\omega\in\Omega_r^{\Sigma,\mathbf s}}(-1)^{\psi(\omega)}=0$.
Correspondingly,
the result of Kane assert that, if $m(\psi)\neq 0$,
then $\psi$ produces error term in the essential average order.

\begin{prop}
\label{p:Kane-char-bound}
Let $N>30$ be a real number and $r\geq 1$ be an integer such that
$\frac12<\frac{r}{\log\log N}<2$.

{\rm(i) (\cite{Kane13}, Proposition 9)}
Let $0<\epsilon<1$. There exists a constant $C>0$ depending only on $\epsilon$ and $\Sigma$,
such that for any above $\psi$ with $m(\psi)>0$, we have
$$
\left|
\sum_{n\in S_r^{\mathbf s}(N,\Sigma)}
(-1)^{\psi(n)}
\right|<C\cdot N\cdot\epsilon^{m(\psi)}.
$$

{\rm(ii) (\cite{Kane13}, Proposition 10)}
There exists a constant $C>0$ depending only on $\Sigma$,
such that for any above $\psi$
with $m(\psi)=0$, we have
$$
\left|
\sum_{n\in S_r^{\mathbf s}(N,\Sigma)}
(-1)^{\psi(n)}-\frac{\#S_r^{\mathbf s}(N,\Sigma)}{\#\Omega_r^{\Sigma,\mathbf s}}
\sum_{\omega\in\Omega_r^{\Sigma,\mathbf s}}(-1)^{\psi(\omega)}
\right|<C\cdot\frac{N\log\log\log N}{\log\log N}.
$$
\end{prop}

For each $0\leq m\leq r$, we estimate the number of $\psi$
such that $m(\psi)=m$.
\begin{itemize}
\item
For $m=0$ case, it means that $\Lambda$ is an unlinked subset,
contained in a maximal unlinked subset,
which has $2^c\cdot\prod_{i=1}^c(2^i+1)$ of them.
Hence the number of choices of $\psi$ is
$\leq 2^c\cdot\prod_{i=1}^c(2^i+1)\cdot 2^{b+d}\cdot(2^a)^r$.
\item
For $1\leq m\leq r$ case, let $\Lambda_0:=\{\lambda\in\Lambda\mid
\lambda,\lambda'\text{ unlinked for all }\lambda'\in\Lambda\}\subsetneqq\Lambda$.
It's clear that for any $\lambda'\in\Lambda\setminus\Lambda_0$,
the $\Lambda_0\sqcup\{\lambda'\}$ is also an unlinked subset,
hence $\#\Lambda_0\leq 2^a-1$.
There are at most $2^a\cdot 2^c\cdot\prod_{i=1}^c(2^i+1)$
unlinked subsets of size $2^a-1$.
The $m(\psi)=m$ means that there are exactly $r-m$ elements $i\in[r]$
such that $\iota(i)\in\Lambda_0$.
Therefore the number of choices of $\psi$ is
$\leq 2^a\cdot 2^c\cdot\prod_{i=1}^c(2^i+1)\cdot 2^{b+d}
\cdot\binom{r}{m}(2^a-1)^{r-m}(2^{c+a})^m$.
\end{itemize}
Sum of these $\psi$ together,
applying Proposition \ref{p:Kane-char-bound},
for $\frac12<\frac{r}{\log\log N}<2$, we have
$$
\left|\BE\big(2^{\corank(B(n))}\mid n\in S_r^{\mathbf s}(N,\Sigma)\big)
-\BE(2^{X_r})\right| \\
<\frac{C\cdot N}{\#S_r^{\mathbf s}(N,\Sigma)}\cdot
\left(\frac{\log\log\log N}{\log\log N}
+(1-2^{-a}+2^c\epsilon)^r\right),
$$
where $C>0$ is a constant depending only on $\epsilon$, $\Sigma$ and $B$.
Take $\epsilon=2^{-c-a-1}$, summing over $r\sim\log\log N$,
it's easy to obtain
the second part of Theorem \ref{p:natural-average-order}.

\subsection{Proof of main results on $2$-Selmer groups}
\label{s:proof}

The results of Theorem \ref{main theorem}, Theorem \ref{main strict}
and Theorem \ref{main average} are stated using $\Sigma$-equivalence class $\fX$
of elliptic curves, and the condition is written as $E\in\fX$, $N(E)<N$.
On the other hand,
we may represent a $\Sigma$-equivalence class $\fX$
of elliptic curves with full rational $2$-torsion points
as $\fX=\{E^{(n)}\mid n\in S^{\mathbf s}(\Sigma)\}$,
where $E:y^2=x(x-e_1)(x-e_2)$ is a fixed elliptic curve with $e_1,e_2\in\BZ$
such that the prime divisors of $2e_1e_2(e_1-e_2)$ are all contained in $\Sigma$,
and $S^{\mathbf s}(\Sigma)$ represents
a $\Sigma$-equivalence class of square-free integers,
defined in \eqref{e:Ss-Sigma}.
Then there exists a constant $C>0$ depending only on $\Sigma$
such that for any $n\in S^{\mathbf s}(\Sigma)$ the conductor $N(E^{(n)})$ of $E^{(n)}$
satisfies $1/C<N(E^{(n)})/n^2<C$.
Therefore by changing the constants in these theorems
by absolute constants, we may replace
the condition $E\in\fX,N(E)<N$ in these theorems
by $n\in S^{\mathbf s}(N,\Sigma)$.

\begin{proof}[Proof of Theorem \ref{main theorem} and Theorem \ref{main strict}]
They are consequences of
Theorem \ref{p:2-descent-redei-mat} (which relates the $2$-Selmer group to certain R\'edei matrix),
Theorem \ref{redei matrix main thm} (which gives the limit distribution of such R\'edei matrix),
and Theorem \ref{p:natural-density} (which relates limit distribution to natural distribution).
The invariance of parameter $\mathbf t$
comes from Proposition \ref{invariance of parameter}.
In the following we give the details.

Let $r=r(n)$ be the number of prime factors of $n$;
we only consider $r\geq 1$ case.
Theorem \ref{p:2-descent-redei-mat},
\ref{p:2-descent-redei-mat-2}, \ref{p:2-descent-redei-mat-0} and
Corollary \ref{selmer mat unrestricted version}
tell us that the $2$-Selmer group $\Sel_2(E^{(n)}/\BQ)$,
$\pi$-strict Selmer group $\Sel_{2,\pi\sstr}(E^{(n)}/\BQ)$,
modified $\pi$-strict Selmer group $\Sel_{2,\pi\sstr}'(E^{(n)}/\BQ)$,
as well as their essential versions $S(E^{(n)})$,
$S_{\pi\sstr}(E^{(n)})$,
$S_{\pi\sstr}'(E^{(n)})$, depend only on
$\omega=\omega(n)\in\Omega_r^{\Sigma,*}$.
More precisely, we take $B\in\RRM'_{(2(k-1)+t_\fX)\times(2(k-1)+t_\fX)}$
be the matrix $\widetilde B$ in Corollary \ref{selmer mat unrestricted version}
(see also Theorem \ref{p:2-descent-redei-mat},
\ref{p:2-descent-redei-mat-2}),
which is a ``high-rank alternating R\'edei matrix of level $2$''
in Definition \ref{p:high-rank-alternating-defn},
with associated submatrices $B_j'\in\RRM'_{(2(k-1)+t_\fX)\times((k-1)+t_\fX)}$
for $1\leq j\leq s$ defined in \S\ref{s:markov}
(which are $\widetilde B_j'$ in Corollary \ref{selmer mat unrestricted version}),
such that $\dim_{\BF_2}S(E^{(n)})=\corank(B(\omega))$
and $\dim_{\BF_2}S_{\pi_j\sstr}'(E^{(n)})=\corank(B_j'(\omega))$.
On the other hand, for $1\leq j\leq s$
we can take $B_j\in\RRM'_{(2(k-1)+(a+2))\times((k-1)+\#\Sigma)}$,
which is the unrestricted variation of the $B_j$ in Theorem \ref{p:2-descent-redei-mat-0},
such that $\dim_{\BF_2}S_{\pi_j\sstr}(E^{(n)})=\corank(B_j(\omega))$.

As in Theorem \ref{redei matrix main thm}, write $I:=\BZ_{\geq 0}^{s+1}$,
define the random variables
\begin{align*}
Y_k':\Omega_k^{\Sigma,\mathbf s}&\to I,&
\omega&\mapsto
\big(\corank(B(\omega)),(\corank(B_j'(\omega)))_{1\leq j\leq s}\big) \\
& & &=\big(\dim_{\BF_2}S(E^{(\omega)}),
(\dim_{\BF_2}S_{\pi_j\sstr}'(E^{(\omega)}))_{1\leq j\leq s}\big), \displaybreak[0]\\
Y_k:\Omega_k^{\Sigma,\mathbf s}&\to I,&
\omega&\mapsto
\big(\corank(B(\omega)),(\corank(B_j(\omega)))_{1\leq j\leq s}\big) \\
& & &=\big(\dim_{\BF_2}S(E^{(\omega)}),
(\dim_{\BF_2}S_{\pi_j\sstr}(E^{(\omega)}))_{1\leq j\leq s}\big), \displaybreak[0]\\
X_k:\Omega_k^{\Sigma,\mathbf s}&\to\BZ_{\geq 0},&
\omega&\mapsto\corank(B(\omega))=\dim_{\BF_2}S(E^{(\omega)}).
\end{align*}
Let $\xi\in\BR_{\geq 1}$ be any fixed real number.
By applying \eqref{e:redei-matrix-sum-prob-alt}
of Theorem \ref{redei matrix main thm} to the R\'edei matrix $B$,
we may find constants $C>0$ and $0<\alpha<1$ such that for all $k$,
\begin{equation}
\label{e:main-thm-Y'}
\sum_{\mathbf m=(m,m_1',\cdots,m_s')\in I}2^{\xi\cdot m}\cdot\left|\BP(Y_k'=\mathbf m)
-P_{t,\mathbf t}^\Alt(\mathbf m)\right|
<C\cdot\alpha^k.
\end{equation}
To compare $Y_k'$ with $Y_k$, we consider the event $F$
defined in \eqref{e:lin-indep-event}.
We may require the above $C$ such that
$\BP(\text{not }F)\leq C\cdot 2^{-k}$ for all $k$.
The Corollary \ref{p:strict-Sel-eq} tells us that
$\BP(Y_k'=\mathbf m\text{ and }F)=\BP(Y_k=\mathbf m\text{ and }F)$.
Also, $S_{\pi_j\sstr}(E^{(\omega)})
\subset S_{\pi_j\sstr}'(E^{(\omega)})
\subset S(E^{(\omega)})$, hence
$\BP(Y_k'=\mathbf m)=\BP(Y_k=\mathbf m)=0$ unless
$\mathbf m=(m,m_1',\cdots,m_s')$ satisfies
$0\leq m_j'\leq m$
for $1\leq j\leq s$.
By Corollary \ref{redei matrix upper bound},
we may find a constant $\alpha>0$ and require the above $C$ such that
$\BP(X_k=m)<C\cdot 2^{-\alpha m^2}$
for all $k\geq 1$, $m\geq 0$.
Therefore
\begin{multline}
\label{e:main-thm-Y'-Y}
\sum_{\mathbf m=(m,m_1',\cdots,m_s')\in I}2^{\xi\cdot m}\cdot\left|\BP(Y_k'=\mathbf m)
-\BP(Y_k=\mathbf m)\right|
\leq\sum_{m\geq 0}(m+1)^s\cdot 2^{\xi\cdot m}\cdot
\BP(X_k=m\text{ and not }F) \\
\leq\sum_{m=0}^M(m+1)^s\cdot 2^{\xi\cdot m}\cdot C\cdot 2^{-k}
+\sum_{m=M+1}^\infty(m+1)^s\cdot 2^{\xi\cdot m}\cdot C\cdot 2^{-\alpha m^2},
\end{multline}
here $M\in\BZ_{\geq 0}$ is to be chosen later.
In fact, we choose $M=\lfloor\frac{k}{2(1+\xi)}\rfloor$,
and only consider sufficiently large $k$, such that $(m+1)^{s+1}\leq 2^m$
for all $m\geq M$. Then we have
\begin{equation}
\label{e:main-thm-Y'-Y-1}
\sum_{m=0}^M(m+1)^s\cdot 2^{\xi\cdot m}\cdot C\cdot 2^{-k}
\leq(M+1)^{s+1}\cdot 2^{\xi\cdot M}\cdot C\cdot 2^{-k}
\leq 2^{(\xi+1)M}\cdot C\cdot 2^{-k}\leq C\cdot 2^{-k/2}.
\end{equation}
Also, for all $m$ we have
$(\xi+1)m-\alpha m^2\leq-(\xi+1)(m-\frac{\xi+1}{\alpha})$
(taking ``$=$'' when $m=\frac{\xi+1}{\alpha}$), so
\begin{multline}
\label{e:main-thm-Y'-Y-2}
\sum_{m=M+1}^\infty(m+1)^s\cdot 2^{\xi\cdot m}\cdot C\cdot 2^{-\alpha m^2}
\leq\sum_{m=M+1}^\infty 2^{(\xi+1)m}\cdot C\cdot 2^{-\alpha m^2}
\leq C\cdot\sum_{m=M+1}^\infty 2^{-(\xi+1)(m-\frac{\xi+1}{\alpha})} \\
=C\cdot 2^{-(\xi+1)(M+1-\frac{\xi+1}{\alpha})}/(1-2^{-(\xi+1)})
\leq C\cdot 2^{-k/2}\cdot 2^{(\xi+1)^2/\alpha}/(1-2^{-(\xi+1)}).
\end{multline}

By enlarging $C$, the ``sufficiently large $k$'' condition in the above can be removed.
Combine \eqref{e:main-thm-Y'},
\eqref{e:main-thm-Y'-Y}, \eqref{e:main-thm-Y'-Y-1},
\eqref{e:main-thm-Y'-Y-2}, we know that
there are constants $C>0$ and $0<\alpha<1$ such that for all $k$,
\begin{equation}
\label{e:main-thm-Y}
\sum_{\mathbf m=(m,m_1',\cdots,m_s')\in I}2^{\xi\cdot m}\cdot\left|\BP(Y_k=\mathbf m)
-P_{t,\mathbf t}^\Alt(\mathbf m)\right|
<C\cdot\alpha^k.
\end{equation}
In particular, $\lim_{k\to\infty}\BP(Y_k=\mathbf m)=P_{t,\mathbf t}^\Alt(\mathbf m)$.
Hence we may apply Theorem \ref{p:natural-density}
(taking any $0<\epsilon<1$)
to the random variables $Y_k$
(need to assign a junk value on
$\Omega_k^{\Sigma,*}\setminus\Omega_k^{\Sigma,\mathbf s}$)
and which tells us that
\begin{multline}
\label{e:main-thm-Y-nat}
\sum_{d,d_1',\cdots,d_s'\geq 0}\left|\BP\left(
\begin{array}{l}
\dim_{\BF_2}S(E^{(n)})=d\text{ and} \\
\dim_{\BF_2}S_{\pi_i\sstr}(E^{(n)})=d_i'\text{ for }1\leq i\leq s
\end{array}
\middle|~n\in S_{\geq 1}^{\mathbf s}(N,\Sigma)
\right)-P_{r,\mathbf t}^\Alt(d,d_1',\cdots,d_s')\right| \\
\leq I_\epsilon(\log\log N)\cdot C_1\cdot\exp(-c(\log\log\log N)^{1/2})
+R_\epsilon(\log\log N),
\end{multline}
here for each $k$, the finite subset $I_k\subset I$ still need to be chosen.
We choose $I_k:=\{(m,m_1',\cdots,m_s')\in I\mid
0\leq m_j'\leq m\leq\log k\text{ for }1\leq j\leq s\}$.
Then $\#I_k\leq(1+\log k)^3$
and $I_\epsilon(x)\leq 1+(1+\log(1+\epsilon)+\log x)^3
\leq 2(\log x)^3$ when $x$ is sufficiently large.
As for $R_\epsilon(x)$, similar to the above argument,
we only consider sufficiently large $x$, such that
for any $k\geq(1-\epsilon)x$, any $m>\log k$, we have $(m+1)^s\leq 2^m$,
then by \eqref{e:main-thm-Y} and Corollary \ref{redei matrix upper bound},
it is easy to see that
\begin{align*}
R_\epsilon(x)&\leq
\max_{(1-\epsilon)x\leq k\leq(1+\epsilon)x}
\left(C\cdot\alpha^k
+2C\cdot 2^{1/\alpha-\log k}\right)
\leq C\cdot\alpha^{(1-\epsilon)x}+2C\cdot 2^{1/\alpha-\log(1-\epsilon)-\log x}.
\end{align*}
From \eqref{e:main-thm-Y-nat} and the upper bound of $I_\epsilon(x)$ and $R_\epsilon(x)$
it is easy to obtain \eqref{e:main-theorem-refined} of Theorem \ref{main theorem}
(need to change $C_1$ and $c$ by absolute constants).

The \eqref{e:main-theorem} of Theorem \ref{main theorem}
and the $t_\pi\neq-\infty$ case
(namely, corresponding to $S_{\pi_i\sstr}$ for $1\leq i\leq s$)
of Theorem \ref{main strict}
are direct consequences of \eqref{e:main-theorem-refined}.
Now we prove Theorem \ref{main strict} for
$t_\pi=-\infty$ case.
In this case, we let $B\in\RRM'_{(2(k-1)+(a+2))\times((k-1)+\#\Sigma)}$,
which is the unrestricted variation of the $B_j$ in Theorem \ref{p:2-descent-redei-mat-0},
such that $\dim_{\BF_2}S_{\pi\sstr}(E^{(n)})=\corank(B(\omega))$.
Then $B$
satisfies the conditions in Proposition \ref{p:high-rank-non-square},
hence if we define random variables
\begin{align*}
Z_k:\Omega_k^{\Sigma,\mathbf s}&\to\BZ_{\geq 0},&
\omega&\mapsto\corank(B(\omega))=\dim_{\BF_2}S_{\pi\sstr}(E^{(\omega)}),
\end{align*}
then there is a constant $C>0$ such that for any $m\geq 1$ and $k\geq 1$,
$\BP(Z_k\geq m)\leq C\cdot k\cdot 2^{-m}(3/4)^k$,
in particular, for any $m\geq 0$, $\lim_{k\to\infty}\BP(Z_k=m)=P_{-\infty}^\Mat(m)$.
Again by Theorem \ref{p:natural-density} we have
\begin{multline}
\label{e:main-thm-Z-nat}
\sum_{d\geq 0}\left|\BP\left(
\dim_{\BF_2}S_{\pi\sstr}(E^{(n)})=d
~\middle|~n\in S_{\geq 1}^{\mathbf s}(N,\Sigma)
\right)-P_{-\infty}^\Mat(d)\right| \\
\leq I_\epsilon(\log\log N)\cdot C_1\cdot\exp(-c(\log\log\log N)^{1/2})
+R_\epsilon(\log\log N).
\end{multline}
In this case, we just take $I_k=\{0\}$ for all $k$, hence
$I_\epsilon(x)=2$ and
$$
R_\epsilon(x)\leq\max_{(1-\epsilon)x\leq k\leq(1+\epsilon)x}
\left(2\cdot C\cdot k\cdot 2^{-1}(3/4)^k\right)
\leq 2\cdot(1-\epsilon)x\cdot(3/4)^{(1-\epsilon)x}
$$
when $x$ is sufficiently large.
From \eqref{e:main-thm-Z-nat} and the upper bound of $I_\epsilon(x)$ and $R_\epsilon(x)$
it is easy to obtain $t_\pi=-\infty$ case of Theorem \ref{main strict}.
\end{proof}

\begin{proof}[Proof of Theorem \ref{main average}]
It comes from
the \eqref{e:redei-matrix-sum-prob-alt} of Theorem \ref{redei matrix main thm} (which relates the limit of $\xi$-th moment of R\'edei matrix to that of the corresponding matrix model),
Appendix \ref{average order matrix model} (which computes the $\xi$-th moment
of the matrix model in $\xi=1$ case),
and Theorem \ref{p:natural-average-order} (which relates the limit of $\xi$-th moment
to the natural and essential $\xi$-th moment).
In the following we give the details.

We take $\widetilde B\in\RRM'_{(2(k-1)+t_\fX)\times(2(k-1)+t_\fX)}$
be the unrestricted R\'edei matrix $\widetilde B$
in Corollary \ref{selmer mat unrestricted version},
take $B\in\RRM_{(2k+t_\fX)\times(2k+t_\fX)}$
be the restricted R\'edei matrix $B$
in Theorem \ref{p:2-descent-redei-mat},
such that $\dim_{\BF_2}S(E^{(n)})+2
=\dim_{\BF_2}\Sel_2(E^{(n)}/\BQ)
=\corank(\widetilde B(\omega))+2
=\corank(B(\omega))$.
Define the random variables
\begin{align*}
X_k:\Omega_k^{\Sigma,\mathbf s}&\to\BZ_{\geq 0},&
\omega&\mapsto\corank(\widetilde B(\omega))
=\corank(B(\omega))-2=\dim_{\BF_2}S(E^{(\omega)}).
\end{align*}
Then by \eqref{e:redei-matrix-sum-prob-alt}
of Theorem \ref{redei matrix main thm}
(or by \eqref{e:main-thm-Y'}
or \eqref{e:main-thm-Y})
we know that
for any $\xi\in\BZ_{\geq 1}$
there exists $C>0$ and $0<\alpha<1$ such that for all $k\geq 1$,
\begin{equation}
\label{e:main-thm-X}
\sum_{m\geq 0}2^{\xi\cdot m}\cdot\left|\BP(X_k=m)
-P_{t,\mathbf t}^\Alt(m)\right|
<C\cdot\alpha^k.
\end{equation}
Apply Theorem \ref{p:natural-average-order} to the restricted
R\'edei matrix $B^{\oplus\xi}=\diag(B,\cdots,B)$,
we know that the natural average order satisfies
$$
\left|2^{2\xi}\cdot\BE\big(\#S(E^{(n)})^\xi\mid n\in S^{\mathbf s}(N,\Sigma)\big)
-E^{(\infty)}\right|<C\cdot\frac{(\log\log N)^{2^{2\xi+1}}}{(\log N)^{1/2^{2\xi}}},
$$
and the ``essential'' average order satisfies
$$
\left|2^{2\xi}\cdot\BE\big(\#S(E^{(n)})^\xi\mid n\in S^{\mathbf s}(N,\Sigma),r(n)\sim\log\log N\big)-E^{(\infty)}\right|
<C\cdot\frac{\log\log\log N}{(\log\log N)^{1/3}}.
$$
Here
$$
E^{(\infty)}=\lim_{k\to\infty}\BE(2^{\xi(X_k+2)})
=2^{2\xi}\cdot\lim_{k\to\infty}
\sum_{m\geq 0}2^{\xi\cdot m}\cdot\BP(X_k=m)
=2^{2\xi}
\sum_{d=0}^\infty P_{t, \mathbf{t}}^{\Alt}(d) 2^{\xi d},
$$
where the last equality is by \eqref{e:main-thm-X}.
Hence we obtain the two inequalities
\eqref{e:average-order-1} and \eqref{e:essential-average-order-1}
in Theorem \ref{main average}.
The \eqref{e:average-moment} in Theorem \ref{main average} is a consequence of them.
For the $\xi=1$ case, by Appendix \ref{average order matrix model}
we know that
$\sum_{d=0}^\infty P_{t, \mathbf{t}}^{\Alt}(d) 2^{\xi d}
=3+\sum_i2^{t_i}$, hence it gives \eqref{e:average-order}
in Theorem \ref{main average}.

For the $S_{\pi\sstr}$ it is similar.
When $t_\pi\neq-\infty$, the \eqref{e:main-thm-Y}
implies that
\begin{equation}
\label{e:main-thm-Z}
\sum_{m\geq 0}2^{\xi\cdot m}\cdot\left|\BP(Z_k=m)
-P_{t_\pi}^\Mat(m)\right|
<C\cdot\alpha^k.
\end{equation}
When $t_\pi=-\infty$ it also holds, similar to the argument
from \eqref{e:main-thm-Y'} to \eqref{e:main-thm-Y}, using
the following upper bounds:
by Corollary \ref{redei matrix upper bound} and Proposition \ref{p:high-rank-non-square},
we may find $C>0$ and $\alpha>0$ such that
$\BP(Z_k=m)<C\cdot 2^{-\alpha m^2}$ (notice $Z_k\leq X_k$) for all
$k\geq 1$, $m\geq 0$,
and such that $\BP(Z_k\geq m)\leq C\cdot k\cdot 2^{-m}(3/4)^k$ for all
$k\geq 1$, $m\geq 1$.
The desired result comes from applying
Theorem \ref{p:natural-average-order} to the restricted
R\'edei matrix $B_j$ in Theorem \ref{p:2-descent-redei-mat-0},
and repeat the above argument.
\end{proof}

\appendix

\section{Markov chain, model of alternating matrices}

\subsection{Basic property of Markov chains}
\label{s:basic-property-markov}

Let $I$ be a countable set, usually called the state space.
Let $\Omega$ be a probability space.
If $X:\Omega\to I$ is a random variable, its
probability mass function $x=(x_i)_{i\in I}$,
$x_i=\BP(X=i)$ is an element in the $\ell^1$-space
$$
\ell^1(I):=\left\{x=(x_i)_{i\in I}\in\BR^I~\middle|~
\|x\|=\|x\|_1:=\sum_{i\in I}|x_i|<\infty\right\},
$$
such that $x_i\geq 0$ and $\|x\|=1$.
Namely, $x$ is contained in the following closed convex subset of $\ell_1(I)$:
$$
\ell_1^1(I):=\left\{x=(x_i)_{i\in I}\in\ell^1(I)~\middle|~
x_i\geq 0\text{ for all }i,\text{ and }\|x\|=1
\right\}.
$$
We also define
$$
\ell_+^1(I):=\BR_{\geq 0}\cdot\ell_1^1(I)
=\left\{x=(x_i)_{i\in I}\in\ell^1(I)~\middle|~
x_i\geq 0\text{ for all }i
\right\}.
$$
Elements in $\ell^1(I)$ are viewed as (finite or infinite) column vectors.
If $i\in I$, let $e_i\in\ell_1^1(I)$ whose $i$-th component is $1$,
other components are $0$.

A \emph{transition matrix} $P$ is
a bounded linear operator on $\ell^1(I)$ preserving $\ell_1^1(I)$.
When elements of $\ell^1(I)$ are viewed as column vectors,
the $P$ is viewed as a (finite or infinite) matrix $P=(P_{ij})_{i,j\in I}$
such that $P_{ij}\geq 0$ for all $i,j$,
and $\sum_{i\in I}P_{ij}=1$ for all $j$.
The $P_{ij}$ is ``the transition probability from $j$ to $i$''.
An element $x\in\ell_1^1(I)$ is called an \emph{equilibrium distribution} of $P$
if $P(x)=x$.

\begin{defn}
A sequence $\{X_k:\Omega\to I\}_{k=0}^\infty$
of random variables is called a \emph{Markov chain},
if the following properties are satisfied:
\begin{itemize}
\item
(independent of the past)
For any $k\geq 0$ and $x_0,\cdots,x_{k+1}\in I$,
$$
\BP\big(X_{k+1}=x_{k+1}\mid X_i=x_i\text{ for all }0\leq i\leq k\big)
=\BP\big(X_{k+1}=x_{k+1}\mid X_k=x_k\big)
$$
if both conditional probabilities are well-defined.
\item
(time homogeneous)
There exists a transition matrix
$P=(P_{ij})$ such that
for any $k\geq 0$ and $i,j\in I$,
$$
\BP\big(X_{k+1}=i\mid X_k=j\big)=P_{ij}
$$
if the conditional probability in the left hand side is well-defined.
The $P=(P_{ij})$ is called the \emph{transition matrix} of $\{X_k\}$.
\end{itemize}
\end{defn}

In our application the $\Omega$ is usually
an inverse limit $\varprojlim_k\Omega_k$ of probability spaces,
each $\Omega_k$ is usually finite,
and each $X_k$ is factored through $\Omega_k$.

If random variables $\{X_k\}$ form a Markov chain,
then they satisfy $\BP(X_{k+1}=i)=\sum_{j\in I}P_{ij}\cdot\BP(X_k=j)$,
namely,
their probability mass functions $\{x_k\}$ satisfy $x_{k+1}=P(x_k)$
and hence $x_k=P^k(x_0)$.
Therefore, to understand the behavior of $X_k$ as $k\to\infty$,
an important aspect is to study $P^k(x)$ as $k\to\infty$
where $x\in\ell_1^1(I)$.

Inspired by the the later
results, we may introduce the following definition.

\begin{defn}
A sequence $\{X_k:\Omega\to I\}_{k=0}^\infty$
of random variables is called an \emph{almost Markov chain},
if the above ``independent of the past'' property is satisfied,
and the following property is satisfied:
\begin{itemize}
\item
(almost time homogeneous)
There exists a transition matrix
$P=(P_{ij})$ such that
$$
\sum_{k=0}^\infty\sum_{i\in I}\left|\BP(X_{k+1}=i)-\sum_{j\in I}P_{ij}\cdot\BP(X_k=j)\right|
$$
converges.
The $P=(P_{ij})$ is called the \emph{limit transition matrix} of $\{X_k\}$.
\end{itemize}
\end{defn}

Note that if $\{X_k\}$ is a Markov chain, then the above infinite sum is zero.

\subsubsection*{Ergodicity and geometric ergodicity}

The property of $P^k(x)$ as $k\to\infty$
is affected by the following important properties of $P$:
\begin{itemize}
\item
$P$ is called \emph{irreducible}, if for any $i,j\in I$,
there exists $k\geq 0$ such that $(P^k)_{ij}>0$.
\item
$P$ is called \emph{aperiodic}, if for any $i\in I$,
there exists $N$ such that $(P^k)_{ii}>0$ for all $k\geq N$.
\end{itemize}

If these two properties are satisfied, then we have the ergodicity result
of Markov chains (as well as almost Markov chains):

\begin{thm}
[see \cite{Nor97}, Theorem 1.8.3 and \cite{SD08}, \S5]
\label{p:MC}
Suppose $P$ is irreducible and aperiodic,
and has an equilibrium distribution $x_\infty\in\ell_1^1(I)$.

{\rm(i)}
For any $x\in\ell_1^1(I)$, we have
$\|P^n(x)-x_\infty\|\to 0$ as $n\to\infty$.
In particular, $P$ has exactly one equilibrium distribution.

{\rm(ii)}
Suppose $\{x_n\}_{n=0}^\infty$ is a sequence in $\ell_+^1(I)$
such that $\|x_n\|\to 1$ as $n\to\infty$ and
$\sum_{n=0}^\infty\|P(x_n)-x_{n+1}\|$ converges.
Then $\|x_n-x_\infty\|\to 0$ as $n\to\infty$.
\end{thm}

The above (i) is a well-known result in the theory of Markov chains.
The (ii) follows from (i) easily.

Under some more stronger conditions,
the convergence rate of Markov chains (as well as almost Markov chains)
to the equilibrium distribution
is geometric, which is called geometric ergodicity.
In other words, it is an effective version of Theorem \ref{p:MC}.

Define
$$
\ell_0(I):=\left\{x=(x_i)_{i\in I}\in\BR^I~\middle|~
x_i=0\text{ for all but finitely many }i
\right\}.
$$
Suppose $\nu:I\to\BR_{\geq 1}$ is a map, such that
for any $C>0$, $\#\{i\in I\mid\nu(i)\leq C\}<\infty$,
equivalently, the image of $\nu$ is discrete.
Then it induces a norm
$\|\ \|_\nu:\ell_0(I)\to\BR$,
$(x_i)_{i\in I}\mapsto\sum_{i\in I}|x_i|\nu(i)$,
in particular $\|x_i\|_\nu=\nu(i)$.
Clearly $\|x\|_\nu\geq\|x\|$ for all $x\in\ell_0(I)$.
Define
$$
\ell_\nu(I):=\left\{x=(x_i)_{i\in I}\in\BR^I~\middle|~
\|x\|_\nu:=\sum_{i\in I}|x_i|\nu(i)<\infty\right\},
$$
which is ``the completion of $\ell_0(I)$ with respect to $\|\ \|_\nu$''.
Clearly $\ell_0(I)\subset\ell_\nu(I)\subset\ell^1(I)$.

\begin{defn}
Let $P$ be a transition matrix on $I$.
\begin{itemize}
\item
$P$ is called \emph{finite}, if for any $i\in I$,
$\#\{j\in I\mid P_{ij}>0\}<\infty$
and $\#\{j\in I\mid P_{ji}>0\}<\infty$.
\item
$P$ is called \emph{extremely aperiodic}, if for any $i\in I$,
$P_{ii}>0$.
\item
$P$ is called \emph{$\nu$-driftable}, if there exists $\lambda<1$ such that $\|P(e_j)\|_\nu\leq\lambda\|e_j\|_\nu$
for all but finitely many $j\in I$.
\end{itemize}
\end{defn}

\begin{thm}[see \cite{KP20}, Theorem 2.1, Theorem 3.1, also \cite{MT09}]
\label{p:MC-2}
Suppose $P$ is finite, irreducible, extremely aperiodic,
and $\nu$-driftable.

{\rm(i)}
$P$ has a unique equilibrium distribution $x_\infty\in\ell_\nu(I)\cap\ell_1^1(I)$.
There exist $C>0$ and $0<\alpha<1$ depending only on $\nu$ and $P$, such that
for any $x\in\ell_\nu(I)\cap\ell_1^1(I)$ and any $n\geq 0$,
$$
\|P^n(x)-x_\infty\|
\leq\|P^n(x)-x_\infty\|_\nu\leq C\cdot\alpha^n\cdot\|x\|_\nu.
$$

{\rm(ii)}
Suppose $\gamma\geq 1$ such that
\begin{itemize}
\item
$\nu(i)\leq\gamma\cdot\nu(j)$ for all $i,j\in I$ such that $P_{ij}\neq 0$.
\end{itemize}
Suppose $0<\beta<1$ and $\epsilon>0$.
There exist $C_1>0$ and $0<\alpha_1<1$ depending only on $\beta,\gamma,\epsilon,\nu$
and $P$, such that if $C_0>0$ and
$\{x_n\}_{n=0}^\infty$ is a sequence in $\ell_\nu(I)\cap\ell_+^1(I)$
satisfying that
\begin{itemize}
\item
$\big|\|x_n\|-1\big|\leq C_0\cdot\beta^n$ for all $n\geq 0$,
\item
$\|P(x_n)-x_{n+1}\|\leq C_0\cdot\beta^n$ for all $n\geq 0$,
\item
$\|x_n\|_\nu\leq C_0$ and
$\|x_n^{\nu>C}\|_\nu\leq C_0\cdot C^{-\epsilon}$ for all $n\geq 0$ and $C\geq 1$,
here if $x_n=\sum_{i\in I}x_{n,i}e_i$,
define $x_n^{\nu>C}:=\sum_{i\in I,\nu(i)>C}x_{n,i}e_i$,
\end{itemize}
then for any $n\geq 0$,
$$
\|x_n-x_\infty\|_\nu\leq C_0\cdot C_1\cdot\alpha_1^n.
$$
\end{thm}

\begin{proof}
(i)
By \cite{MT09}, Theorem 15.0.1,
the $x_\infty$ exists,
and there exists $R>0$ and $r>1$ such that
for any $x\in\ell_\nu(I)\cap\ell_1^1(I)$
$$
\sum_{n=0}^\infty r^n\|P^n(x)-x_\infty\|_\nu\leq R\cdot\|x\|_\nu.
$$
Hence $r^n\|P^n(x)-x_\infty\|_\nu\leq R\cdot\|x\|_\nu$ for any $n\geq 0$,
which yields the desired result.

(ii)
Denote $y_n:=P(x_n)-x_{n+1}$.
First we claim that there exists $C_2>0$ and $0<\alpha_2<1$
depending only on $\beta,\gamma,\epsilon,\nu$
and $P$,
such that for any $n\geq 0$,
$$
\|y_n\|_\nu\leq C_0\cdot C_2\cdot\alpha_2^n.
$$
In fact, we may write
$$
y_n=P(x_n)-x_{n+1}=\Big(P(x_n)^{\nu\leq\gamma C}-x_{n+1}^{\nu\leq\gamma C}\Big)
+P(x_n)^{\nu>\gamma C}-x_{n+1}^{\nu>\gamma C}.
$$
We have
\begin{align*}
\Big\|P(x_n)^{\nu\leq\gamma C}-x_{n+1}^{\nu\leq\gamma C}\Big\|_\nu
&\leq\gamma C\cdot\Big\|P(x_n)^{\nu\leq\gamma C}-x_{n+1}^{\nu\leq\gamma C}\Big\|
\leq\gamma C\cdot\|P(x_n)-x_{n+1}\|
\leq\gamma C\cdot C_0\cdot\beta^n, \\
\|P(x_n)^{\nu>\gamma C}\|_\nu
&=\|P(x_n^{\nu>C})^{\nu>\gamma C}\|_\nu
\leq\|P(x_n^{\nu>C})\|_\nu
\leq\|P\|_\nu\cdot C_0\cdot C^{-\epsilon}, \\
\|x_{n+1}^{\nu>\gamma C}\|_\nu
&\leq C_0\cdot(\gamma C)^{-\epsilon}.
\end{align*}
In the second formula we use the fact that
$P(x_n^{\nu\leq C})^{\nu>\gamma C}=0$.
Therefore if take $C=\beta^{-n/2}$, then
$$
\|y_n\|_\nu\leq C_0\cdot\left(\gamma\cdot\beta^{n/2}+(\|P\|_\nu+\gamma^{-\epsilon})
\cdot\beta^{n\epsilon/2}\right),
$$
so we can take $C_2=\gamma+\|P\|_\nu+\gamma^{-\epsilon}$
and $\alpha_2=\max\{\beta^{1/2},\beta^{\epsilon/2}\}$.

Now let $C$ and $\alpha$ be constants in (i) depending only on $\nu$ and $P$,
fix an arbitrary integer $N\geq 1$,
for any $k\geq 0$ write $k=m+n$ with $m=\lfloor\frac{k}{N+1}\rfloor$,
so that $n\geq\frac{Nk}{N+1}$ and $m\geq\frac{k-N}{N+1}$.
To estimate $\|x_k-x_\infty\|_\nu$ we first note that
$$
P^n(x_m)-x_{m+n}
=\sum_{i=0}^{n-1}P^{n-1-i}(y_{m+i}).
$$
Denote $s_n:=\|y_n\|$.
If $s_{m+i}\neq 0$ then
$y_{m+i}/s_{m+i}\in\ell_\nu(I)\cap\ell_1^1(I)$, hence by (i) it is easy to see that
$$
\|P^{n-1-i}(y_{m+i})-s_{m+i}\cdot x_\infty\|_\nu
\leq C\cdot\alpha^{n-1-i}\cdot\|y_{m+i}\|_\nu
\leq C\cdot C_0\cdot C_2\cdot\alpha^{n-1-i}\cdot\alpha_2^{m+i}.
$$
This is also true for $s_{m+i}=0$. Therefore, enlarge $\alpha_2$ such that
$\alpha<\alpha_2<1$, then
$$
\left\|P^n(x_m)-x_{m+n}
-x_\infty\cdot\sum_{i=0}^{n-1}s_{m+i}\right\|_\nu
\leq C\cdot C_0\cdot C_2\cdot\sum_{i=0}^{n-1}
\alpha^{n-1-i}\cdot\alpha_2^{m+i}
\leq C\cdot C_0\cdot C_2\cdot\frac{\alpha_2^{m+n}}{\alpha_2-\alpha}.
$$
On the other hand, denote $r_n:=\|x_n\|$,
if $r_m\neq 0$ then $x_m/r_m\in\ell_\nu(I)\cap\ell_1^1(I)$,
hence by (i),
$$
\|P^n(x_m)-r_mx_\infty\|_\nu
\leq C\cdot\alpha^n\cdot\|x_m\|_\nu
\leq C\cdot C_0\cdot\alpha^n.
$$
This is also true for $r_m=0$.
Therefore
\begin{align*}
\|x_k-x_\infty\|_\nu
&\leq C\cdot C_0\cdot\alpha^n
+C\cdot C_0\cdot C_2\cdot\frac{\alpha_2^{m+n}}{\alpha_2-\alpha}
+\|x_\infty\|_\nu\cdot\left(|r_m-1|+\sum_{i=0}^{n-1}s_{m+i}\right) \\
&\leq C_0\cdot\left(C\cdot\alpha^{\frac{Nk}{N+1}}
+C\cdot C_2\cdot\frac{\alpha_2^k}{\alpha_2-\alpha}
+\|x_\infty\|_\nu\cdot\beta^{\frac{k-N}{N+1}}\cdot\frac{2-\beta}{1-\beta}
\right),
\end{align*}
so we can take $C_1=C+\frac{C\cdot C_2}{\alpha_2-\alpha}
+\|x_\infty\|_\nu\cdot\frac{2-\beta}{1-\beta}
\cdot\beta^{-\frac{N}{N+1}}$
and $\alpha_1=\max\{\alpha^{\frac{N}{N+1}},\alpha_2,\beta^{\frac{1}{N+1}}\}$.
\end{proof}

\subsection{Models of matrices and corresponding Markov chains}
\label{s:model}

To compute
$$
P_{t,\mathbf t}^\Alt(m):=\lim_{k\to\infty}
\BP\big(\corank(B)=m\mid
B\in M_{2k+t,\mathbf t}^\Alt(\BF_2)\big)
$$
in Theorem \ref{main theorem},
a key ingredient is the corresponding Markov chain.

We consider the following slightly generalized model $M_{dk+t,\mathbf t}^\Alt(\BF_2)$.
Let $d\geq 2$ be an even integer
(it is $d=2$ in Theorem \ref{main theorem}), $0\leq s\leq d$ be an integer,
$t$ be an integer (the model only depends
on the residue class of $t$ modulo $d$),
$\mathbf t=(t_1,\cdots,t_s)$ be an unordered tuple of $s$ integers, such that
$t_1\equiv\cdots\equiv t_s\equiv -t\pmod d$.
In the case that $s=d$ we further require that
$\sum_{i=1}^st_i\leq 0$.
Denote by $\Omega_k=M_{dk+t,\mathbf t}^\Alt(\BF_2)$ the set of all $B\in M_{dk+t}^\Alt(\BF_2)$ of form
$$
B=\begin{pmatrix}
0 & B_{12} & \cdots & B_{1,s+1} \\
B_{21} & \ddots & \ddots & \vdots \\
\vdots & \ddots & 0 & B_{s,s+1} \\
B_{s+1,1} & \cdots & B_{s+1,s} & B_{s+1,s+1}
\end{pmatrix}
$$
such that for $1\leq i\leq s$, $1\leq j\leq s$,
the $B_{ij}$ is of size $(k+\frac{t_i+t}d)\times(k+\frac{t_j+t}d)$.
Endow $\Omega_k$ with normalized counting measure.
Define the projection map $\Omega_{k+1}\twoheadrightarrow\Omega_k$,
$B_\new=(B_{\new,ij})\mapsto B=(B_{ij})$
where each $B_{ij}$ is the top-left corner of $B_{\new,ij}$.
This allows us to define the probability space $\Omega=\varprojlim_k\Omega_k$.
Consider the random variables on $\Omega$.
For example, for $X_k:\Omega_k\to\BZ_{\geq 0}$, $B\mapsto\corank(B)$,
then $P_{t,\mathbf t}^\Alt(m)=\lim_{k\to\infty}\BP(X_k=m)$.
When $s\neq 0$ the $\{X_k\}$ do not form a Markov chain.
To make them Markovian, the state space should be refined.

\begin{thm}
\label{model main thm}
For an element $B=(B_{ij})$ of $\Omega_k$,
denote $B_j':=\left(\begin{smallmatrix}
B_{1j} \\ \vdots \\ B_{s+1,j}
\end{smallmatrix}\right)$ for $1\leq j\leq s$.
Then random variables
$$
Y_k:\Omega_k\to I:=\BZ_{\geq 0}^{s+1},\qquad
B\mapsto\big(\corank(B),\corank(B_1'),\cdots,\corank(B_s')\big)
$$
form an almost Markov chain.
For $d=2$ or $s=0$, they actually form a Markov chain,
whose transition probabilities are listed in
\S\ref{s:ordinary},
\S\ref{s:1-trivial}
and \S\ref{s:2-trivial}.
\end{thm}

\begin{proof}
Consider $B_\new=(B_{\new,ij})\in\Omega_{k+1}$.
It is a matrix $B=(B_{ij})\in\Omega_k$ with newly added $d$ rows and columns.
To show $\{Y_k\}$ form a Markov chain,
we only need to show that for any elements $\mathbf m$
and $\mathbf m_\new$ of $I$,
for any fixed $k$ and fixed $B\in\Omega_k$
such that $Y_k(B)=\mathbf m$, the conditional probability
\begin{equation}
\label{e:model-markov-1}
\begin{array}{l}
\BP\big(Y_{k+1}(B_\new)=\mathbf m_\new\mid
B_\new\mapsto B\big)
\text{ only depends on $\mathbf m$, $\mathbf m_\new$,} \\
\text{and is independent of $B$ and $k$.}
\end{array}
\end{equation}
Here $B_\new\mapsto B$ means that
under the natural projection map $\Omega_{k+1}\twoheadrightarrow\Omega_k$,
the image of $B_\new$ is $B$.
This implies that
$\BP\big(Y_{k+1}=\mathbf m_\new\mid
Y_k=\mathbf m\big)$ is also equal to the above conditional probability.

We may divide the process $B\to B_\new$ into $d$ steps
$$
B=B^{(0)}\to B^{(1)}\to\cdots\to B^{(d)}=B_\new,
$$
each $B^{(\ell)}\in
M_{dk+t+\ell}^\Alt(\BF_2)$ is alternating,
and is $B^{(\ell-1)}$ with a new row and column added.
To describe this more precisely,
write $B^{(\ell)}=(B_{ij}^{(\ell)})$,
then $B^{(\ell)}\to B^{(\ell+1)}$
is adding a row to $B_{rj}^{(\ell)}$ for all $j$,
and adding a column to $B_{ir}^{(\ell)}$ for all $i$,
and leaving other submatrices unchanged,
here $r:=\min\{\ell+1,s+1\}$.
In other words, we have
$$
\Omega_k=\Omega_k^{(0)}
\twoheadleftarrow\Omega_k^{(1)}\twoheadleftarrow\cdots
\twoheadleftarrow\Omega_k^{(d)}=\Omega_{k+1},
$$
each $B^{(\ell)}$ is an element of $\Omega_k^{(\ell)}$,
and $\Omega_k^{(\ell+1)}=\Omega_k^{(\ell)}\times\Omega_{k,\Delta}^{(\ell)}$,
the space $\Omega_{k,\Delta}^{(\ell)}$ consists of all possible
added rows and columns in the step $B^{(\ell)}\to B^{(\ell+1)}$.
Moreover,
consider the natural decomposition
$\Omega_k=\prod_{1\leq i\leq j\leq s+1}\Omega_{k,ij}$
where $\Omega_{k,ii}:=0\subset
M_{k_i}^\Alt(\BF_2)$ for $1\leq i\leq s$,
$\Omega_{k,s+1,s+1}:=M_{k_{s+1}}^\Alt(\BF_2)$,
and $\Omega_{k,ij}:=M_{k_i\times k_j}(\BF_2)$ for $1\leq i<j\leq s+1$,
which makes $B_{ij}=B_{ji}^\RT\in\Omega_{k,ij}$
for $1\leq i\leq j\leq s+1$,
then we have
$$
\Omega_{k,ij}=\Omega_{k,ij}^{(0)}
\twoheadleftarrow\Omega_{k,ij}^{(1)}\twoheadleftarrow\cdots
\twoheadleftarrow\Omega_{k,ij}^{(d)}=\Omega_{k+1,ij},
$$
each $B_{ij}^{(\ell)}$ is an element of $\Omega_{k,ij}^{(\ell)}$,
and $\Omega_{k,ij}^{(\ell+1)}=\Omega_{k,ij}^{(\ell)}\times\Omega_{k,ij,\Delta}^{(\ell)}$,
the space $\Omega_{k,ij,\Delta}^{(\ell)}$ consists of all possible
added rows and columns in the step $B_{ij}^{(\ell)}\to B_{ij}^{(\ell+1)}$.
The precise definition of $\Omega_{k,ij}^{(\ell)}$
as well as the decomposition is as follows:
for $1\leq i\leq s$,
$$
\begin{cases}
\Omega_{k,ii}^{(0)}=\Omega_{k,ii}^{(1)}
=\cdots=\Omega_{k,ii}^{(i-1)}=0\subset M_{k_i}^\Alt(\BF_2), \\
\Omega_{k,ii}^{(i)}=\Omega_{k,ii}^{(i+1)}
=\cdots=\Omega_{k,ii}^{(d)}=0\subset M_{k_i+1}^\Alt(\BF_2),
\end{cases}
$$
for $i=s+1$,
$$
\begin{cases}
\Omega_{k,ii}^{(0)}=\Omega_{k,ii}^{(1)}
=\cdots=\Omega_{k,ii}^{(s)}=M_{k_i}^\Alt(\BF_2), \\
\Omega_{k,ii}^{(\ell+1)}=M_{k_i+\ell+1-s}^\Alt(\BF_2)
=\Omega_{k,ii}^{(\ell)}\times\BF_2^{k_i+\ell-s},&\text{if }s\leq\ell\leq d-1,
\end{cases}
$$
for $1\leq i<j\leq s$,
$$
\begin{cases}
\Omega_{k,ij}^{(0)}=\Omega_{k,ij}^{(1)}
=\cdots=\Omega_{k,ij}^{(i-1)}=M_{k_i\times k_j}(\BF_2), \\
\Omega_{k,ij}^{(i)}=\Omega_{k,ij}^{(i+1)}
=\cdots=\Omega_{k,ij}^{(j-1)}=M_{(k_i+1)\times k_j}(\BF_2)
=\Omega_{k,ij}^{(i-1)}\times\BF_2^{k_j}, \\
\Omega_{k,ij}^{(j)}=\Omega_{k,ij}^{(j+1)}
=\cdots=\Omega_{k,ij}^{(d)}=M_{(k_i+1)\times(k_j+1)}(\BF_2)
=\Omega_{k,ij}^{(j-1)}\times\BF_2^{k_i+1},
\end{cases}
$$
and for $1\leq i<j=s+1$,
$$
\begin{cases}
\Omega_{k,ij}^{(0)}=\Omega_{k,ij}^{(1)}
=\cdots=\Omega_{k,ij}^{(i-1)}=M_{k_i\times k_j}(\BF_2), \\
\Omega_{k,ij}^{(i)}=\Omega_{k,ij}^{(i+1)}
=\cdots=\Omega_{k,ij}^{(s)}=M_{(k_i+1)\times k_j}(\BF_2)
=\Omega_{k,ij}^{(i-1)}\times\BF_2^{k_j}, \\
\Omega_{k,ij}^{(\ell+1)}=M_{(k_i+1)\times(k_j+\ell+1-s)}(\BF_2)
=\Omega_{k,ij}^{(\ell)}\times\BF_2^{k_i+1},&\text{if }s\leq\ell\leq d-1.
\end{cases}
$$

Similar to $B_j'$, we need to consider
$B_j^{(\ell)}:=(B_{ij}^{(\ell)})_{1\leq i\leq s+1}$
for $1\leq j\leq s$.
Then $B_j^{(\ell)}\in M_{(dk+t+\ell)\times k_j}(\BF_2)$ for $0\leq\ell\leq j-1$,
and $B_j^{(\ell)}\in M_{(dk+t+\ell)\times(k_j+1)}(\BF_2)$ for $j\leq\ell\leq d$.
For simplicity of notation, write
$$
Y_k^{(\ell)}:\Omega_k^{(\ell)}\to I:=\BZ_{\geq 0}^{s+1},
\quad
B^{(\ell)}\mapsto\big(\corank(B^{(\ell)}),\corank(B_1^{(\ell)}),
\cdots,\corank(B_s^{(\ell)})\big).
$$
To show \eqref{e:model-markov-1}, we only need to show that
each step $B^{(\ell)}\to B^{(\ell+1)}$ is independent of previous steps,
and is independent of $k$. More precisely,
we only need to show that for any elements $\mathbf m^{(\ell)}$
and $\mathbf m^{(\ell+1)}$ of $I$,
for any fixed $k$ and fixed $B^{(\ell)}\in\Omega_k^{(\ell)}$
such that $Y_k^{(\ell)}(B^{(\ell)})=\mathbf m^{(\ell)}$, the conditional probability
\begin{equation}
\label{e:model-markov-2}
\begin{array}{l}
\BP\big(Y_k^{(\ell+1)}(B^{(\ell+1)})=\mathbf m^{(\ell+1)}\mid
B^{(\ell+1)}\mapsto B^{(\ell)}\big)
\text{ only depends on $\ell$, $\mathbf m^{(\ell)}$, $\mathbf m^{(\ell+1)}$,} \\
\text{and is independent of $B^{(\ell)}$ and $k$.}
\end{array}
\end{equation}
In the following the condition in \eqref{e:model-markov-2}
will be abbreviated as $\cdots$.
Write $\mathbf m^{(\ell)}=(m^{(\ell)},m_1^{(\ell)},\cdots,m_s^{(\ell)})$.
Then $\rank(B^{(\ell+1)})-\rank(B^{(\ell)})\in\{0,2\}$
and $\rank(B_j^{(\ell+1)})-\rank(B_j^{(\ell)})\in\{0,1\}$
(note that $B_{jj}^{(\ell)}=0$ for all $\ell$), so
$m^{(\ell+1)}-m^{(\ell)}\in\{\pm 1\}$,
$m_j^{(\ell+1)}-m_j^{(\ell)}\in\{0,-1\}$ for $\ell\neq j-1$,
and $m_j^{(\ell+1)}-m_j^{(\ell)}\in\{0,1\}$ for $\ell=j-1$.
This means that each $m^{(\ell)}\to m^{(\ell+1)}$
and $m_j^{(\ell)}\to m_j^{(\ell+1)}$ has $2$ possibilities,
and $\mathbf m^{(\ell)}\to\mathbf m^{(\ell+1)}$ has $2^{s+1}$
possibilities.
Therefore to compute \eqref{e:model-markov-2}
is equivalent to compute
\begin{equation}
\label{e:model-markov-P1}
P_S^{(\ell)}:=\BP\big(m_j^{(\ell+1)}=m_j^{(\ell)}+\delta_{j,\ell+1}
\text{ for all }j\in S\mid\cdots\big)
\end{equation}
and
\begin{equation}
\label{e:model-markov-P2}
P_{\{0\}\cup S}^{(\ell)}:=\BP\big(m^{(\ell+1)}=m^{(\ell)}+1
\text{ and }m_j^{(\ell+1)}=m_j^{(\ell)}+\delta_{j,\ell+1}
\text{ for all }j\in S\mid\cdots\big)
\end{equation}
for every subsets $S$ of $\{1,2,\cdots,s\}$.
Here $\delta_{j,\ell+1}$ is the Kronecker delta for $j$ and $\ell+1$.

Let $v\in\BF_2^{dk+t+\ell}$.
If $0\leq\ell\leq s-1$,
write $v=(v_1;\cdots;v_{s+1})$
with $v_i\in\BF_2^{k_i+1}$ for $1\leq i\leq\ell$
and $v_i\in\BF_2^{k_i}$ for $\ell+1\leq i\leq s+1$,
such that $v_{\ell+1}=0$ and $v_i$ varies arbitrarily for $i\neq\ell+1$.
If $s\leq\ell\leq d-1$, let $v$ varies arbitrarily.
Then the space $\Omega_{k,\Delta}^{(\ell)}$ is just all the possible $v$,
and $B^{(\ell+1)}$ is $B^{(\ell)}$ with newly added column $v$
and newly added row $v^\RT$.

Let's compute \eqref{e:model-markov-P1}.
We have
$$
P_S^{(\ell)}=\BP\big(m_j^{(\ell+1)}=m_j^{(\ell)}+\delta_{j,\ell+1}
\text{ for all }j\in S\mid\cdots\big)
=\BP\big(\rank(B_j^{(\ell+1)})=\rank(B_j^{(\ell)})
\text{ for all }j\in S\mid\cdots\big).
$$
If $\ell+1\notin S$, then
\begin{align*}
P_S^{(\ell)}
=\BP\big(m_j^{(\ell+1)}=m_j^{(\ell)}\text{ for all }j\in S\mid\cdots\big)
&=\BP\big(\rank(B_j^{(\ell+1)})=\rank(B_j^{(\ell)})\text{ for all }j\in S\mid\cdots\big) \\
&=\BP\big(v_j\in\operatorname{span}((B_j^{(\ell)})^\RT)
\text{ for all }j\in S\mid\cdots\big)
=2^{-\sum_{j\in S}m_j^{(\ell)}}.
\end{align*}
If $\ell+1\in S$, write $S=\{\ell+1\}\sqcup S'$, then
\begin{align*}
P_S^{(\ell)}&=\BP\big(m_{\ell+1}^{(\ell+1)}=m_{\ell+1}^{(\ell)}+1
\text{ and }m_j^{(\ell+1)}=m_j^{(\ell)}
\text{ for all }j\in S'\mid\cdots\big) \\
&=\BP\big(\rank(B_j^{(\ell+1)})=\rank(B_j^{(\ell)})
\text{ for all }j\in S\mid\cdots\big) \\
&=\BP\big(v\in\operatorname{span}(B_{\ell+1}^{(\ell)})
\text{ and }v_j\in\operatorname{span}((B_j^{(\ell)})^\RT)
\text{ for all }j\in S'\mid\cdots\big) \\
&\stackrel{(*)}=\BP\big(v\in\operatorname{span}(B_{\ell+1}^{(\ell)})\mid\cdots\big)
=P_{\{\ell+1\}}^{(\ell)}
=2^{2k_{\ell+1}-dk-t-\ell-m_{\ell+1}^{(\ell)}}
=2^{(2-d)k+2(t_{\ell+1}+t)/d-t-\ell-m_{\ell+1}^{(\ell)}}.
\end{align*}
Here $(*)$ holds since if
$v\in\operatorname{span}(B_{\ell+1}^{(\ell)})$,
then for any $j\in S'$ we have
$v_j\in\operatorname{span}(B_{j,\ell+1}^{(\ell)})
=\operatorname{span}((B_{\ell+1,j}^{(\ell)})^\RT)
\subset\operatorname{span}((B_j^{(\ell)})^\RT)$.
In particular, we know that the \eqref{e:model-markov-P1}
only depends on $\ell$, $m_j^{(\ell)}$ and $m_j^{(\ell+1)}$
for $j\in S$,
and is independent of $m^{(\ell)}$, $m^{(\ell+1)}$,
$m_j^{(\ell)}$ and $m_j^{(\ell+1)}$
for $j\notin S$, $B^{(\ell)}$ and $k$
(at least for $d=2$).

Let's compute \eqref{e:model-markov-P2}.
First consider $S=\varnothing$ case.
We have $m^{(\ell+1)}-m^{(\ell)}\in\{\pm 1\}$ and
$$
P_{\{0\}}^{(\ell)}=\BP\big(m^{(\ell+1)}=m^{(\ell)}+1\mid\cdots\big)
=\BP\big(\rank(B^{(\ell+1)})=\rank(B^{(\ell)})\mid\cdots\big)
=\BP\big(v\in\operatorname{span}(B^{(\ell)})\mid\cdots\big).
$$
Hence it's clear that if $s\leq\ell\leq d-1$, then
$$
P_{\{0\}}^{(\ell)}=\BP\big(m^{(\ell+1)}=m^{(\ell)}+1\mid\cdots\big)
=2^{-m^{(\ell)}}.
$$
If $0\leq\ell\leq s-1$, then
$$
P_{\{0\}}^{(\ell)}=\BP\big(m^{(\ell+1)}=m^{(\ell)}+1\mid\cdots\big)
=\frac{\#\big(\Omega_{k,\Delta}^{(\ell)}\cap\operatorname{span}(B^{(\ell)})\big)}
{\#\Omega_{k,\Delta}^{(\ell)}}
=\frac{\#\operatorname{span}(B^{(\ell)})}{\#\big(\Omega_{k,\Delta}^{(\ell)}+\operatorname{span}(B^{(\ell)})\big)},
$$
and
\begin{align*}
&\Omega_{k,\Delta}^{(\ell)}+\operatorname{span}(B^{(\ell)}) \\
&=\operatorname{span}\left(\begin{array}{ccccccc|ccccccc}
0 & B_{12}^{(\ell)} & \cdots & \cdots & \cdots & \cdots & B_{1,s+1}^{(\ell)} & I \\
B_{21}^{(\ell)} & \ddots & \ddots & & & & \vdots & & \ddots \\
\vdots & \ddots & \ddots & \ddots & & & \vdots & & & I \\
\vdots & & \ddots & \ddots & \ddots & & \vdots & & & & 0 \\
\vdots & & & \ddots & \ddots & \ddots & \vdots & & & & & I \\
\vdots & & & & \ddots & 0 & B_{s,s+1}^{(\ell)} & & & & & & \ddots \\
B_{s+1,1}^{(\ell)} & \cdots & \cdots & \cdots & \cdots & B_{s+1,s}^{(\ell)} & B_{s+1,s+1}^{(\ell)} & & & & & & & I
\end{array}\right) \\
&=\operatorname{span}\left(\begin{array}{ccccccc|ccccccc}
0 & \cdots & \cdots & \cdots & \cdots & \cdots & 0 & I \\
\vdots & & & & & & \vdots & & \ddots \\
0 & \cdots & \cdots & \cdots & \cdots & \cdots & 0 & & & I \\
B_{\ell+1,1}^{(\ell)} & \cdots & B_{\ell+1,\ell}^{(\ell)} & 0 & B_{\ell+1,\ell+2}^{(\ell)} & \cdots & B_{\ell+1,s+1}^{(\ell)} & & & & 0 \\
0 & \cdots & \cdots & \cdots & \cdots & \cdots & 0 & & & & & I \\
\vdots & & & & & & \vdots & & & & & & \ddots \\
0 & \cdots & \cdots & \cdots & \cdots & \cdots & 0 & & & & & & & I
\end{array}\right) \\
&\cong\operatorname{span}(B_{\ell+1}^{(\ell)})\oplus\Omega_{k,\Delta}^{(\ell)},
\end{align*}
hence
$$
P_{\{0\}}^{(\ell)}=\BP\big(m^{(\ell+1)}=m^{(\ell)}+1\mid\cdots\big)
=\frac{\#\operatorname{span}(B^{(\ell)})}
{\#\operatorname{span}(B_{\ell+1}^{(\ell)})\cdot\#\Omega_{k,\Delta}^{(\ell)}}
=2^{m_{\ell+1}^{(\ell)}-m^{(\ell)}}.
$$

Now let $S$ be any subset of $\{1,2,\cdots,s\}$
and we want to compute
\begin{align*}
P_{\{0\}\cup S}^{(\ell)}&=\BP\big(m^{(\ell+1)}=m^{(\ell)}+1
\text{ and }m_j^{(\ell+1)}=m_j^{(\ell)}+\delta_{j,\ell+1}
\text{ for all }j\in S\mid\cdots\big) \\
&=\BP\big(\rank(B^{(\ell+1)})=\rank(B^{(\ell)})
\text{ and }\rank(B_j^{(\ell+1)})=\rank(B_j^{(\ell)})
\text{ for all }j\in S\mid\cdots\big).
\end{align*}
Similar to the computation of \eqref{e:model-markov-P1}, if $\ell+1\in S$ then we have
$$
P_{\{0\}\cup S}^{(\ell)}
=\BP\big(v\in\operatorname{span}(B_{\ell+1}^{(\ell)})\mid\cdots\big)
=P_{\{\ell+1\}}^{(\ell)}
=2^{2k_{\ell+1}-dk-t-\ell-m_{\ell+1}^{(\ell)}}
=2^{(2-d)k+2(t_{\ell+1}+t)/d-t-\ell-m_{\ell+1}^{(\ell)}}.
$$
If $\ell+1\notin S$, then
\begin{align*}
P_{\{0\}\cup S}^{(\ell)}
&=\BP\big(v\in\operatorname{span}(B^{(\ell)})
\text{ and }v_j\in\operatorname{span}((B_j^{(\ell)})^\RT)
\text{ for all }j\in S\mid\cdots\big) \\
&\stackrel{(*)}=\BP\big(v\in\operatorname{span}(B^{(\ell)})\mid\cdots\big)
=P_{\{0\}}^{(\ell)}
=\begin{cases}
2^{-m^{(\ell)}},&\text{if }s\leq\ell\leq d-1, \\
2^{m_{\ell+1}^{(\ell)}-m^{(\ell)}},&\text{if }0\leq\ell\leq s-1.
\end{cases}
\end{align*}
Here $(*)$ holds since if
$v\in\operatorname{span}(B^{(\ell)})$,
then for any $j\in S$ we have
$v_j\in\operatorname{span}(B_{j1}^{(\ell)},\cdots,B_{j,s+1}^{(\ell)})
=\operatorname{span}((B_j^{(\ell)})^\RT)$.
\end{proof}

\begin{remark}
When $d\geq 4$ it turns out that
there exists $C>0$ such that
$\BP\big(\corank(B_j')=0\big)\geq 1-C\cdot 2^{-k}$
for all $1\leq j\leq s$ and all $k\geq 0$,
and not only $\{Y_k\}$, but also $\{X_k\}$ form an almost Markov chain.
\end{remark}

Later we only consider $d=2$ cases of Theorem \ref{model main thm}
which are models used in
Theorem \ref{main theorem}.
There are three models: type (A) model which is $s=0$,
type (B) model which is $s=1$,
and type (C) model which is $s=2$.
We will write down the transition probability and limit distribution
of them explicitly.

Denote $P_{t,\mathbf t}^\Alt(\mathbf m):=
\lim_{k\to\infty}\BP(Y_k=\mathbf m)$.
Once we know that $\{Y_k\}$ form a Markov chain,
when $t$ and $\mathbf t$ are fixed,
denote the transition probability $\BP\big(Y_{k+1}=\mathbf m_\new\mid
Y_k=\mathbf m\big)$ by $P_{\mathbf m\to\mathbf m_\new}^\Alt$
which is independent of $k$.
It's easy to see that such transition probality is $0$ unless
$\mathbf m_\new-\mathbf m\in\{0,\pm 2\}\times\prod_{j=1}^s\{0,\pm 1\}$,
hence when $s=0$, $1$ and $2$ there are $3$, $9$ and $27$
transition probabilities, respectively.

\subsubsection{Type (A)}
\label{s:ordinary}

The property of type (A) model was studied in, for example,
\cite{HB94}, \cite{SD08}, \cite{Kane13}, \cite{KMR13}, \cite{Smith}.
The transition probability
$P_{m\to m_\new}^\Alt$ is
listed in the following table,
here for the simplicity of notation, denote $x:=2^{-m}$.
\begin{center}
\begin{tabular}{|c|c|}
\hline
$m_\new$ & $P_{m\to m_\new}^\Alt$ \\
\hline
$m+2$ & $x^2/2$ \\
\hline
$m$ & $x(3-5x/2)$ \\
\hline
$m-2$ & $(1-x)(1-2x)$ \\
\hline
\end{tabular}
\end{center}
Note that we must have $m\equiv t\pmod 2$,
and when we only consider such $m$, it's easy to see that the
transition matrix $P_{m\to m_\new}^\Alt$
is irreducible and aperiodic,
hence by Theorem \ref{p:MC} we know that
the $P_t^\Alt$ must be the unique
equilibrium state of $P_{m\to m_\new}^\Alt$.

\begin{prop}
[{see \cite{HB94}, \cite{SD08}, \cite{Kane13}, \cite{KMR13}, \cite{Smith}}]
\label{p:prob-ordinary}
We have
$$
P_t^\Alt(m)=\begin{cases}
\displaystyle
\prod_{i=1}^m\frac{2}{2^i-1}\prod_{i=1}^\infty(1+2^{-i})^{-1},
&\text{if }m\equiv t\pmod 2, \\
0,&\text{if }m\not\equiv t\pmod 2.
\end{cases}
$$
\end{prop}

\begin{cor}
[{see \cite{Kane13}, \S5 and \cite{HB94}, Theorem 1}]
\label{p:moment-ordinary}
Define the formal power series
$$
F_0(X):=\frac{\prod_{j=0}^\infty(1+2^{-j}X)}
{\prod_{j=0}^\infty(1+2^{-j})}\in\BR[[X]]
$$
and $F(X):=F_0(X)+(-1)^tF_0(-X)\in\BR[[X]]$.
Then $F(X)$ is the generating function of $P_t^\Alt$:
$$
\sum_mP_t^\Alt(m)X^m=F(X).
$$
It converges on the whole complex plane.
In particular, for $\xi\in\BZ_{\geq 0}$,
the $\xi$-th moment of the size of $\ker(B)$ as $B\in\Omega_k$, $k\to\infty$ is
$$
\sum_mP_t^\Alt(m)2^{\xi\cdot m}
=F(2^\xi)=F_0(2^\xi)
=\prod_{j=1}^\xi(1+2^j)
=\prod_{j=1}^\xi\frac{2^{2j}-1}{2^j-1},
$$
which is independent of $t$,
and taking $\xi=1$, the average size of $\ker(B)$ as $B\in\Omega_k$, $k\to\infty$ is $3$.
\end{cor}

\subsubsection{Type (B)}
\label{s:1-trivial}

In this case
the transition probability
$P_{\mathbf m\to\mathbf m_\new}^\Alt$
is as follows.
It is zero unless listed in the table,
here for the simplicity of notation, denote
$\tau:=2^{t_1}$,
$x:=2^{m_1'-m}$, $y:=2^{-m_1'}$,
so that $xy=2^{-m}$.

\begin{center}
{\begin{tabular}{|c|c|c|c|c|}
\hline
\multirow{2}{*}{$m'_{1,\new}$} & \multicolumn{4}{c|}{$m_\new$} \\
\cline{2-5}
& $m-2$ & $m$ & $m+2$ & $*$ \\
\hline
$m_1'+1$ & $0$
& $\tau (1-x)y^2/2$ & $\tau xy^2/2$ & $\tau y^2/2$ \\
\hline
$m_1'$ & $y(1-x)(1-2x)$
& $y(\tau-3\tau y/2+3x-5x^2/2+\tau xy/2)$
& $xy(x-\tau y)/2$ & $y(1+\tau-3\tau y/2)$ \\
\hline
$m_1'-1$ & $(1-x)(1-y)$
& $(x-\tau y)(1-y)$ & $0$ & $(1-y)(1-\tau y)$ \\
\hline
$*$ & $(1-x)(1-2xy)$ & $x(2y+1-5xy/2)$ & $x^2y/2$ & $1$ \\
\hline
\end{tabular}}
\end{center}

Let $I$ be the set of $\mathbf m=(m,m_1')$ satisfying $0\leq m_1'\leq m\leq 2m_1'-t_1$
and $m_1'\geq\max\{t_1,0\}$
and $m\equiv t\pmod 2$.
It's easy to check that if $\mathbf m\in I$, then all the entries in the above table
are non-negative, moreover, the entry is non-zero
if and only if $\mathbf m_\new\in I$
and is not equal to $(m+2,m_1'-1)$ and $(m-2,m_1'+1)$.
Hence when we only consider such $\mathbf m$, it's easy to see that the obtaining
Markov chain is irreducible and aperiodic,
by Theorem \ref{p:MC}
we know that the $P_{t,\mathbf t}^\Alt$
must be the unique equilibrium distribution.

Note that the $P_{(m,m_1')\to(*,m_{1,\new}')}^\Alt$ only depends on $m_1'$,
and in fact it's clear that
$$
P_{t,\mathbf t}^\Alt(*,m_1')=P_{t_1}^\Mat(m_1')
:=\lim_{k\to\infty}\BP\left(\corank(B)=m_1'\mid B\in M_{(k-t_1)\times k}(\BF_2)\right).
$$
From $P_{t,\mathbf t}^\Alt(*,m_1')$ we construct
$(P_{t,\mathbf t}^\Alt(m,m_1'))_{\mathbf m=(m,m_1')\in I}$ as follows.
It's straightforward to check it is indeed an equilibrium distribution,
which is automatically unique.

\begin{prop}
\label{p:1-trivial,m,m'}
For $\ell\geq 0$ even and $m_1'\geq\max\{t_1+\ell,0\}$, the following formula
$$
P_{t,\mathbf t}^\Alt(2m_1'-t_1-\ell,m_1')=P_{t_1}^\Mat(m_1')\cdot
2^{-(m_1'-t_1)(m_1'-t_1-1)/2}
\prod_{i=0}^{\ell-1}(2^{m_1'-t_1-i}-1)
\prod_{i=1}^{\ell/2}\frac{4^{i-1}}{4^i-1}
$$
gives an equilibrium distribution.
\end{prop}

For $m\geq\max\{t_1,0\}$ such that $m\equiv t\pmod 2$,
the $P_{t,\mathbf t}^\Alt(m)$ is the sum of
$P_{t,\mathbf t}^\Alt(m,m_1')$ as $m_1'$ varies:
$$
P_{t,\mathbf t}^\Alt(m)=\sum_{m_1'=\max\{(m+t_1)/2,0\}}^m
P_{t,\mathbf t}^\Alt(m,m_1').
$$
For other $m$, we have $P_{t,\mathbf t}^\Alt(m)=0$, in fact,
in this case there is no $B\in\Omega_k$ such that $\corank(B)=m$.
As a special case, if $t_1\leq 0$ is even, then
the probability of $\corank(B)$ being minimal is
$$
P_{t,\mathbf t}^\Alt(0)
=P_{t,\mathbf t}^\Alt(0,0)
=\prod_{i=1}^{-t_1/2}(1-4^{-i})^{-1}
\prod_{i=1}^\infty
(1-2^{-i}).
$$
Clearly it is different from the type (A) in Proposition \ref{p:prob-ordinary},
for example, the average size of $\ker(B)$ for these two types are different.

\begin{cor}
\label{p:moment-type-B}
Define the formal power series
$$
F_0(X):=\frac{\prod_{j=0}^\infty(1+2^{-j}X)}
{\prod_{j=0}^\infty(1+2^{-j})}
\sum_{i=0}^\infty 2^{it_1}\prod_{j=1}^i\frac{2^{1-j}X-1}{2^{2j}-1}\in\BR[[X]]
$$
and $F(X):=F_0(X)+(-1)^tF_0(-X)\in\BR[[X]]$,
Then $F(X)$ is the generating function of $P_{t,\mathbf t}^\Alt$:
$$
\sum_mP_{t,\mathbf t}^\Alt(m)X^m=F(X).
$$
It converges on the whole complex plane.
In particular, for $\xi\in\BZ_{\geq 0}$,
the $\xi$-th moment of the size of $\ker(B)$ as $B\in\Omega_k$, $k\to\infty$ is
$$
\sum_mP_{\infty;t,\mathbf t}^\Alt(m)2^{\xi\cdot m}
=F(2^\xi)=F_0(2^\xi)
=\sum_{i=0}^\xi 2^{it_1}\prod_{j=1}^{\xi-i}\frac{2^{2(i+j)}-1}{2^j-1}.
$$
Taking $\xi=1$, the average size of $\ker(B)$ as $B\in\Omega_k$, $k\to\infty$ is $3+2^{t_1}$.
\end{cor}

But this result suggests that type (A) is the limit of type (B)
as $t_1\to-\infty$. In fact we have the following result.

\begin{cor}
For any fixed $m\geq 0$, let $t_1:=t-\rho$ varies, then we have the limit
$$
\lim_{\substack{t_1\to-\infty\\
t_1\equiv m\pmod 2}}P_{t,(t_1)}^\Alt(m)
=\prod_{i=1}^m\frac{2}{2^i-1}\prod_{i=1}^\infty(1+2^{-i})^{-1},
$$
which equals the type (A) $P_t^\Alt(m)$
in Proposition \ref{p:prob-ordinary}.
\end{cor}

\subsubsection{Type (C)}
\label{s:2-trivial}

Denote $\tau_i:=2^{t_i}$,
$y_i:=2^{-m_i'}$.
Then for $P_{(*,m_1',m_2')\to(*,m_{1,\new}',m_{2,\new}')}^\Alt$
we have the following table.

\begin{center}
\resizebox{0.99\textwidth}{!}
{\begin{tabular}{|c|c|c|c|c|}
\hline
\multirow{2}{*}{$m'_{2,\new}$} & \multicolumn{4}{c|}{$m'_{1,\new}$} \\
\cline{2-5}
& $m_1'-1$ & $m_1'$ & $m_1'+1$ & $*$ \\
\hline
$m_2'+1$ & $0$
& $\tau_2y_2(y_2-\tau_1y_1)/2$ & $\tau_1\tau_2y_1y_2/2$ & $\tau_2y_2^2/2$ \\
\hline
$m_2'$ & $(1-y_1)(y_2-\tau_1y_1)$
& $\begin{array}{c}
\tau_1y_1+\tau_2y_2+y_1y_2+\tau_1\tau_2y_1y_2/2 \\
{}-3(\tau_1y_1^2+\tau_2y_2^2)/2
\end{array}$
& $\tau_1y_1(y_1-\tau_2y_2)/2$ & $y_2(1+\tau_2-3\tau_2y_2/2)$ \\
\hline
$m_2'-1$ & $(1-y_1)(1-y_2)$
& $(1-y_2)(y_1-\tau_2y_2)$ & $0$ & $(1-y_2)(1-\tau_2y_2)$ \\
\hline
$*$ & $(1-y_1)(1-\tau_1y_1)$ & $y_1(1+\tau_1-3\tau_1y_1/2)$ & $\tau_1y_1^2/2$ & $1$ \\
\hline
\end{tabular}}
\end{center}
Since $m_1'-m_2'\geq t_1$ and $m_2'-m_1'\geq t_2$,
we have $y_2-\tau_1y_1\geq 0$ and $y_1-\tau_2y_2\geq 0$.
It's easy to see that all the entries in the above table
are non-negative.

Again, we have
\begin{align*}
P_{t,\mathbf t}^\Alt(*,m_1',*)
&=P_{t_1}^\Mat(m_1')
=2^{-m_1'(m_1'-t_1)}
\prod_{i=1}^{m_1'}
(1-2^{-i})^{-1}
\prod_{i=1}^{m_1'-t_1}
(1-2^{-i})^{-1}
\prod_{i=1}^\infty
(1-2^{-i}), \\
P_{t,\mathbf t}^\Alt(*,*,m_2')
&=P_{t_2}^\Mat(m_2')
=2^{-m_2'(m_2'-t_2)}
\prod_{i=1}^{m_2'}
(1-2^{-i})^{-1}
\prod_{i=1}^{m_2'-t_2}
(1-2^{-i})^{-1}
\prod_{i=1}^\infty
(1-2^{-i}).
\end{align*}
For $P_{t,\mathbf t}^\Alt(*,m_1',m_2')$, it's easy to check
that $P_{(*,m_1',m_2')\to(*,m_{1,\new}',m_{2,\new}')}^\Alt$
in the above table
is irreducible and aperiodic, and we have the following result.

\begin{prop}
\label{p:2-trivial,*,m1',m2'}
For $m_1',m_2'\in\BZ_{\geq 0}$ such that
$m_1'-m_2'\geq t_1$ and $m_2'-m_1'\geq t_2$,
the following formula
$$
P_{t,\mathbf t}^\Alt(*,m_1',m_2')=
\frac{\displaystyle
2^{m_1'm_2'-m_1'(m_1'-t_1)-m_2'(m_2'-t_2)}
\prod_{i=1}^{-t_1-t_2}(1-2^{-i})
}{\displaystyle
\prod_{i=1}^{m_1'}
(1-2^{-i})
\prod_{i=1}^{m_1'-m_2'-t_1}
(1-2^{-i})
\prod_{i=1}^{m_2'}
(1-2^{-i})
\prod_{i=1}^{m_2'-m_1'-t_2}
(1-2^{-i})
}
\prod_{i=1}^\infty
(1-2^{-i})
$$
gives an equilibrium distribution.
\end{prop}

Denote $z:=2^{-m}$, then
for $P_{\mathbf m\to\mathbf m_\new}^\Alt$
we have the following tables.

\begin{center}
$m_\new=m+2$\nopagebreak\\
\begin{tabular}{|c|c|c|c|c|}
\hline
\multirow{2}{*}{$m'_{2,\new}$} & \multicolumn{4}{c|}{$m'_{1,\new}$} \\
\cline{2-5}
& $m_1'-1$ & $m_1'$ & $m_1'+1$ & $*$ \\
\hline
$m_2'+1$ & $0$
& $\tau_2y_2(z-\tau_1y_1^2)/2y_1$ & $\tau_1\tau_2y_1y_2/2$ & $\tau_2y_2z/2y_1$ \\
\hline
$m_2'$ & $0$
& $(z-\tau_1y_1^2)(z-\tau_2y_2^2)/2y_1y_2$
& $\tau_1y_1(z-\tau_2y_2^2)/2y_2$ & $z(z-\tau_2y_2^2)/2y_1y_2$ \\
\hline
$m_2'-1$ & $0$
& $0$ & $0$ & $0$ \\
\hline
$*$ & $0$ & $z(z-\tau_1y_1^2)/2y_1y_2$ & $\tau_1y_1z/2y_2$ & $z^2/2y_1y_2$ \\
\hline
\end{tabular}

\medskip

$m_\new=m-2$\nopagebreak\\
{\begin{tabular}{|c|c|c|c|c|}
\hline
\multirow{2}{*}{$m'_{2,\new}$} & \multicolumn{4}{c|}{$m'_{1,\new}$} \\
\cline{2-5}
& $m_1'-1$ & $m_1'$ & $m_1'+1$ & $*$ \\
\hline
$m_2'+1$ & $0$
& $0$ & $0$ & $0$ \\
\hline
$m_2'$ & $(1-y_1)(y_2-z/y_1)$
& $(y_1y_2-z)(y_1y_2-2z)/y_1y_2$
& $0$ & $(y_2-z/y_1)(1-2z/y_2)$ \\
\hline
$m_2'-1$ & $(1-y_1)(1-y_2)$
& $(1-y_2)(y_1-z/y_2)$ & $0$ & $(1-y_2)(1-z/y_2)$ \\
\hline
$*$ & $(1-y_1)(1-z/y_1)$ & $(y_1-z/y_2)(1-2z/y_1)$ & $0$ & $1-z/y_1-z/y_2-z+2z^2/y_1y_2$ \\
\hline
\end{tabular}}

\medskip

$m_\new=m$\nopagebreak\\
\resizebox{0.99\textwidth}{!}
{\begin{tabular}{|c|c|c|c|c|}
\hline
\multirow{2}{*}{$m'_{2,\new}$} & \multicolumn{4}{c|}{$m'_{1,\new}$} \\
\cline{2-5}
& $m_1'-1$ & $m_1'$ & $m_1'+1$ & $*$ \\
\hline
$m_2'+1$ & $0$
& $\tau_2y_2(y_1y_2-z)/2y_1$ & $0$ & $\tau_2y_2(y_1y_2-z)/2y_1$ \\
\hline
$m_2'$ & $(1-y_1)(z-\tau_1y_1^2)/y_1$
& $\begin{array}{c}
\tau_1y_1+\tau_2y_2
-3(\tau_1y_1^2+\tau_2y_2^2)/2 \\
{}+z(\tau_1y_1^2+\tau_2y_2^2-5z)/2y_1y_2+3z
\end{array}$
& $\tau_1y_1(y_1y_2-z)/2y_2$ & $\begin{array}{c}
\tau_2y_2-3\tau_2y_2^2/2
-5z^2/2y_1y_2 \\
{}+\tau_2y_2z/2y_1+z/y_1+2z
\end{array}$ \\
\hline
$m_2'-1$ & $0$
& $(1-y_2)(z-\tau_2y_2^2)/y_2$ & $0$ & $(1-y_2)(z-\tau_2y_2^2)/y_2$ \\
\hline
$*$ & $(1-y_1)(z-\tau_1y_1^2)/y_1$ &
$\begin{array}{c}
\tau_1y_1-3\tau_1y_1^2/2
-5z^2/2y_1y_2 \\
{}+\tau_1y_1z/2y_2+z/y_2+2z
\end{array}$ & $\tau_1y_1(y_1y_2-z)/2y_2$ & $z/y_1+z/y_2+z-5z^2/2y_1y_2$ \\
\hline
\end{tabular}}
\end{center}
From the proof of Theorem \ref{model main thm}
(that is, divide $B\to B_\new$ as two steps)
we know that all the entries in the above table are non-negative.
Now it's easy to check
that $P_{\mathbf m\to\mathbf m_\new}^\Alt$
is irreducible and aperiodic, and we have the following result.

\begin{prop}
The
$$
P_{t,\mathbf t}^\Alt(m,m_1',m_2')=P_{t,\mathbf t}^\Alt(*,m_1',m_2')
\cdot\lambda_{T,A,B}
$$
gives an equilibrium distribution,
where $T:=-(t_1+t_2)/2\in\BZ$, $A:=m_2'-m_1'-t_2\in\BZ$,
$B:=m-m_1'-m_2'\in\BZ$,
which satisfy $T\geq 0$, $0\leq A\leq 2T$, $0\leq B\leq\min\{A,2T-A\}$,
$B\equiv A\pmod 2$,
and
$$
\lambda_{T,A,B}=
2^{(A-B)(2T-A-B)/2}
\frac{\prod_{i=1}^{A-B}(2^{B+i}-1)}
{\prod_{i=1}^{(A-B)/2}(2^{2i}-1)}
\cdot\frac{\prod_{i=(A-B)/2}^{\lfloor A/2\rfloor-1}(2^{2T+2-2A+2i}-1)}
{\prod_{i=0}^{\lfloor A/2\rfloor-1}(2^{2T-1-2i}-1)},
$$
which satisfies $\lambda_{T,A,B}=\lambda_{T,2T-A,B}$.

In particular, if $t_1\leq 0$, $t_2\leq 0$ are both even, then
$$
P_{t,\mathbf t}^\Alt(0)
=P_{t,\mathbf t}^\Alt(0,0,0)
=\prod_{i=1}^{-(t_1+t_2)/2}(1-4^{-i})
\prod_{i=1}^{-t_1/2}(1-4^{-i})^{-1}\prod_{i=1}^{-t_2/2}(1-4^{-i})^{-1}
\prod_{i=1}^\infty
(1-2^{-i}).
$$
\end{prop}

\begin{cor}
For $m\geq\max\{t_1,t_2,0\}$ such that $m\equiv t\pmod 2$,
the $P_{t,\mathbf t}^\Alt(m)$ is positive, equal to the sum of
$P_{t,\mathbf t}^\Alt(m,m_1',m_2')$ over all possible $m_1'$ and $m_2'$.
For other $m$, we have $P_{t,\mathbf t}^\Alt(m)=0$, in fact,
in this case there is no $\omega\in\Omega_k$ such that $\corank(B(\omega))=m$.
\end{cor}

\begin{cor}
\label{p:moment-type-C}
If $t_1+t_2\leq 0$, define the formal power series
$$
F_0(X):=\frac{\prod_{j=0}^\infty(1+2^{-j}X)}
{\prod_{j=0}^\infty(1+2^{-j})}
\sum_{i_1,i_2\geq 0}2^{i_1t_1+i_2t_2}
\cdot\frac{2^{2i_1i_2}\prod_{j=1}^{i_1+i_2}(2^{1-j}X-1)}
{\prod_{j=1}^{i_1}(2^{2j}-1)\prod_{j=1}^{i_2}(2^{2j}-1)}\in\BR[[X]]
$$
and $F(X):=F_0(X)+(-1)^tF_0(-X)\in\BR[[X]]$.
Then $F(X)$ is the generating function of $P_{t,\mathbf t}^\Alt$:
$$
\sum_mP_{t,\mathbf t}^\Alt(m)X^m=F(X).
$$
It converges on the whole complex plane.
In particular, for $\xi\in\BZ_{\geq 0}$,
the $\xi$-th moment of the size of $\ker(B)$ as $B\in\Omega_k$, $k\to\infty$ is
$$
\sum_mP_{t,\mathbf t}^\Alt(m)2^{\xi\cdot m}
=F(2^\xi)=F_0(2^\xi)
=\sum_{\substack{i_1,i_2\geq 0\\
i_1+i_2\leq\xi}}2^{i_1t_1+i_2t_2}
\cdot\frac{\displaystyle 2^{2i_1i_2}\prod_{j=1}^\xi(2^{2j}-1)}
{\displaystyle\prod_{j=1}^{i_1}(2^{2j}-1)\prod_{j=1}^{i_2}(2^{2j}-1)
\prod_{j=1}^{\xi-i_1-i_2}(2^j-1)}.
$$
Taking $\xi=1$, the average size of $\ker(B)$ as $B\in\Omega_k$, $k\to\infty$ is $3+2^{t_1}+2^{t_2}$.
\end{cor}

\subsection{Average order of kernel of matrix models}
\label{average order matrix model}

We apply Heath-Brown's method \cite{HB93}, \cite{HB94} (see also
Proposition \ref{p:high-rank-moment-limit})
to the matrix model $M_{2k+t,\mathbf t}^\Alt(\BF_2)$,
to compute the $\omega$-average size of the kernel.

\subsubsection{Type (A) model}

For $B=(b_{ij})\in M_{2k+t}^\Alt(\BF_2)$,
let $r:=2k+t$
and $V=W=\BF_2^r$.
For $v\in V$ and $w\in W$, define $\iota:[r]\to\BF_2^2$,
$i\mapsto(v_i;w_i)$.
For $\lambda\in\BF_2^2$, let $[r]_\lambda$ be the preimage of $\lambda$ under $\iota$.
Then $w^\RT Bv=\sum_{(i,j)\in S}b_{ij}$
where $S=([r]_{01}\times[r]_{10})\sqcup([r]_{01}\times[r]_{11})
\sqcup([r]_{10}\times[r]_{11})$,
hence
$$
\sum_{B\in M_{2k+t}^\Alt(\BF_2)}(-1)^{w^\RT Bv}
=\begin{cases}
\# M_{2k+t}^\Alt(\BF_2),&\text{if }S=\varnothing, \\
0,&\text{otherwise}.
\end{cases}
$$
Note that $S$ being empty or not depends only on $\Lambda:=\Im(\iota)$.
When $\Lambda$ is fixed, the number of $\iota$
such that $\Im(\iota)=\Lambda$ is equal to
$\#\{[r]\twoheadrightarrow[\#\Lambda]\}=(\#\Lambda)^r+O((\#\Lambda-1)^r)$,
hence only the $\Lambda$ with $\#\Lambda\geq 2$
and $S=\varnothing$
have non-zero contributions
when $k\to\infty$.
Such $\Lambda$ are $\{00,01\}$, $\{00,10\}$ and $\{00,11\}$.
Therefore
\begin{align*}
\BE\big(2^{\corank(B)}\mid B\in M_{2k+t}^\Alt(\BF_2)\big)
&=\frac{1}{\#M_{2k+t}^\Alt(\BF_2)\cdot\#W}
\sum_{v\in V,w\in W}\sum_{B\in M_{2k+t}^\Alt(\BF_2)}(-1)^{w^\RT Bv} \\
&=\frac{1}{2^r}
\sum_{\substack{\Lambda\text{ such that}\\
\#\Lambda\geq 2\text{ and }S=\varnothing}}
\#\{[r]\twoheadrightarrow[\#\Lambda]\}+O((1/2)^r)
=3+O((1/4)^k).
\end{align*}

\subsubsection{Type (B) model}

Similar to type (A),
the map $\iota:[r]\to\BF_2^2$ restricts to
$\iota_1:[k_1]\to\BF_2^2$ and $\iota_2:[k_1+1\lddot r]\to\BF_2^2$.
For $\ell\in\{1,2\}$ and $\lambda\in\BF_2^2$ let $[r]_\lambda^{(\ell)}$
be the preimage of $\lambda$ under $\iota_\ell$, then
$w^\RT Bv=\sum_{(i,j)\in S}b_{ij}$
where
$$
S=\left(\bigsqcup_{\lambda,\lambda'\in\BF_2^2\setminus\{0\},\lambda\neq\lambda'}
[r]_\lambda^{(1)}\times[r]_{\lambda'}^{(2)}
\right)
\sqcup([r]_{01}^{(2)}\times[r]_{10}^{(2)})
\sqcup([r]_{01}^{(2)}\times[r]_{11}^{(2)})
\sqcup([r]_{10}^{(2)}\times[r]_{11}^{(2)}).
$$
The $S$ being empty or not depends only on $\Lambda_1:=\Im(\iota_1)$
and $\Lambda_2:=\Im(\iota_2)$,
and only the $\Lambda_1,\Lambda_2$ with $\#\Lambda_1\cdot\#\Lambda_2\geq 4$
have non-zero contributions when $k\to\infty$.
Such $\Lambda_1,\Lambda_2$ are
$\Lambda_1=\Lambda_2=\{00,01\}$, $\{00,10\}$ and $\{00,11\}$;
and $\Lambda_1=\BF_2^2$, $\Lambda_2=\{00\}$.
Therefore
\begin{align*}
\BE\big(2^{\corank(B)}\mid B\in M_{2k+t,(t_1)}^\Alt(\BF_2)\big)
&=\frac{1}{\#M_{2k+t,(t_1)}^\Alt(\BF_2)\cdot\#W}
\sum_{v\in V,w\in W}\sum_{B\in M_{2k+t,(t_1)}^\Alt(\BF_2)}(-1)^{w^\RT Bv} \\
&=\frac{3\cdot 2^{k_1}\cdot 2^{k_2}+4^{k_1}}{2^r}
+O((3/4)^k)
=3+2^{t_1}+O((3/4)^k).
\end{align*}

\subsubsection{Type (C) model}

Similar to type (A),
the map $\iota:[r]\to\BF_2^2$ restricts to
$\iota_1:[k_1]\to\BF_2^2$,
$\iota_2:[k_1+1\lddot k_1+k_2]\to\BF_2^2$,
and $\iota_3:[k_1+k_2+1\lddot r]\to\BF_2^2$.
For $\ell\in\{1,2,3\}$ and $\lambda\in\BF_2^2$ let $[r]_\lambda^{(\ell)}$
be the preimage of $\lambda$ under $\iota_\ell$, then
$w^\RT Bv=\sum_{(i,j)\in S}b_{ij}$
where
$$
S=\left(\bigsqcup_{\substack{
1\leq\ell<\ell'\leq 3 \\
\lambda,\lambda'\in\BF_2^2\setminus\{0\},\lambda\neq\lambda'}}
[r]_\lambda^{(\ell)}\times[r]_{\lambda'}^{(\ell')}
\right)
\sqcup([r]_{01}^{(3)}\times[r]_{10}^{(3)})
\sqcup([r]_{01}^{(3)}\times[r]_{11}^{(3)})
\sqcup([r]_{10}^{(3)}\times[r]_{11}^{(3)}).
$$
The $S$ being empty or not depends only on $\Lambda_\ell:=\Im(\iota_\ell)$,
$\ell\in\{1,2,3\}$,
and only the $(\Lambda_\ell)$ with $\#\Lambda_1\cdot\#\Lambda_2\geq 4$
have non-zero contributions when $k\to\infty$.
Such $(\Lambda_\ell)$ are
$\Lambda_3\subset\Lambda_1=\Lambda_2=\{00,01\}$, $\{00,10\}$ and $\{00,11\}$;
$\Lambda_1=\BF_2^2$, $\Lambda_3\subset\Lambda_2=\{00\}$;
and $\Lambda_2=\BF_2^2$, $\Lambda_3\subset\Lambda_1=\{00\}$.
Therefore
$$
\BE\big(2^{\corank(B)}\mid B\in M_{2k+t,(t_1,t_2)}^\Alt(\BF_2)\big)
=\frac{3\cdot 2^{k_1}\cdot 2^{k_2}\cdot 2^{k_3}+4^{k_1}+4^{k_2}}{2^r}
+O((3/4)^k)
=3+2^{t_1}+2^{t_2}+O((3/4)^k).
$$

\end{document}